\documentclass[11pt]{article}
\usepackage{array,epsfig}

\usepackage{IEEEtrantools}

\usepackage{caption}
\usepackage{amsmath}
\usepackage{amsfonts}
\usepackage{amssymb}
\usepackage{amsxtra}
\usepackage{amsthm}
\usepackage{mathrsfs}
\usepackage{color}
\usepackage[dvipsnames]{xcolor}
\usepackage{tikz}
\usepackage{mathabx,epsfig}
\usepackage{comment}
\usepackage{enumerate}
\usepackage{mathtools}
\usepackage{ytableau}
\usepackage{hyperref}
\usepackage{stmaryrd}
\usepackage{extarrows}

\hypersetup{
    colorlinks=true,
    linkcolor=blue,
    filecolor=magenta,      
    urlcolor=cyan,
    citecolor=teal
}
\theoremstyle{theorem}

\newtheorem{ex}{Example}

\newtheorem{theorem}{Theorem}[section]
\newtheorem{corollary}[theorem]{Corollary}
\newtheorem{lemma}[theorem]{Lemma}
\newtheorem{prop}[theorem]{Proposition}
\newtheorem{conjecture}[theorem]{Conjecture}

\newtheorem{question}[theorem]{Question}

\theoremstyle{definition}
\newtheorem{definition}[theorem]{Definition}

\newtheorem{remark}[theorem]{Remark}
\newtheorem{notation}[theorem]{Notation}

\newcommand{\lin}[1]{\langle #1 \rangle}

\newcommand{\surj}{\twoheadrightarrow}

\newcommand{\bb}[1]{\mathbb{#1}}
\newcommand{\mc}[1]{\mathcal{#1}}
\newcommand{\Z}{\bb{Z}}
\newcommand{\Q}{\bb{Q}}
\newcommand{\F}{\bb{F}}
\newcommand{\X}{\bb{X}}
\newcommand{\Y}{\bb{Y}}
\newcommand{\T}{\bb{T}}
\newcommand{\cT}{\mathcal{T}}

\newcommand{\del}{\partial}

\newcommand{\flag}{\textcolor{red}}

\newcommand{\sgn}{\mathrm{sgn}}

\newcommand{\Link}{\mathrm{Link}}

\newcommand{\w}{\omega}
\newcommand{\io}{\iota}
\newcommand{\coker}{\mathrm{coker}}
\newcommand{\im}{\mathrm{im}}
\newcommand{\het}{\mathrm{ht}}
\newcommand{\Ab}{\mathrm{Ab}}
\newcommand{\St}{\mathrm{St}}
\newcommand{\Gr}{\mathrm{Gr}}

\newcommand{\llb}{\llbracket}
\newcommand{\rrb}{\rrbracket}
\newcommand{\Sp}{\mathrm{Sp}}
\newcommand{\B}{\mathrm{B}}
\newcommand{\U}{\mathrm{U}}
\newcommand{\BA}{\mathrm{BA}}
\newcommand{\BD}{\mathrm{BD}}
\newcommand{\BDA}{\mathrm{BDA}}
\newcommand{\BAO}{\mathrm{BAO}}
\newcommand{\E}{\mathrm{E}}
\newcommand{\h}{\mathrm{H}}

\newcommand{\C}{\text{C}}
\newcommand{\R}{\text{R}}
\newcommand{\K}{\text{K}}

\newcommand{\SL}{\mathrm{SL}}

\title{The top cohomology of principal congruence subgroups of special linear groups over Euclidean number rings}
\author{Urshita Pal\footnote{urshita@umich.edu; University of Michigan, Department of Mathematics, Ann Arbor MI 48109, USA}}

\begin{document}

\maketitle

\begin{abstract}
    For $R$ a Euclidean number ring, let $\Gamma_n(p)$ be the level-$p$ principal congruence subgroup of $\SL_n(R)$. Borel--Serre showed that the cohomology of $\Gamma_n(p)$ vanishes above a degree $\nu$ that is quadratic in $n$. Let $K$ be the fraction field of $R$, and $\mc{T}_n(K)$ be the Tits building of $\SL_n(K)$. For $R=\Z$, Lee--Szczarba asked when $\h^\nu(\Gamma_n(p))$ is isomorphic to $\widetilde{\h}_{n-2}(\mc{T}_n(K)/\Gamma_n(p))$, which was answered by Miller--Patzt--Putman. We study a generalized version of Lee--Szczarba's question. We prove that for a prime $p$ in a Euclidean number ring $R$ with fraction field $K$, that a natural map $\h^\nu(\Gamma_n(p)) \to \widetilde{\h}_{n-2}(\mc{T}_n(K)/\Gamma_n(p))$ is always surjective, and give a sufficent condition on $p \in R$ that guarantees when this map is an isomorphism. We conjecture that this condition is also necessary to have an isomorphism.
\end{abstract}

\tableofcontents

\section{Introduction}

The cohomology of arithmetic groups plays a fundamental role in algebraic K-theory and number theory. Examples of arithmetic groups include special linear groups over a number ring and their finite-index subgroups.
An important class of finite-index subgroups are \emph{principal congruence subgroups}, i.e. the kernel $\Gamma_{n}(J)$ of the map $\SL_{n}(R) \to \SL_{n}(R/J)$ that reduces coefficients mod $J$, where $J \subset R$ is an ideal such that $R/J$ is finite. In this paper, we study the high degree cohomology of $\Gamma_n(J)$ when $R$ is a Euclidean domain and $J= (p)$ for $p \in R$ a prime. 

\paragraph{Borel-Serre Duality.}
Borel--Serre \cite{borel1973corners} proved that when $R$ is a number ring with fraction field $K$ and $\Gamma$ is a finite-index subgroup of $\SL_n(R)$, then $\Gamma$ is a rational duality group of some dimension $\nu$, which implies that
\begin{align*}
    \h^{\nu-i}(\Gamma; \Q) \cong \h_i(\Gamma; \mathfrak{D}\otimes \Q) \text{ for all } i
\end{align*}
for a $\Gamma$-module $\mathfrak{D}$ called the \emph{dualizing module}. This statement also holds with $\Z$-coefficients if $\Gamma$ is torsion-free.
In particular, we have
\begin{align*}
    \h^\nu(\Gamma; \Q) \cong \h_0(\Gamma; \mathfrak{D} \otimes \Q) \cong (\mathfrak{D}\otimes\Q)_{\Gamma},
\end{align*}
where the subscript indicates that we are taking coinvariants.
The dimension $\nu$ is given by:
\begin{align*}
    \nu = r \binom{n+1}{2} + cn^2 - n - r - c +1
\end{align*}
where
\begin{itemize}
    \item $r$ is the number of embeddings $K \hookrightarrow \R$
\item $c$ is the number of pairs of complex conjugate embeddings $K \hookrightarrow \bb{C}$ that do not factor through $\R$.
\end{itemize}

\paragraph{Steinberg Modules.} In the case of $\SL_n(R)$ for $R$ a number ring, the dualizing module $\mathfrak{D}$ has the following description.
For $K$ a field, let $\mc{T}_n(K)$ be the \emph{Tits building} for $\SL_n(K)$, that is, the simplicial complex whose $k$-simplices are flags
\begin{align*}
    0 \subsetneq V_0 \subsetneq \dots \subsetneq V_k \subsetneq K^n
\end{align*}
The Solomon-Tits theorem \cite{solomon1969steinberg, brown1989buildings} says that $\mc{T}_n(K)$ is homotopy equivalent to a wedge of spheres of dimension $(n-2)$. The \emph{Steinberg module} for $\SL_n(K)$, denoted $\St_n(K)$, is $\widetilde{\h}_{n-2}(\mc{T}_n(K))$.
The action of $\SL_n(K)$ (and its subgroups) on $\mc{T}_n(K)$ descends to an action on $\St_n(K)$. When $K$ is the field of fractions of a number ring, Borel--Serre proved that the dualizing module for a finite-index subgroup $\Gamma \subset \SL_n(R)$ is $\St_n(K)$.

The cohomology groups we are interested in are thus isomorphic to $(\St_n(K)\otimes \Q)_{\Gamma_n(p)}$, where $K$ is the field of fractions of a Euclidean number ring $R$, and $p \in R$ is prime. The $\otimes \Q$ is unnecessary if $\Gamma_n(p) \subset \SL_n(R)$ is torsion-free. 

\paragraph{Congruence subgroups of $\SL_n(\Z)$.} In \cite{miller2021top}, Miller--Patzt--Putman studied the top cohomology of congruence subgroups of $\SL_n(\Z)$, and answered a question of Lee--Szczarba \cite[remark on p. 28]{lee1976homology}, which we now describe.
The quotient map $\mc{T}_n(\Q) \to \mc{T}_n(\Q)/\Gamma_n(p)$ induces a $\Gamma_n(p)$-equivariant map
\begin{align}
\label{eqn:lsmap}
    \St_n(\Q) = \widetilde{\h}_{n-2}(\mc{T}_n(\Q)) \to \widetilde{\h}_{n-2}(\mc{T}_n(\Q)/\Gamma_n(p))
\end{align}

In \cite{lee1976homology}, Lee--Szczarba asked when this map induces an isomorphism on the $\Gamma_n(p)$-coinvariants of $\St_n(\Q)$ for primes $p \in \Z$.

\begin{question}[Lee--Szczarba]
\label{conj:lsiso}
    For a prime $p \in \Z$ and $n \geq 2$, when is the map
    \begin{align*}
        \h^{\binom{n}{2}}(\Gamma_n(p)) \to \widetilde{\h}_{n-2}(\mc{T}_n(\Q)/\Gamma_n(p))
    \end{align*}
    induced by \eqref{eqn:lsmap} an isomorphism?
\end{question}

Lee--Szczarba proved this map is an isomorphism for $p=3$. Their proof also yields an isomorphism for $p=2$ (in that case one has to tensor with $\Q$ because of torsion in $\Gamma_n(2) \subset \SL_n(\Z)$). They used the fact that all units of $\F_2$ and $\F_3$ can be lifted to units of $\Z$. 

In \cite{miller2021top}, Miller--Patzt--Putman gave a complete resolution of Question \ref{conj:lsiso}.

\begin{theorem}[{\cite[Theorem A]{miller2021top}}]
    \label{thm:MPPlsconj}
    For a prime $p \in \Z$ and $n \geq 2$, the map
    \begin{align*}
        \h^{\binom{n}{2}}(\Gamma_n(p)) \to \widetilde{\h}_{n-2}(\mc{T}_n(\Q)/\Gamma_n(p))
    \end{align*}
    induced by \eqref{eqn:lsmap} is a surjection. However, it is an injection if and only if $p \leq 5$. (For $p=2$ one needs to tensor with $\Q$ due to torsion in $\Gamma_n(2)$).
\end{theorem}

Their proof used the high-connectivity of certain simplicial complexes that they denoted as $\BDA_n^{\pm}(\F_p)$. The isomorphism for $p \leq 5$ in the above theorem came from a better connectivity range for these simplicial complexes for $p \leq 5$, which does not hold for $p >5$. They also used Church--Putman's proof of the Bykovskii presentation for $\St_n(\Q)$ (see \cite{church2017codimension}).

In \cite{kupers2022generalized}, Kupers--Miller--Patzt--Wilson generalized Church--Putman's proof of the Bykovskii presentation for $\St_n(K)$ when $K$ is the fraction field of the Gaussian integers $\Z[\io]$ or the Eisenstein integers $\Z[\w]$. They also posed the following question.

\begin{question}[{\cite[Question 5.5]{kupers2022generalized}}]
    For what primes in the Gaussian or Eisenstein integers does the isomorphism of Theorem \ref{thm:MPPlsconj} hold? That is, for $R = \Z[\io]$ or $\Z[\w]$ with fraction field $K$, for what primes $p \in R$ is the map
    \begin{align*}
        (\St_n(K))_{\Gamma_n(p)} \to \widetilde{\h}_{n-2}(\mc{T}_n(K)/\Gamma_n(p))
    \end{align*}
    an isomorphism?
\end{question}

We give an answer to this question in this paper.

\begin{theorem}
\label{thm:kmpwanswer}
    Let $R$ be a Euclidean number ring with fraction field $K$ and $p \in R$ a prime. Examples of pairs $(R,p)$ for which the map
    \begin{align*}
        (\St_n(K))_{\Gamma_n(p)} \to \widetilde{\h}_{n-2}(\mc{T}_n(K)/\Gamma_n(p))
    \end{align*}
    is an isomorphism for $n \geq 2$ include: $(\Z[\io], 3)$, $(\Z[\io], 2\io+1)$, $(\Z[\w], 4\w+1)$, $(\Z[\w], 4\w+3)$, $(\Z[\w], 3\w+1)$, $(\Z[\w], 3\w+2)$, $(\Z[\sqrt{2}], 5)$, $(\Z[\sqrt{2}], 3+\sqrt{2})$, $(\Z[\zeta_5], 3)$ (and many more). 
    
    Here $\zeta_5$ is a primitive fifth root of unity.
\end{theorem}

In fact, we prove a more general result, which forms the main theorem of this paper. We generalize the high-connectivity result of Miller--Patzt--Putman \cite[Proposition 2.50]{miller2021top}. Our proof also avoids their use of the Bykovskii presentation, which lets us compute the $\Gamma_n(p)$-coinvariants of $\St_n(K)$ for rings such as $\Z[\sqrt{2}]$ and $\Z[\zeta_5]$. Our result is as follows.

\begin{theorem}
    \label{thm:generallsiso}
    Let $R$ be a Euclidean number ring and $p \in R$ a prime. Let $\F = R/(p)$ and let $K$ be the field of fractions of $R$.
    Then the map
    \begin{align}
    \label{eqn:lsgeniso}
        (\St_n(K))_{\Gamma_n(p)} \to \widetilde{\h}_{n-2}(\mc{T}_n(K)/\Gamma_n(p))
    \end{align}
    is always surjective for $n \geq 2$.
    
    Let $\U \subset \F^\times$ be the image of $\R^\times$ under the map $R \to \F$. The map \ref{eqn:lsgeniso} is additionally an isomorphism for $n \geq 2$ if the quotient group $\F^\times /\U$ is trivial or has order 2.
\end{theorem}

Combining this result with Borel--Serre duality \cite{borel1973corners} yields the following corollary.

\begin{corollary}\label{cor:topcong}
    Let $R$ be a Euclidean number ring with fraction field $K$, and let $p \in R$ be a prime. Let $F$ be the quotient field $R/(p)$.

    Then for $n \geq 2$, we have a surjection
    \begin{align}
    \label{eqn:corlsiso}
        \h^\nu(\Gamma_n(p); \Q) \twoheadrightarrow \widetilde{\h}_{n-2}(\mc{T}_n(K)/\Gamma_n(p); \Q)
    \end{align}
    where $\nu$ is given by
    \begin{align*}
    \nu = r \binom{n+1}{2} + cn^2 - n - r - c +1
\end{align*}
where
\begin{itemize}
    \item $r$ is the number of embeddings $K \hookrightarrow \R$
\item $c$ is the number of pairs of complex conjugate embeddings $K \hookrightarrow \bb{C}$ that do not factor through $\R$.
\end{itemize}

Let $\U \subset \F^\times$ be the image of $\R^\times$ under the map $R \to \F$.
    Suppose in addition that the quotient group $\F^\times /\U$ is trivial or has order 2.
    Then, the map \eqref{eqn:corlsiso} is in fact an isomorphism for $n \geq 2$. The statement also holds with $\Z$-coefficients if $\Gamma_n(p)$ is torsion-free.

\end{corollary}

Theorem \ref{thm:kmpwanswer} follows from Theorem \ref{thm:generallsiso} since in the cases listed in Theorem \ref{thm:kmpwanswer}, it turns out that $\F$ and $\U$ satisfy the condition that $|\F^\times/\U| = 1$ or 2.

In the case when $|\F^\times/\U| = 1$, the prime $p$ is said to be \emph{universal side divisor} of $R$. Universal side divisors are easy to generate in a Euclidean domain $R$ with multiplicative norm: any element $p \in R$ with minimal norm that is not a unit is necessarily a prime, and every element of $R$ will either be a multiple of $p$ or differ from a multiple of $p$ by a unit in $R^\times$. It is not known to the author how to generate examples of pairs $(R,p)$ where $|\F^\times /\U|=2$, but as evidenced by Theorem \ref{thm:kmpwanswer}, many examples exist. We conjecture that the condition that $|\F^\times /\U| = 1$ or $2$ is necessary for the map \ref{eqn:lsgeniso} to be an isomorphism.

\begin{conjecture}
    Let $R$ be a Euclidean number ring and $p \in R$ a prime. Let $\F = R/(p)$ and let $K$ be the field of fractions of $R$. Let $\U \subset \F^\times$ be the image of $\R^\times$ under the map $R \to \F$.
    Then the map
    \begin{align}
        (\St_n(K))_{\Gamma_n(p)} \to \widetilde{\h}_{n-2}(\mc{T}_n(K)/\Gamma_n(p))
    \end{align}
    is an isomorphism for all $n \geq 2$ if and only if $|\F^\times/\U| = 1$ or $2$.
\end{conjecture}

When $p$ is a universal side divisor of $R$, i.e. when $|\F^\times/\U| = 1$, the analogue of Theorem \ref{thm:generallsiso} was established for congruence subgroups of the symplectic group $\Sp_{2n}(R)$ in recent work of the author (\cite[Theorem 1.5]{pal2026projective}). The analogue of the assertion of surjectivity from Theorem \ref{thm:generallsiso} was established for congruence subgroups of $\Sp_{2n}\Z$ in recent work of Capovilla-Searle \cite[Theorem A]{capovilla2026top}.

Recent work of Scalamandre \cite{scalamandre2026high} studies the structure of the cohomology group $\h^\nu(\Gamma_n(p);\mathbb{C})$ as a representation of $\SL_n(R/(p))$. Scalamandre proves that for \emph{any} number ring $R$ and prime ideal $\mathfrak{p} \subset R$, that the Steinberg representation $\St_n(R/\mathfrak{p}) \otimes \mathbb{C}$ of $\SL_n(R/\mathfrak{p})$ appears with multiplicity one in $\h^\nu(\Gamma_n(\mathfrak{p}); \mathbb{C})$ (\cite[Theorem A]{scalamandre2026high}). Furthermore, in the cases when $\h^\nu(\Gamma_n(\mathfrak{p}); \mathbb{C})$ is completely understood (such as in the cases implied by Theorem \ref{thm:generallsiso}), the $\SL_n(R/\mathfrak{p})$-representation $\St_n(R/\mathfrak{p}) \otimes \mathbb{C}$ does not appear in $\h^{\nu-1}(\Gamma_n(\mathfrak{p}); \mathbb{C})$ (\cite[Theorem B]{scalamandre2026high}).


An advantage of the isomorphism of Equation \eqref{eqn:corlsiso} is that it is possible to compute the rank of $\widetilde{\h}_{n-2}(\mc{T}_n(K)/\Gamma_n(p))$. Miller--Patzt--Putman \cite[Theorem C]{miller2021top} gave a recursive formula for the rank, which is summarized below. Though they considered the specific case where $R = \Z$ and $K = \Q$, the same proof works more generally. Their proof proceeded by relating the quotient $\mc{T}_n(K)/\Gamma_n(p)$ to a simplicial complex similar to the Tits building, and used a discrete Morse theory argument to count the number of spheres in its homotopy type.

\begin{prop}[{\cite[Theorem C]{miller2021top}}]
    \label{prop:Urank}
     Let $R$ be a PID and $K$ its field of fractions. Let $p \in R$ be a prime and suppose $\F \coloneq R/(p)$ is finite. Let $\U$ be the image of $R^\times$ under the map $R \to \F$.
     For $n \geq 1$, let $t_n$ be the rank of $\widetilde{\h}_{n-2}(\cT_n(K)/\Gamma_n(p))$. We then have $t_1 = 1$ and 
    \begin{align*}
        t_n = \left((|\F^\times/\U|-1) + (|\F^\times/\U|)|\F|^{n-1}\right)t_{n-1} + |\F/\U|(|\F^\times/\U|-1)\sum_{k=1}^{n-2}|\F|^k |\Gr_k(\F^{n-1})|t_kt_{n-k-1}
    \end{align*}
    for $n \geq 2$.
    Here $\Gr_k(V)$ denotes the set of $k$-dimensional subspaces of an $\F$-vector space $V$.
\end{prop}

\begin{remark}
    In particular, Proposition \ref{prop:Urank} above gives a loose lower bound on the rank of $t_n = \widetilde{\h}_{n-2}(\cT_n(K)/\Gamma_n(p))$. In particular, the recursive formula implies that
    \begin{align*}
        t_n \geq |\F|^{\binom{n}{2}}|\F^\times/\U|^{n-1}
    \end{align*}
    for $n \geq 1$ (with strict inequality for $n>1$).
    Indeed, note that equality holds for $n=1$, and for $n>1$ the formula of \ref{prop:Urank} recursively gives
    \begin{align*}
        t_n & > (|\F^\times/\U|)|\F|^{n-1}t_{n-1} \geq (|\F^\times/\U|)|\F|^{n-1} |\F|^{\binom{n-1}{2}}|\F^\times/\U|^{n-2} \\
        & > |\F|^{\binom{n}{2}}|\F^\times/\U|^{n-1}
    \end{align*}

    This is analogous to Paraschivescu's \cite{paraschivescu1997generalization} lower bound on $\h^{\binom{n}{2}}(\Gamma_n(p))$ for $p \in \Z$.
\end{remark}

Miller--Patzt--Putman also used the failure of certain complexes that they call $\BDA^\pm_n(\F_p)$ to be highly acyclic for $p >5$ to produce better lower bounds for $(\St_n(\Q))_{\Gamma_n(p)}$ than the one given by Theorem \ref{thm:MPPlsconj}. In the case of $n=2$, the complexes $\BDA^\pm_2(\F_p)$ are easily seen to not be 1-connected for $p >5$ by relating them to the level-$p$ modular curve, which has genus $\frac{(p+2)(p-3)(p-5)}{24}$. Miller--Patzt--Putman then used a novel spectral sequence argument to import this failure of acyclicity to $\BDA^\pm_n(\F_p)$ for $n >2$ (and $p >5)$. For number rings that are not $\Z$, it is not clear if there is a similar number theoretic interpretation of the analogues of the complexes $\BDA^\pm_2(\F_p)$. But if one can show that they fail to be 1-connected, then one should be able to use Miller--Patzt--Putman's methods more generally to obtain better lower bounds for $\h^\nu(\Gamma_n(p))$, though we expect these bounds to be highly non-optimal. We have not pursued this approach in the present article.


\paragraph{Outline.} The first main ingredient in the proof of Theorem \ref{thm:generallsiso} is a high-connectivity result for certain simplicial complexes, that are generalizations of the complexes $\BDA_n^\pm(\F_p)$ defined in \cite{miller2021top}. This high-connectivity result is proven in Section \ref{sec:simpcplx}. Our arguments are inspired by those in \cite[Section 2]{miller2021top}. One point of difference is in Lemma \ref{choicehelper}, that shows that a certain type of map from a sphere into our simplicial complexes is nullhomotopic. Miller--Patzt--Putman's proof of the corresponding statement relied on a high-connectivity result for $\Z$ that is a key ingredient in Church--Putman's proof of the Bykovskii presentation for $\St_n(\Q)$ in \cite{church2017codimension}. We give a different proof that avoids this assumption. 

The second main ingredient for Theorem \ref{thm:generallsiso} is a spectral sequence argument, that is described in Section \ref{sec:ssargument}. This argument is relatively standard and is similar to the one used by Church--Putman \cite{church2017codimension}.
This is another point where we avoid assuming the Bykovskii presentation for the Steinberg module. To do this we need a cyclicity result about $\St_n(K)$ as an $\SL_n(R)$-module when $R$ is a Euclidean domain with fraction field $K$, which was proved by Ash--Rudolph \cite{ash1979modular}, with a different proof later given by Church--Putman \cite{church2017codimension}, that used high-connectivity arguments on certain simplicial complexes. 

\paragraph{Acknowledgements.} I would like to thank my advisor Jenny Wilson for the unending support, encouragement, and immense generosity with her time. I thank Jeremy Miller and Peter Patzt for helpful conversations, and for feedback on drafts of this paper. I am grateful for the support of the University of Michigan's Rackham one-term dissertation fellowship and the Rackham predoctoral fellowship. I am grateful for travel support from the NSF grant DMS-2142709.


\section{Simplicial complexes associated to free $R$-modules}
\label{sec:simpcplx}

We start by defining several simplicial complexes that we will use in our proofs. Throughout all these definitions, $R$ will denote a commutative ring, and $\U$ a multiplicative subgroup of its group of units $R^\times$ such that $-1 \in \U$. We give the same definitions as Miller--Patzt--Putman in \cite{miller2021top} -- they work with the specific case when $\U = \{\pm 1\}$, but all their definitions easily generalize.

\subsection{Complexes of bases and augmented bases}

\subsubsection{Partial $\U$-bases}

Let $R$ be a commutative ring, and let $R^{\times}$ denote the set of units of $R$. Let $\U$ be a multiplicative subgroup of $R^\times$ such that $-1 \in \U$. We make the following definitions.

\begin{definition}
    A \emph{$\U$-vector} in $R^n$ is a set of the form $\{ c\vec{v} | c \in \U \}$ for a nonzero vector $\vec{v} \in R^n$. Given a non-zero $\vec{v} \in R^n$, we will write $[\vec{v}]$ for the associated $\U$-vector.
\end{definition}

\begin{definition}
    A \emph{partial-basis} for $R^n$ is a set of elements of $R^n$ that is a subset of a free basis for $R^n$.
    A \emph{partial $\U$-basis} for $R^n$ is a set $\{ [\vec{v}_1], \dots , [\vec{v}_k]\}$ of $\U$-vectors such that $\{ \vec{v}_1, \dots, \vec{v}_k\}$ is a partial-basis for $R^n$. This does not depend on the choice of representatives $\vec{v_i}$.
\end{definition}

We now turn these into a simplicial complex as follows. Here and throughout the rest of this paper, we will define simplicial complexes by specifying that their simplices are certain sets.
What we mean by this is that the $k$-simplices are such sets containing $(k + 1)$-elements, and the face relations between simplices are simply inclusions of sets.

\begin{definition}
    The \emph{complex of partial $\U$-bases}, denoted $\B_n^\U(R)$, is the simplicial complex whose simplices are partial $\U$-bases for $R^n.$
\end{definition}

To understand $\B_n^\U(R)$ in an inductive way, we will need to understand the links of its simplices. We thus make the following definition.

\begin{definition}
    Let $\{ \vec{e}_1, \vec{e}_2, \dots, \vec{e}_{m+n} \}$ denote the standard basis for $R^{m+n}$. Define $\B_{n,m}^\U(R) = \Link_{\B_{m+n}^\U(R)}\{[\vec{e}_1], \dots, [\vec{e}_m] \}$.
\end{definition}

Recall that a simplicial complex $X$ is \emph{Cohen-Macaulay} of dimension $r$ if it satisfies the following condition:
\begin{itemize}
    \item X is $r$-dimensional and $(r-1)$-connected.
    \item For all $k$-simplices $\sigma$ of $X$, $\Link_X(\sigma)$ is $(r-k-1)$-dimensional and $(r-k-2)$-connected.
\end{itemize}
Miller--Patzt--Putman \cite{miller2021top} proved the following theorem for $\F$ a field, and $\U \subset \F^\times$ a multiplicative subgroup with $-1 \in \U$. They state their result in the specific case where $\U = \{\pm1\}$, but the same proof adapts to the general case:

\begin{prop}[{\cite[Proposition 2.21]{miller2021top}}]
\label{BUconnectivity}
    For a field $\F$ and $\U \subset \F^\times$ a multiplicative subgroup with $-1 \in \U$, the complex $\B_{n,m}^\U(\F)$ is Cohen-Macaulay of dimension $(n-1)$ for all $n,m \geq 0$.
\end{prop}

\subsubsection{Augmented $\U$-bases}

We now define the augmented version of $\B_{n,m}^\U(R)$.

\begin{definition}
    An \emph{augmented partial $\U$-basis} for $R^n$ is a set $\{[\vec{v}_0], [\vec{v}_1], \dots, [\vec{v}_k]\}$ of $\U$-vectors that can be reordered so that 
    \begin{itemize}
        \item $\{[\vec{v}_1], \dots, [\vec{v}_k]\}$ is a partial $\U$-basis for $R^n$
        \item There exist units $\lambda, \nu \in R^\times$ such that $\vec{v}_0 = \lambda\vec{v}_1 + \nu \vec{v}_2$. The existence of $\lambda$ and $\nu$ is independent of the choice of representatives $\vec{v}_1, \vec{v}_2$.
    \end{itemize}
\end{definition}

We will call the $\U$-vectors $\{ [\vec{v}_0], [\vec{v}_1], [\vec{v}_2]\}$ the \emph{additive core} of $\{[\vec{v}_0], [\vec{v}_1], \dots, [\vec{v}_k]\}$.

A subset of an augmented $\U$-basis is either itself an augmented $\U$-basis (if the subset contains the entire additive core) or is a partial $\U$-basis. We can thus make the following definition:

\begin{definition}
    The \emph{complex of augmented partial $\U$-bases} for $R^n$, denoted $\BA_n^\U$, is the simplicial complex whose simplices consist either of partial $\U$-bases of $R^n$ or of augmented partial $\U$-bases of $R^n$.
\end{definition}

Again, we will need to study the links of simplices in $\BA_{n}^\U(R)$. However, for technical reasons, we will not look at the entire link, but rather the following subcomplex of it:

\begin{definition}
    Let $\sigma = \{[\vec{v}_1], \dots, [\vec{v}_k]\}$ be a simplex of $\BA_n^\U(R)$. The \emph{augmented link} of $\sigma$, denoted $\widehat{\Link}_{\BA_n^\U(R)}(\sigma)$, is the full subcomplex of $\Link_{\BA_n^\U(R)}(\sigma)$ spanned by vertices $[\vec{w}]$ of $\Link_{\BA_n^\U(R)}(\sigma)$ such that $\vec{w} \not\in \langle \vec{v}_1, \dots, \vec{v}_k \rangle$. This definition does not depend on the choice of representatives $\vec{v}_i$ or $\vec{w}$.
\end{definition}

The simplices of $\widehat{\Link}_{\BA_n^\U(R)}(\sigma)$ fall into the following 3 classes:

\begin{definition}
Let $\sigma = \{ [\vec{v}_1], \dots, [\vec{v}_k]\}$ be a simplex of $\BA_n^\U(R)$. Let $\eta$ be a simplex of $\widehat{\Link}_{\BA_n^\U(R)}(\sigma)$. Then one of the following three conditions hold:
\begin{itemize}
    \item $\eta$ is a partial $\U$-basis for $R^n$. In this case we will call $\eta$ a \emph{standard simplex}.
    \item $\eta$ is an augmented partial $\U$-basis for $R^n$, i.e. we can write $\eta = \{ [\vec{w}_0], \dots, [\vec{w}_l]\}$ so that $\vec{w}_0 = \lambda \vec{w}_1 + \nu \vec{w}_2$ for some $\lambda, \nu \in R^\times$. In this case we will call $\eta$ an \emph{internally additive simplex}.
    \item We can write $\eta = \{ [\vec{w}_0], \dots, [\vec{w}_l]\}$ with $\vec{w}_0 = \lambda \vec{w}_1 + \nu \vec{v}_i$ for some $\lambda, \nu \in R^\times$ and $1 \leq i \leq k$. In this case we call $\eta$ an \emph{externally additive simplex}.
\end{itemize}

We will sometimes call a simplex that is either internally or externally additive simply as an \emph{additive simplex}.

\end{definition}

Just as we did for $\B_{n,m}^\U(R)$, we make the following definition.

\begin{definition}
    Let $\{ \vec{e}_1, \vec{e}_2, \dots, \vec{e}_{m+n} \}$ denote the standard basis for $R^{m+n}$. Define $\BA_{n,m}^\U(R) = \widehat{\Link}_{\BA_{m+n}^\U(R)}\{[\vec{e}_1], \dots, [\vec{e}_m] \}$.
\end{definition}

Miller--Patzt--Putman proved the following analogue of Proposition \ref{BUconnectivity} for the complexes $\BA_{n,m}^\U(\F)$. They state their result in the specific case when $\U = \{\pm1\}$, but the same proof adapts to the general setting.

\begin{prop}[{\cite[Proposition 2.29]{miller2021top}}]
    \label{BAUconnectivity}
     For a field $\F$ and $\U \subset \F^\times$ a multiplicative subgroup with $-1 \in \U$, the complex $\BA_{n,m}^\U(\F)$ is Cohen-Macaulay of dimension $n$ for all $n \geq 1$ and $m \geq 0$ with $n+m \geq 2$.
\end{prop}

\subsection{Complexes of determinant $\U$ partial $\U$-bases}
\label{sec:augdet}

For our purposes, what we really need are subcomplexes of the complexes of partial $\U$-bases where we impose a determinant condition.

\subsubsection{Determinant-$\U$ partial $\U$-bases}

\begin{definition}
    A partial $\U$-basis $\{[\vec{v_1}], \dots, [\vec{v_k}]\}$ for $R^n$ is a \emph{determinant-$\U$ partial $\U$-basis} if it satisfies the following:
    \begin{itemize}
        \item If $k=n$, then the determinant of the matrix $(\vec{v}_1, \dots, \vec{v}_n)$ whose columns are the $\vec{v_i}$ is in $\U$. This does not depend on the choice of representatives $\vec{v_i}$ or their ordering.
        \item If $k<n$, then no additional condition needs to be satisfied.
    \end{itemize}
\end{definition}

These form a simplicial complex:

\begin{definition}
    The \emph{complex of determinant-$\U$ partial $\U$-bases}, denoted $\BD_n^\U(R)$, is the simplicial complex whose simplices are determinant-$\U$ partial $\U$-bases for $R^n$.
\end{definition}

Just as for $\B_n^\U(R)$, we will need to consider links as well:

\begin{definition}
    Let $\{ \vec{e}_1, \dots, \vec{e}_{m+n}\}$ denote the standard basis for $R^n$. Define $\BD_{n,m}^\U(R) = \Link_{\BD_{m+n}^\U(R)}\{ [\vec{e}_1], \dots, [\vec{e}_m] \}$.
\end{definition}

\subsubsection{Augmented determinant-$\U$ partial $\U$-bases}

We now define the augmented versions of the determinant-$\U$ partial $\U$-bases.

\begin{definition}
    An \emph{augmented determinant-$\U$ partial-$\U$ basis} is a set $\{[\vec{v}_0], \dots, [\vec{v}_k]\}$ of $\U$-vectors that can be reordered so that:
    \begin{itemize}
        \item $\{ [\vec{v}_1], \dots, [\vec{v}_k]\}$ is a determinant-$\U$ partial $\U$-basis for $R^n$, and
        \item There exist $\lambda, \nu \in \U$ such that $\vec{v}_0 = \lambda\vec{v}_1 + \nu \vec{v}_2$.
    \end{itemize}
    We shall call the $\U$-vectors $\{[\vec{v}_0], [\vec{v}_1], [\vec{v}_2]\}$ the \emph{additive core} of $\{[\vec{v}_0], [\vec{v}_1], \dots, [\vec{v}_k]\}$.
\end{definition}

A subset of an augmented determinant-$\U$ partial $\U$-basis is either itself an augmented determinant-$\U$ partial-$\U$ basis (if the subset contains the entire additive core) or a determinant-$\U$ partial $\U$-basis (this uses the fact that $\lambda, \nu$ are in $\U$ rather than in $R^\times$). We can thus make the following definition:

\begin{definition}
    The \emph{complex of augmented determinant-$\U$ partial $\U$-bases} for $R^n$, denoted $\BDA_n^\U(R)$, is the simplicial complex whose simplices are determinant-$\U$ partial $\U$-bases or augmented determinant-$\U$ partial $\U$-bases for $R^n$.
\end{definition}

We now make a series of definitions that are similar to the ones we made for $\BA_n^\U(R)$.

\begin{definition}
    Let $\sigma = \{[\vec{v}_1], \dots, [\vec{v}_k]\}$ be a simplex of $\BDA_n^\U(R)$. The \emph{augmented link} of $\sigma$, denoted $\widehat{\Link}_{\BDA_n^\U(R)}(\sigma)$, is the full subcomplex of $\Link_{\BDA_n^\U(R)}(\sigma)$ spanned by vertices $[\vec{w}]$ of $\Link_{\BDA_n^\U(R)}(\sigma)$ such that $\vec{w} \not\in \langle \vec{v}_1, \dots, \vec{v}_k \rangle$. This definition does not depend on the choice of representatives $\vec{v}_i$ or $\vec{w}$.
\end{definition}

The simplices of $\widehat{\Link}_{\BDA_n^\U(R)}(\sigma)$ fall into the following 3 classes:

\begin{definition}
Let $\sigma = \{ [\vec{v}_1], \dots, [\vec{v}_k]\}$ be a simplex of $\BDA_n^\U(R)$. Let $\eta$ be a simplex of $\widehat{\Link}_{\BDA_n^\U(R)}(\sigma)$. Then one of the following three conditions hold:
\begin{itemize}
    \item $\eta$ is a partial $\U$-basis for $R^n$. In this case we will call $\eta$ a \emph{standard simplex}.
    \item $\eta$ is an augmented partial $\U$-basis for $R^n$, i.e. we can write $\eta = \{ [\vec{w}_0], \dots, [\vec{w}_l]\}$ so that $\vec{w}_0 = \lambda \vec{w}_1 + \nu \vec{w}_2$ for some $\lambda, \nu \in \U$. In this case we will call $\eta$ an \emph{internally additive simplex}.
    \item We can write $\eta = \{ [\vec{w}_0], \dots, [\vec{w}_l]\}$ with $\vec{w}_0 = \lambda \vec{w}_1 + \nu \vec{v}_i$ for some $\lambda, \nu \in \U$ and $1 \leq i \leq k$. In this case we call $\eta$ an \emph{externally additive simplex}.
\end{itemize}

We will sometimes call a simplex that is either internally or externally additive simply an \emph{additive simplex}.

\end{definition}

Just as we did for $\BD_{n,m}^\U(R)$, we make the following definition.

\begin{definition}
    Let $\{ \vec{e}_1, \vec{e}_2, \dots, \vec{e}_{m+n} \}$ denote the standard basis for $R^{m+n}$. Define $\BDA_{n,m}^\U(R) = \widehat{\Link}_{\BDA_{m+n}^\U(R)}\{[\vec{e}_1], \dots, [\vec{e}_m] \}$.
\end{definition}
\subsubsection{The complexes of (un)augmented determinant-$\U$ partial $\U$-bases are highly connected}

In this section we summarize some results from Miller--Patzt--Putman \cite{miller2021top}, whose proofs go through with minor changes in our more general setting:

\begin{lemma}[{\cite[Lemma 2.46]{miller2021top}}]
\label{BDretract}
    Suppose $R=\F$ is a field, and let $\U$ be a multiplicative subgroup of $\F^\times$ such that $-1 \in \U$. Then the complex $\BD_{n,m}^\U(\F)$ is a retract of $\B_{n,m}^\U(\F)$.
\end{lemma}
Since we will make use of this retraction in Section \ref{retraction}, we include a proof here.
\begin{proof}
    Let $\{ \vec{e}_1, \dots, \vec{e}_{m+n}\}$ denote the standard basis for $\F^{m+n}$. It is enough to say how $\rho$ maps a simplex of $\B_{n,m}^\U(\F)$ that is not in $\BD_{n,m}^\U(\F)$. The only such simplices are $(n-1)$-dimensional simplices $\sigma = \{ [\vec{v}_1], \dots, [\vec{v}_n]\}$ such that the matrix with columns $(\vec{e}_1, \dots, \vec{e_m}, \vec{v}_1, \dots, \vec{v}_n)$ has determinant equal to $d  \in \F^\times \setminus U$.
    Let $S(\sigma)$ be the result of subdividing $\sigma$ with a new vertex $x_{\sigma}$. Thus the top-dimensional simplices of $S(\sigma)$ are of the form
    \begin{align*}
        \{ [\vec{v}_1], \dots, [\widehat{\vec{v}}_i], \dots, [\vec{v}_n], x_{\sigma}\} \hspace{1cm}  (1 \leq i \leq n) 
    \end{align*}

Define
\begin{align*}
    \rho|_{\sigma} : \sigma \cong S(\sigma) \to \BD_{n,m}^\U(\F)
\end{align*}
to be the map that fixes the vertices $[\vec{v}_1], \dots, [\vec{v}_n]$ and takes $x_{\sigma}$ to $\frac{1}{d}(\vec{v}_1 + \dots + \vec{v}_n)$.
The following calculation checks that this extends over the top-dimensional simplices of $S(\sigma)$:

\begin{align*}
    & \det \left( \vec{e_1}, \dots, \vec{e}_m, \vec{v}_1, \dots \widehat{\vec{v}}_i, \dots, \vec{v}_n, \frac{1}{d}(\vec{v}_1 + \dots + \vec{v}_n)\right)  \\
    & = \frac{1}{d}  \det \left( \vec{e_1}, \dots, \vec{e}_m, \vec{v}_1, \dots \widehat{\vec{v}}_i, \dots, \vec{v}_n, (\vec{v}_1 + \dots + \vec{v}_n)\right)  = \pm d/d = \pm 1 \in \U \\
\end{align*}
    
\end{proof}

As a consequence of this we have the following:

\begin{prop}[{\cite[Proposition 2.45]{miller2021top}}]
\label{BDconn}
     For $R = \F$ a field and $\U \subset \F^\times$ a multiplicative subgroup with $-1 \in \U$, the complex $\BD_{n,m}^\U(\F)$ is Cohen-Macaulay of dimension $n-1$ for all $n, m \geq 0$.
\end{prop}
\begin{proof}
    Combining Lemma \ref{BDretract} with Proposition \ref{BUconnectivity} (which says that $\B_{n,m}^\U(\F)$ is $(n-2)$-connected), we deduce that $\BD_{n,m}^\U(\F)$ is $(n-2)$-connected. Since the link of a $k$-simplex in $\BD_{n,m}^\U(\F)$ is isomorphic to $\BD^\U_{n-k-1,m+k+1}(\F)$, this implies that $\BD_{n,m}^\U(\F)$ is Cohen-Macaulay, as desired.
\end{proof}

We now turn to stating high-connectivity results for the augmented versions $\BDA_{n,m}^\U(\F)$.

\begin{lemma}[{\cite[Lemma 2.49]{miller2021top}}]
\label{BDApiSurj}
 For $R=\F$ a field and $\U \subset \F^\times$ a multiplicative subgroup with $-1 \in \U$, the inclusion map $\BD_{n,m}^\U(\F) \hookrightarrow \BDA_{n,m}^\U(\F)$ induces a surjection on $\pi_k$ for $1 \leq k \leq n-1$ for all $n,m \geq 0$.
\end{lemma}

As an immediate consequence of Lemma \ref{BDApiSurj} and Lemma \ref{BDconn}, we have the following:

\begin{prop}[{\cite[Proposition 2.47]{miller2021top}}]
\label{BDAconnectivity}
   For $R=\F$ a field and $U \subset \F^\times$ a multiplicative subgroup with $-1 \in \U$, the complex $\BDA_{n,m}^\U(\F)$ is $(n-2)$-connected.
\end{prop}

\subsection{Improving the connectivity in some cases}

In this section we show that the connectivity of $\BDA_{n,m}^\U(\F)$ can be improved when $\F$ and $\U$ satisfy some additional conditions. We state our main result and a skeleton of its proof in Section \ref{sec:connskeleton}. The proof is divided amongst subsequent sections.

\subsubsection{Statement and Skeleton of Proof}
\label{sec:connskeleton}

Our result is as follows:

\begin{prop}
\label{primeconditions}
    Suppose $\F$ is a finite field and $\U \subset \F^\times$ a multiplicative subgroup such that $-1 \in \U$, and the quotient $\F^\times / \U$ is isomorphic to the cyclic group $\{\pm 1\}$ of order 2.
    
Then the complex $\BDA_{n}^\U(\F)$ is $(n-1)$-connected for all $n \geq 1$.  
\end{prop}

\begin{remark}
    Note that $\BDA_1^\U(\F) = \{*\}$ is a singleton set, so the above proposition is trivially true for $n=1$. Thus the content of the Proposition lies in the case $n \geq 2$.
\end{remark}

\begin{remark}
We will primarily be interested when $\F$ and $\U$ arise in the following way: Let $R$ be a Euclidean domain, and let $\F$ be a field that is a quotient of $R$ by a prime $p \in R$. Let $\U$ be the subgroup of $\F$ formed by the image of $R^\times$ under the map $R \to \F$. We will prove at the end of Section \ref{sec:thm1.4&1.5proof} that the hypothesis of Proposition \ref{primeconditions} is satisfied when the pair $(R,p)$ is among $(\Z[\io], 3)$, $(\Z[\w], 4\w+1)$, $(\Z[\w], 4\w+3)$, $(\Z[\sqrt{2}], 5)$, and $(\Z[\zeta_5], 3)$ (where $\zeta_5$ is a primitive fifth root of unity). Thus the above theorem generalizes a result of Miller--Patzt--Putman \cite[Proposition 2.50]{miller2021top}, which is a special case of Proposition \ref{primeconditions} when $R=\Z$ and $p = 5$.
\end{remark}

\begin{notation}
    Let $X$ be a simplicial complex and let $\Delta^{k-1}$ be a $(k-1)$-simplex.

    \begin{itemize}
        \item Suppose $v_1, v_2, \dots, v_k$ are (not necessarily distinct) vertices of $X$, so that $\{ v_1, \dots, v_k\}$ is a simplex. Define $\llb v_1, \dots, v_k \rrb$ to be the map
        \begin{align*}
            \llb v_1, \dots, v_k \rrb : \Delta^{k-1} \to X
        \end{align*}
        taking the vertices of $\Delta^{k-1}$ to the $v_i$.

        \item Suppose $v_1, \dots, v_k$ are vertices of $X$ such that $\{v_1, \dots, \hat{v}_i, \dots, v_k\}$ is a simplex of $X$ for every $1 \leq i \leq k$.
        Define $\overline{\llb v_1, \dots, v_k \rrb}$ to be the map
        \begin{align*}
            \overline{\llb v_1, \dots, v_k \rrb} : \del \Delta^{k-1} \to X
        \end{align*}
        taking the vertices of $\Delta^{k-1}$ to the $v_i$.

        \item Suppose $Y, Z$ are simplicial complexes, and we have simplicial maps $f: Y \to X$ and $g: Z \to X$ such that for every simplex $\sigma$ of $Y$ and $\eta$ of $Z$, $f(\sigma) * g(\eta)$ is a simplex of $X$. Let $f * g$ denote the natural map $f*g : Y * Z \to X$.
    \end{itemize}
\end{notation}

The following lemma from Miller--Patzt--Putman \cite{miller2021top} gives generators for $\pi_{n-1}(\B_n^{\U}(\F))$. For a finite-dimensional $\F$-vector space $V$, we write $\B^{\U}(V)$ for the complex of partial $\U$-bases of $V$, so $\B^\U(\F^n) = \B_n^\U(\F)$.

\begin{lemma}[{\cite[Lemma 2.53]{miller2021top}}]
\label{pigenerators}
    For $n \geq 3$, the group $\pi_{n-1}(\B_n^{\U}(\F))$ is generated by the following two families of generators:

    \begin{itemize}
        \item The \textbf{initial D-triangle maps.} Let $\sigma = \{ [\vec{v}_0], [\vec{v}_1], [\vec{v}_2]\}$ be an additive simplex of $\BA_n^{\U}(\F)$ with $\vec{v}_0 = \lambda \vec{v}_1 + \nu \vec{v}_2$ for some $\lambda, \nu \in \U$. Let $f : S^{n-3} \to 
        \Link_{\BA_n^\U(\F)}(\sigma)$ be a simplicial map for some triangulation of $S^{n-3}$. Then the associated initial D-triangle map is 
        \begin{align*}
            \overline{\llb [\vec{v}_0], [\vec{v}_1], [\vec{v}_2] \rrb} * f : \del \Delta^2 * S^{n-3} \cong S^{n-1} \to \B_n^\U(\F).
        \end{align*}

        \item The \textbf{initial D-suspend maps.} Let $\vec{v} \in \F^n$ be a non-zero vector, and let $W$ be an $(n-1)$-dimensional subspace of $\F^n$ such that $\F^n = \langle \vec{v} \rangle \oplus W$, and let $\vec{w} \in W$ be non-zero. Let $f: S^{n-2} \to \B^{\U}(W)$ be a simplicial map for some triangulation of $S^{n-2}$. The associated initial D-suspend map is
        \begin{align*}
            \overline{\llb [\vec{v}], [\vec{v}+\vec{w}] \rrb} * f : \del \Delta^1 \cong S^{n-1} \to \B_n^\U(\F).
        \end{align*}
    \end{itemize}
\end{lemma}

The proof will require some modificatons to the one given in \cite{miller2021top}. We postpone the proof to Section \ref{generators}.

To finish the proof of Proposition \ref{primeconditions}, it is enough to prove the following lemmas. 

\begin{lemma}
\label{killtriangle}
Let $\F$ be a finite field and $\U \subset \F^\times$ a multiplicative subgroup such that $-1 \in \U$, and the quotient $\F^\times / \U$ is isomorphic to the cyclic group $\{\pm 1\}$ of order 2.

For some $n \geq 3$, let $g: S^{n-1} \to \B_n^\U(\F)$ be an initial D-triangle map, let $\rho: \B_n^\U(\F) \to \BD_n^\U(\F)$ be the retraction from Lemma \ref{BDretract}, and let $\iota : \BD_n^\U(\F) \to \BDA_n^\U(\F)$ be the inclusion. Then the composite $\iota \circ \rho \circ g : S^{n-1} \to \BDA_n^\U(\F)$ is nullhomotopic.
\end{lemma}

\begin{lemma}
\label{killsuspend}
Let $\F$ be a finite field and $\U \subset \F^\times$ a multiplicative subgroup such that $-1 \in \U$, and the quotient $\F^\times / \U$ is isomorphic to the cyclic group $\{\pm 1\}$ of order 2.

    For some $n\geq 3$, let $g: S^{n-1} \to \B_n^\U(\F)$ be an initial D-suspend map, let $\rho: \B_n^\U(\F) \to \BD_n^\U(\F)$ be the retraction given by Lemma \ref{BDretract}, and let $\iota: \BD_n^\U(\F) \hookrightarrow \BDA_n^\U(\F)$ be the inclusion. Assume that $\pi_{n-2}(\BDA_{n-1}^\U(\F)) = 0$. Then $\iota \circ \rho \circ g: S^{n-1} \to \BDA_n^\U(\F)$ is nullhomotopic.
\end{lemma}

\begin{lemma}
    \label{lem:BDA2conn}
    Let $\F$ be a finite field and $\U \subset \F^\times$ a multiplicative subgroup such that $-1 \in \U$, and the quotient $\F^\times / \U$ is isomorphic to the cyclic group $\{\pm 1\}$ of order 2. Then $\BDA_2^\U(\F)$ is simply connected.
\end{lemma}

Here is an outline of the remainder of this section. In \ref{generators} we prove Lemma \ref{pigenerators}. In \ref{sec:F/UImplications}, deduce some consequences of the restriction of $\F^\times / \U \cong \{\pm 1\}$ that we will frequently use later on. In \ref{retraction}, we prove some preliminary results about the retraction given by Lemma \ref{BDretract}. In \ref{triangles} and \ref{suspends}, we prove Lemma \ref{killtriangle} and \ref{killsuspend}, respectively. In Section \ref{sec:BDA2}, we prove Lemma \ref{lem:BDA2conn}.

\subsubsection{A generating set for $\pi_{n-1}(\B_{n,m}^\U(\F))$}\label{generators}

Let $\F$ be a field, and $\U$ a multiplicative subgroup of $\F^\times$ containing $-1$. This section proves Lemma \ref{pigenerators}, which identifies generators for $\pi_{n-1}(\B_{n,m}^\U(\F))$.
We start by summarizing some results by Miller--Patzt--Putman \cite[Section 2.3.2]{miller2021top}. Their statements apply to the specific case when $\U = \{\pm1\}$, but the proofs for the general case are essentially the same as theirs. The main idea of the proofs is to include $\B_{n,m}^\U(\F)$ into $\BA_{n,m}^\U(\F)$, which by Proposition \ref{BAUconnectivity} is $(n-1)$-connected.
Their results are phrased in terms of homology groups rather than homotopy groups, which allows us to avoid worrying about basepoints. They later use the Hurewicz theorem to translate this into information about homotopy groups.
Throughout this section, our convention is that $S^{-1}$ is the empty set.

\begin{lemma}[{\cite[Lemma 2.57]{miller2021top}}]
    Let $\F$ be a field, and $\U$ a multiplicative subgroup of $\F^\times$ with $-1 \in \U$. Let $n \geq 1$ and $m \geq 0$ be such that $n+m \geq 2$. Then the group $\widetilde{\h}_{n-1}(\B_{n,m}^\U(\F))$ is generated by the images of the fundamental classes under the following two families of maps.

    \begin{itemize}
        \item The \textbf{initial triangle maps}, which require $n \geq 2$. Let $\sigma = \{ [\vec{v}_0], [\vec{v}_1], [\vec{v}_2]\}$ be a 2-dimensional internally additive simplex of $\BA_{n,m}^\U(\F)$, so $\vec{v}_0 = \lambda \vec{v}_1 + \nu \vec{v}_2$ for some $\lambda, \nu \in \F^\times$. Let $f: S^{n-3} \to \Link_{\BA_{n,m}^\U(\F)}(\sigma)$ be a simplicial map for some triangulation of $S^{n-3}$. The associated initial triangle map is then
        \begin{align*}
            \overline{\llb [\vec{v}_0], [\vec{v}_1], [\vec{v}_2] \rrb} * f : \del \Delta^2 * S^{n-3} \cong S^{n-1} \to \B_{n,m}^\U(\F)
        \end{align*}

        \item The \textbf{initial external suspend maps}, which require $m \geq 1$. Let $\sigma = \{[\vec{v}_0], [\vec{v}_1]\}$ be a 1-dimensional externally additive simplex of $\BA_{n,m}^\U(\F)$. Let $f: S^{n-2} \to \Link_{\BA_{n,m}^\U(\F)}(\sigma)$ be a simplicial map for some triangulation of $S^{n-2}$. The associated initial external suspend map is then
        \begin{align*}
            \overline{\llb [\vec{v}_0], [\vec{v}_1] \rrb} * f : \del \Delta^1 * S^{n-2} \cong S^{n-1} \to \B_{n,m}^\U(\F). 
        \end{align*}
    \end{itemize}
\end{lemma}

\begin{lemma}[{\cite[Lemma 2.58]{miller2021top}}]
\label{absgenerators}
    Let $\F$ be a field, and $\U$ a multiplicative subgroup of $\F^\times$ with $-1 \in \U$. Let $n \geq 1$ and $m \geq 0$ be such that $n+m \geq 2$. Then the group $\widetilde{\h}_{n-1}(\B_{n,m}^\U(\F))$ is generated by the images of the fundamental classes under maps of the form 
    \begin{align*}
        f_1 * \dots * f_k : \del \Delta^{r_1} * \dots * \del \Delta^{r_k} \cong S^{n-1} \to \B_{n,m}^\U(\F),
    \end{align*}
    where the $f_i$ are as follows. There exists a decomposition $\F^{n+m} = \F^m \oplus A_1 \oplus \dots \oplus A_k$, and for $1 \leq i \leq k$ the map $f_i$ falls into one of the following two classes:
    \begin{itemize}
        \item A \textbf{triangle}. There exists a 2-dimensional internally additive simplex $\{[\vec{v}_0], [\vec{v}_1], [\vec{v}_2]\}$ of $\BA_{n,m}^\U(\F)$ such that 
        \begin{align*}
            f_i = \overline{\llb [\vec{v}_0], [\vec{v}_1], [\vec{v}_2]\rrb} : \del \Delta^2 \to \B_{n,m}^\U(\F)
        \end{align*}
        and such that $A_i = \langle \vec{v}_0, \vec{v}_1, \vec{v}_2 \rangle$. Note that $A_i$ is 2-dimensional.

        \item A \textbf{suspend}. There exist nonzero vectors $\vec{v} \in A_i$ and $\vec{w} \in \F^m \oplus A_1 \oplus \dots \oplus A_{i-1}$ and some $\lambda \in \F^\times$ such that
        \begin{align*}
            f_i = \overline{\llb [\vec{v}], [\lambda \vec{v} + \vec{w}] \rrb} : \del \Delta^1 \to \B_{n,m}^\U(\F)
        \end{align*}
        and such that $A_i = \langle \vec{v} \rangle$. Note that $A_i$ is 1-dimensional.
    \end{itemize}
\end{lemma}

To prove Lemma \ref{pigenerators}, we shall need a variant of Lemma \ref{absgenerators}. This follows essentially the same way as \cite[Lemma 2.59]{miller2021top}, but since our more general setting requires some modifications to both the statement and its proof, we give a full proof here. 

\begin{lemma}
    \label{absgenvariant}
    Let $n \geq 1$ and $m \geq 0$ be such that $n + m \geq 2$. 
    Let $\F$ be a field and $\U$ a multiplicative subgroup of $\F^\times$ with $-1 \in \U$. 
    Then the group $\widetilde{\h}_{n-1}(\B_{n,m}^\U(\F))$ is generated by the images of the fundamental classes under maps of the form 
    \begin{align*}
        f_1 * \dots * f_k : \del \Delta^{r_1} * \dots * \del \Delta^{r_k} \cong S^{n-1} \to \B_{n,m}^\U(\F),
    \end{align*}
    where the $f_i$ are as follows. There exists a decomposition $\F^{n+m} = \F^m \oplus A_1 \oplus \dots \oplus A_k$, and for $1 \leq i \leq k$ the map $f_i$ falls into one of the following three classes:

    \begin{itemize}
        \item A \textbf{D-triangle}. There is a 2-dimensional internally additive simplex $\{[\vec{v}_0], [\vec{v}_1], [\vec{v}_2]\}$ of $\BA_{n,m}^\U(\F)$ with $\vec{v}_0 = \lambda \vec{v}_1 + \nu \vec{v}_2$ for some $\lambda, \nu \in \U$ such that
         \begin{align*}
            f_i = \overline{\llb [\vec{v}_0], [\vec{v}_1], [\vec{v}_2]\rrb} : \del \Delta^2 \to \B_{n,m}^\U(\F)
        \end{align*}
        and such that $A_i = \langle \vec{v}_0, \vec{v}_1, \vec{v}_2 \rangle$. Note that $A_i$ is 2-dimensional.

        \item A \textbf{D-suspend}.  There exist nonzero vectors $\vec{v} \in A_i$ and $\vec{w} \in \F^m \oplus A_1 \oplus \dots \oplus A_{i-1}$ such that
        \begin{align*}
            f_i = \overline{\llb [\vec{v}], [\vec{v} + \vec{w}] \rrb} : \del \Delta^1 \to \B_{n,m}^\U(\F)
        \end{align*}
        and such that $A_i = \langle \vec{v} \rangle$. Note that $A_i$ is 1-dimensional.

        \item A \textbf{multi-suspend}. There is a nonzero vector $\vec{v} \in A_i$ and some $c \in \F^\times \setminus \U$ such that 
        \begin{align*}
            f_i = \overline{\llb [\vec{v}], [\vec{cv}] \rrb} : \del \Delta^1 \to \B_{n,m}^\U(\F)
        \end{align*}
        and such that $A_i = \langle \vec{v} \rangle$. Note that $A_i$ is 1-dimensional.
    \end{itemize}
    Moreover, if $m=0$ then atleast one of the $f_i$ is either a D-triangle or a D-suspend.
\end{lemma}

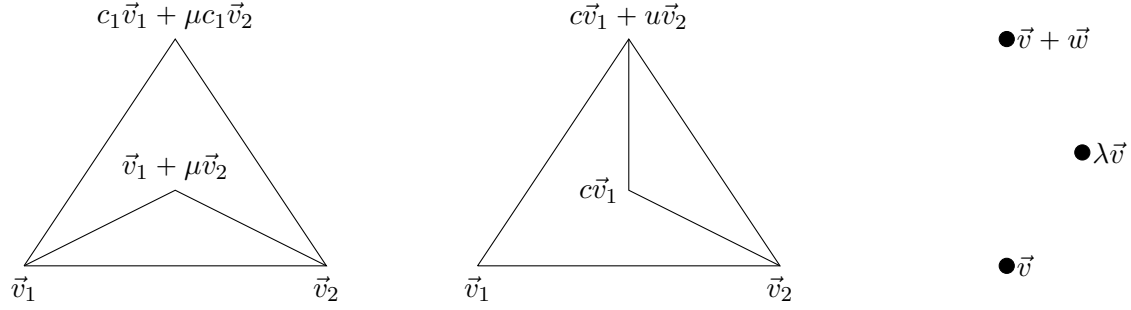
\begin{figure}
    \begin{tikzpicture}
        \draw (0,0) -- (2,3) -- (4,0) -- (0,0);
        \draw (0,0) -- (2,1) -- (4,0);
        \node [below] at (0,0) {$\vec{v}_1$};
        \node [below] at (4,0) {$\vec{v}_2$};
        \node [above] at (2,3) {$c_1\vec{v}_1 + \mu c_1 \vec{v}_2$};
        \node [above] at (2,1) {$\vec{v}_1 + \mu \vec{v}_2$};
        \draw (6,0) -- (8,3) -- (10,0) -- (6,0);
        \draw (8,3) -- (8,1) -- (10,0);
        \node [below] at (6,0) {$\vec{v}_1$};
        \node [below] at (10,0) {$\vec{v}_2$};
        \node [above] at (8,3) {$c\vec{v}_1 + u\vec{v}_2$};
        \node [left] at (8,1) {$c\vec{v}_1$};
        \draw[fill] (13,0) circle [radius=0.1];
        \draw[fill] (13,3) circle [radius=0.1];
        \draw[fill] (14,1.5) circle [radius=0.1];
        \node [right] at (13,0) {$\vec{v}$};
        \node [right] at (13,3) {$\vec{v}+\vec{w}$};
        \node [right] at (14,1.5) {$\lambda \vec{v}$};
    \end{tikzpicture}
    \caption{Decomposing triangles/suspends into sums of D-triangles, D-suspends, and multi-suspends. To avoid clutter, we write $\vec{v}$ instead of $[\vec{v}]$.}
    \label{fig:pigenerators}
\end{figure}

\begin{proof}
    Lemma \ref{absgenerators} says that $\widetilde{\h}_{n-1}(\B_{n,m}^\U(\F))$ is generated by maps $f_1* \dots * f_k$ where each $f_i$ is either a triangle or a suspend. To express this in terms of our new generators, it is enough to show how to write triangles and suspends as sums of D-triangles, D-suspends, and multi-suspends.

    We start with suspends. Consider a suspend $\overline{\llb [\vec{v}], [\lambda \vec{v} + \vec{w}] \rrb}$. If $\lambda \in \U$, then our suspend is already a D-suspend, since we have $\overline{\llb [\vec{v}], [\lambda \vec{v} + \vec{w}] \rrb}$ = $\overline{\llb [\vec{v'}], [\vec{v'} + \vec{w}] \rrb}$, where $\vec{v'} = \lambda \vec{v}$.
    If $\lambda \in \F^\times \setminus \U$, then as in Figure \ref{fig:pigenerators}, we can write
    \begin{align*}
        \overline{\llb [\vec{v}], [\lambda \vec{v} + \vec{w}] \rrb} = \overline{\llb [\vec{v}], [\lambda \vec{v}] \rrb} + \overline{\llb [\lambda \vec{v}], [\lambda \vec{v} + \vec{w}]\rrb}
    \end{align*}
    This is the sum of a multi-suspend and a D-suspend.

    Now we turn to triangles. Consider a triangle $\overline{\llb [\vec{v}_0], [\vec{v}_1], [\vec{v}_2] \rrb}$. By definition, $\{[\vec{v}_0], [\vec{v}_1], [\vec{v}_2]\}$ is an internally additive simplex of $\BA_{n,m}^\U(\F)$. We thus have $\vec{v}_0 = \lambda \vec{v}_1 + \nu \vec{v}_2$ for some $\lambda, \nu \in \F^\times$.
    We now have three cases:
    \begin{itemize}
        \item If $\lambda, \nu \in \U$, then our triangle is already a D-triangle. 

 \item Suppose now that one of $\lambda, \nu$ is in $\U$ and the other is in $\F^\times \setminus \U$. Swapping them if necessary, we can assume that $\lambda = c \in \F^\times \setminus \U$ and $\nu = u \in \U$. As in Figure \ref{fig:pigenerators}, we can write
        \begin{align*}
           \overline{\llb [\lambda \vec{v}_1 + \nu \vec{v}_2], [\vec{v}_1], [\vec{v}_2]\rrb} & =   \overline{\llb [c\vec{v}_1 + u\vec{v}_2], [\vec{v}_1], [\vec{v}_2]\rrb} \\
           & = \overline{\llb [c\vec{v}_1 + u\vec{v}_2], [c\vec{v}_1], [\vec{v}_2] \rrb} - \overline{\llb [\vec{v}_1], [c\vec{v}_1] \rrb} * \overline{[\vec{v}_2], [u\vec{v}_2 + c\vec{v}_1]} 
        \end{align*}
        This is the sum of a D-triangle and the join of a multi-suspend and a D-suspend.

        \item Suppose $\lambda, \nu \in \F^\times \setminus \U$. Write $\lambda = c_1, \nu = c_2$, and let $c_2 = \mu c_1$. Then as in Figure \ref{fig:pigenerators} we can write
        \begin{IEEEeqnarray*}{l}
            \hspace*{-2em}\overline{\llb [\lambda \vec{v}_1 + \nu \vec{v}_2], [\vec{v}_1], [\vec{v}_2] \rrb} \\
            \quad \text{}= \overline{\llb [c_1 \vec{v}_1 + \mu c_1 \vec{v}_2], [\vec{v}_1], [\vec{v}_2] \rrb} \\
            \quad \text{} = \overline{\llb [\vec{v}_1 + \mu\vec{v}_2], [\vec{v}_1], [\vec{v}_2]\rrb} + \overline{\llb [\vec{v}_1 + \mu\vec{v}_2], [c_1(\vec{v}_1 + \mu\vec{v}_2)]\rrb} * \overline{\llb[\vec{v}_2], [\vec{v}_1]\rrb} \\
            \quad \text{} = \overline{\llb [\vec{v}_1 + \mu\vec{v}_2], [\vec{v}_1], [\vec{v}_2]\rrb} + \overline{\llb [\vec{v}_1 + \mu\vec{v}_2], [c_1(\vec{v}_1 + \mu\vec{v}_2)]\rrb} * \overline{\llb[\vec{v}_2], [-\mu\vec{v}_2 + (\vec{v}_1 + \mu\vec{v}_2)]\rrb}
        \end{IEEEeqnarray*}

This is the sum of a triangle, and the join of a multi-suspend with a suspend. Note that the triangle falls into one of the previous two cases we have already considered. Since we have dealt with those cases, and also with the case of a suspend, it follows that we can write the above sum entirely in terms of D-triangles, D-suspends, and multi-suspends.
       
    \end{itemize}

 All that remains to prove is the final claim of the lemma: if $m=0$, then in our generators we can require atleast one of the $f_i$ to either be a D-triangle or a D-suspend. For this, observe that the condition $m=0$ ensures that in the generators $f_1* \dots * f_k$ given by Lemma \ref{absgenerators}, the term $f_1$ must be a triangle (there is no way to choose a $\vec{w}$ as in the definition of a suspend for it). When we expand out the triangle $f_1$ as above, every term that appears has either a D-triangle or a D-suspend. The lemma follows.
\end{proof}

Using Lemma \ref{absgenvariant} we can prove Lemma \ref{pigenerators}. The proof proceeds exactly the same way as \cite[Proof of Lemma 2.53]{miller2021top}, so we omit it here.

\subsubsection{Preliminary Results} 
\label{sec:F/UImplications}

Before proving Lemma \ref{killtriangle} and \ref{killsuspend}, we first deduce some structural consequences on $\F$ from the condition of $\F^\times / \U \cong \{\pm 1\}$ from Proposition \ref{primeconditions} that we will need in our proofs. Even though $\U$ is a multiplicative subgroup of $\F^\times$, the condition that $\F^\times /\U \cong \{\pm 1\}$ is very restrictive, and has implications for the additive structure of $\F$ with respect to $\U$ as well. 

\begin{lemma}
    \label{lem:squares&nonsquares}
    Let $\F$ be a finite field, and $\U \subset \F^\times$ a multiplicative subgroup such that the quotient group $\F/\U \cong \{\pm 1\}$ is of order 2. Then $\U = \{a^2 | a \in \F^\times\}$.
\end{lemma}

\begin{proof}
    The condition that $\F/\U \cong \{\pm 1\}$ implies that the squares of all elements in $\F^\times$ lie in $\U$. From the equation $a^2-b^2 = (a-b)(a+b)$, we see that $b^2 = a^2$ if and only if $b = \pm a$. This implies that the set $\{ a^2 | a \in \F^\times\}$ has size atleast $\frac{|\F^\times|}{2}$. On the other hand, $\F^\times/\U \cong \{ \pm 1\}$ implies that $|\U| = \frac{|\F^\times|}{2}$. Thus we get $\U = \{ a^2 | a \in \F^\times\}$.
\end{proof}

\begin{lemma}
    \label{lem:char2}
Let $\F$ be a finite field, and let $\U \subset \F^\times$ be a multiplicative subgroup such that the quotient group $\F/\U \cong \{\pm 1\}$ has order 2. Then $\mathrm{char}(\F) \neq 2$.
\end{lemma}
\begin{proof}
    If $\F$ had characteristic 2, then the equation $a^2 - b^2 = (a-b)^2$ would hold in $\F$, which would imply that $a^2 = b^2 \iff a=b$. This would then imply that $\F^\times = \{ a^2 | a \in \F^\times\}$, which would contradict Lemma \ref{lem:squares&nonsquares}.
\end{proof}

\begin{lemma}
\label{lem:additivegeneration}
    Let $\F$ be a finite field, and let $\U \subset \F^\times$ be a multiplicative subgroup such that the quotient group $\F/\U \cong \{\pm 1\}$ has order 2. Then $\F$ is additively generated by $\U$, i.e. every element of $\F$ can be written as the sum of a finite number of elements of $\U$.
\end{lemma}
\begin{proof}
The lemma can be rephrased as follows: construct a graph $G$ whose vertices correspond to elements of $\F$, and draw an edge between two vertices if their corresponding elements in $\F$ differ by an element of $\U$. Then we want to show that $G$ is connected. 
Let $|\U| = l$, Then $\F^\times /\U \cong \{\pm 1\}$ implies that $|\F^\times| = 2l$, and so $|\F| = 2l+1$. Note that $G$ is an $l$-regular graph, since for a vertex corresponding to an element $\lambda \in \F$, its neighbors are precisely the $l$ distinct vertices corresponding to $\lambda + u$ as $u$ varies in $\U$. So each connected component of $G$ must have at least $l+1$ vertices, as it must contain all the neighbors of a given vertex in the component. If $G$ had $\geq 2$ connected components, this would imply that $G$ had at least $2(l+1) = 2l+2$ vertices. But $G$ has exactly $|\F| = 2l+1$ vertices. So $G$ must be connected.
\end{proof}

\begin{figure}
    \centering
    \begin{tikzpicture}[scale=1.5]

        \draw[thick] (-2.9,0) -- (2.9,0);
        \draw (-2.03,0.5) -- (2.03,0.5); \draw (-2.03,-0.5) -- (2.03,-0.5);
        \draw (-1.74,1) -- (1.74,1); \draw (-1.74,-1) -- (1.74,-1);
        \draw (-1.45,1.5) -- (1.45,1.5); \draw (-1.45,-1.5) -- (1.45,-1.5);
        \draw (-1.16,2) -- (1.16,2); \draw (-1.16,-2) -- (1.16,-2);
        
        \draw[thick] (-1.45,-2.5) -- (-1.16, -2) -- (-0.87,-1.5) -- (-0.58,-1)--(-0.29,-0.5)--(0,0) -- (0.29,0.5) -- (0.58, 1) -- (0.87, 1.5) -- (1.16, 2) -- (1.45, 2.5);
       
        \draw (-0.58, -2) -- (-0.29,-1.5) -- (0,-1)--(0.29,-0.5)--(0.58,0) -- (0.87,0.5) -- (1.16, 1) -- (1.45, 1.5);

        \draw (0, -2) -- (0.29,-1.5) -- (0.58,-1)--(0.87,-0.5)--(1.16,0) -- (1.45,0.5) -- (1.74, 1) ;

        \draw (0.58, -2) -- (0.87,-1.5) -- (1.16,-1)--(1.45,-0.5)--(1.74,0) -- (2.03,0.5) ;

        \draw (1.16, -2) -- (1.45,-1.5) -- (1.74,-1)--(2.03,-0.5)--(2.32,0) ;

        \draw (-1.45,-1.5) -- (-1.16,-1)--(-0.87,-0.5)--(-0.58,0) -- (-0.29,0.5) -- (0, 1) -- (0.29, 1.5) -- (0.58, 2) ;

        \draw (-1.74,-1)--(-1.45,-0.5)--(-1.16,0) -- (-0.87,0.5) -- (-0.58, 1) -- (-0.29, 1.5) -- (0, 2) ;

        \draw (-2.03,-0.5)--(-1.74,0) -- (-1.45,0.5) -- (-1.16, 1) -- (-0.87, 1.5) -- (-0.58, 2) ;

        \draw (-2.32,0) -- (-2.03,0.5) -- (-1.74, 1) -- (-1.45, 1.5) -- (-1.16, 2);

        \draw[thick] (-1.45, 2.5) -- (-1.16, 2) -- (-0.87, 1.5) -- (-0.58, 1)--(-0.29, 0.5) -- (0,0) -- (0.29,-0.5) -- (0.58, -1) -- (0.87, -1.5) -- (1.16, -2) -- (1.45, -2.5);
    
        \draw  (-0.58, 2) -- (-0.29, 1.5) -- (0, 1)--(0.29, 0.5) -- (0.58,0) -- (0.87,-0.5) -- (1.16, -1) -- (1.45, -1.5) ;

        \draw  (0, 2) -- (0.29, 1.5) -- (0.58, 1)--(0.87, 0.5)--(1.16, 0) -- (1.45, -0.5) -- (1.74, -1) ;

        \draw (0.58, 2) -- (0.87, 1.5) -- (1.16, 1)--(1.45, 0.5)--(1.74,0) -- (2.03, -0.5) ;

        \draw (1.16,  2) -- (1.45, 1.5) -- (1.74, 1)--(2.03, 0.5)--(2.32,0) ;

        \draw (-1.45, 1.5) -- (-1.16, 1)--(-0.87, 0.5)--(-0.58,0) -- (-0.29,-0.5) -- (0, -1) -- (0.29, -1.5) -- (0.58, -2);

        \draw (-1.74, 1)--(-1.45, 0.5)--(-1.16,0) -- (-0.87,-0.5) -- (-0.58, -1) -- (-0.29, -1.5) -- (0, -2) ;

        \draw (-2.03, 0.5)--(-1.74,0) -- (-1.45, -0.5) -- (-1.16, -1) -- (-0.87, -1.5) -- (-0.58, -2) ;

        \draw (-2.32,0) -- (-2.03, -0.5) -- (-1.74, -1) -- (-1.45, -1.5) -- (-1.16, -2);



    \draw[fill] (0,0) circle [radius=0.07]; 
    
    \draw[fill=red,red] (-0.58,2) circle [radius=0.065]; \node[above] at (-0.58,2) {$\textcolor{red}{4\w+1}$}; 

    \draw[fill=red,red] (-2.03,0.5) circle [radius=0.065];
    \draw[fill=red,red] (-1.45,- 1.5) circle [radius=0.065];
    \draw[fill=red,red] (0.58, -2) circle [radius=0.065];
    \draw[fill=red,red] (2.03, -0.5) circle [radius=0.065];
    \draw[fill=red,red] (1.45, 1.5) circle [radius=0.065];


    \draw[fill=blue,blue] (-0.29, 0.5) circle [radius=0.05]; 
    \node[above right] at (-0.29, 0.5) {$\textcolor{blue}{\w}$};

    \draw[fill=blue,blue] (-0.58, 0) circle [radius=0.05]; 
    \draw[fill=blue,blue] (0.58, 0) circle [radius=0.05]; 
    \draw[fill=blue,blue] (0.29, -0.5) circle [radius=0.05]; 
    \draw[fill=blue,blue] (-0.29, -0.5) circle [radius=0.05]; 
    \node[below left] at (-0.29, -0.5) {$\textcolor{blue}{\w^2}$};
    \draw[fill=blue,blue] (0.29, 0.5) circle [radius=0.05]; 

    \draw[fill=blue,blue] (-1.16, 2) circle [radius=0.05];
    \draw[fill=blue,blue] (0, 2) circle [radius=0.05];
    \draw[fill=blue,blue] (-0.29, 1.5) circle [radius=0.05];
    \draw[fill=blue,blue] (-0.87, 1.5) circle [radius=0.05];

    \draw[fill=blue,blue] (-1.45, 0.5) circle [radius=0.05]; 
    \draw[fill=blue,blue] (-1.74, 0) circle [radius=0.05]; 
    \draw[fill=blue,blue] (-2.32, 0) circle [radius=0.05]; 
    \draw[fill=blue,blue] (-1.74, 1) circle [radius=0.05]; 

    \draw[fill=blue,blue] (-0.87, -1.5) circle [radius=0.05]; 
    \draw[fill=blue,blue] (-1.16, -2) circle [radius=0.05]; 
    \draw[fill=blue,blue] (-1.16, -1) circle [radius=0.05]; 
    \draw[fill=blue,blue] (-1.74, -1) circle [radius=0.05]; 

    \draw[fill=blue,blue] (0, -2) circle [radius=0.05]; 
    \draw[fill=blue,blue] (1.16, -2) circle [radius=0.05]; 
    \draw[fill=blue,blue] (0.87, -1.5) circle [radius=0.05]; 
    \draw[fill=blue,blue] (0.29, -1.5) circle [radius=0.05]; 

    \draw[fill=blue,blue] (1.74, -1) circle [radius=0.05]; 
    \draw[fill=blue,blue] (2.32, 0) circle [radius=0.05]; 
    \draw[fill=blue,blue] (1.45, -0.5) circle [radius=0.05]; 
    \draw[fill=blue,blue] (1.74, 0) circle [radius=0.05]; 

    \draw[fill=blue,blue] (1.74, 1) circle [radius=0.05]; 
    \draw[fill=blue,blue] (1.16, 1) circle [radius=0.05]; 
    \draw[fill=blue,blue] (0.87, 1.5) circle [radius=0.05]; 
    \draw[fill=blue,blue] (1.16, 2) circle [radius=0.05]; 


    \draw[fill=orange,orange] (2.03, 0.5) circle [radius=0.08]; \draw[fill=orange,orange] (1.45, 0.5) circle [radius=0.08]; \draw[fill=orange,orange] (0.87, 0.5) circle [radius=0.08]; \draw[fill=orange,orange] (0.58, 1) circle [radius=0.08]; \draw[fill=orange,orange] (0, 1) circle [radius=0.08]; \draw[fill=orange,orange] (-0.58, 1) circle [radius=0.08]; \draw[fill=orange,orange] (-1.16, 1) circle [radius=0.08]; \draw[fill=orange,orange] (-1.45, 1.5) circle [radius=0.08]; 
    \draw[fill=orange,orange] (0.29, 1.5) circle [radius=0.08]; \draw[fill=orange,orange] (0.58, 2) circle [radius=0.08]; 
    \draw[fill=orange,orange] (-0.87, 0.5) circle [radius=0.08]; \draw[fill=orange,orange] (-1.16, 0) circle [radius=0.08]; \draw[fill=orange,orange] (-1.45, -0.5) circle [radius=0.08];
    \draw[fill=orange,orange] (-2.03, -0.5) circle [radius=0.08]; \draw[fill=orange,orange] (-0.87, -0.5) circle [radius=0.08];
    \draw[fill=orange,orange] (-0.58, -1) circle [radius=0.08]; \draw[fill=orange,orange] (0, -1) circle [radius=0.08]; \draw[fill=orange,orange] (0.58, -1) circle [radius=0.08]; \draw[fill=orange,orange] (1.16, -1) circle [radius=0.08];
    \draw[fill=orange,orange] (-0.58, -2) circle [radius=0.08];
    \draw[fill=orange,orange] (-0.29, -1.5) circle [radius=0.08];
    \draw[fill=orange,orange] (1.45, -1.5) circle [radius=0.08];
    \draw[fill=orange,orange] (0.87, -0.5) circle [radius=0.08];
    \draw[fill=orange,orange] (1.16, 0) circle [radius=0.08];

    \draw[very thick, orange] (1.45, -1.5) -- (1.16, -1) -- (0.58, -1) -- (0, -1) -- (-0.58, -1);
    \draw[very thick, orange] (-2.03, -0.5) -- (-1.45, -0.5) -- (-0.87, -0.5);
    \draw[very thick, orange] (-1.16, 1) -- (-0.58, 1) -- (0, 1) -- (0.58, 1);
    \draw[very thick, orange] (0.87, 0.5) -- (1.45, 0.5) -- (2.03, 0.5);

    \draw[very thick, orange] (-1.45, 1.5) -- (-0.87, 0.5);
    \draw[very thick, orange] (-1.45, -0.5) -- (-0.58, 1); 

    \draw[very thick, orange] (-1.16, 0) -- (-0.87, -0.5) -- (-0.58, -1) -- (-0.29, -1.5) -- (-0.58, -2) -- (0, -1);

    \draw[very thick, orange] (1.16, -1) -- (0.87, -0.5) -- (0.58, -1) -- (1.45, 0.5);
    \draw[very thick, orange] (1.16, 0) -- (0.29, 1.5);
    \draw[very thick, orange] (0.58, 2) -- (0, 1);
    
    \end{tikzpicture}
    \caption{An illustration of Lemma \ref{lem:nonsquaresconnected} when $\F = \Z[\w]/(4\w+1)$ and $\U$ is the image of $\{\pm1, \pm \w, \pm \w^2\}$ in $\F$. \textcolor{red}{Red} points indicate multiples of $4\w+1$ by units in $\{\pm1, \pm\w, \pm\w^2\}$, and \textcolor{blue}{blue} points indicate elements that are $\pm1, \pm \w$ or $\pm\w^2$ mod $4\w+1$. \textcolor{orange}{Orange} points indicate elements that map to $\F^\times \setminus \U$, with orange edges indicating those pairs that differ by an element of $\U$.}
    \label{fig:eisenstein}
\end{figure}
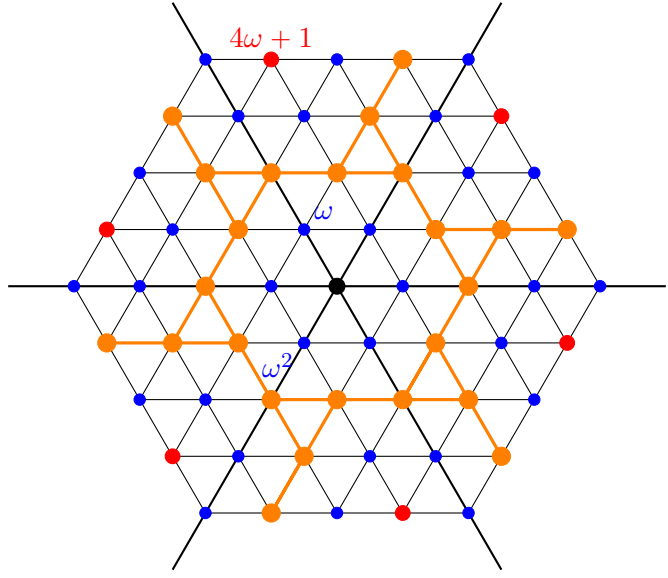

\begin{lemma}
    \label{lem:nonsquaresconnected}
    Let $\F$ be a finite field and $\U \subset \F^\times$ a multiplicative subgroup such that $-1 \in \U$, and the quotient group $\F/\U \cong \{\pm 1\}$ has order 2. Then for any two elements $\lambda, \nu \in \F^\times \setminus \U$, there exist $u_1, \dots, u_k \in \U$ such that $\lambda = \nu + u_1 + \dots + u_k$ and so that $\nu + u_1 + \dots + u_i \in \F^\times \setminus \U$ for all $1 \leq i \leq k$.
\end{lemma}
\begin{proof}
    Similar to Lemma \ref{lem:additivegeneration}, we can rephrase the desired result as follows: construct a graph $G$ whose vertices correspond to the elements of $\F^\times \setminus \U$. Connect two vertices of $G$ by an edge if the corresponding elements of $\F^\times \setminus \U$ differ by an element of $\U$. Then we want to show that $G$ is connected. See Figures \ref{fig:eisenstein} and \ref{fig:gaussian} for an illustration of this property.

By Lemma \ref{lem:char2}, we have $1 \neq -1$ in $\F$, and since $\{\pm 1\}$ is a multiplicative subgroup of $\U$, we have $|\U| = 2k$ for some $k$. The condition that $\F^\times / \U \cong \{\pm 1\}$ implies that $|\F^\times| = 4k$, and thus $|\F^\times \setminus \U| = 2k$. We will show that $G$ is a $k$-regular graph.

Let $\lambda \in \F^\times \setminus \U$. We want to count the cardinality of the set $N_\lambda = \{ u | u \in \U, \lambda + u \in \F^\times \setminus \U\}$. Since $\F^\times /\U \cong \{\pm 1\}$, a given $u \in \U$ can be written as $u = \lambda\nu$ for a unique $\nu \in \F^\times \setminus \U$. Thus $\lambda + u = \lambda(1+\nu)$. So the condition that $\lambda + u \in \F^\times \setminus \U$, along with the fact that $\lambda \in \F^\times \setminus \U$, implies that we must have $1 + \nu \in \U$.
Thus we want to count the set $N = \{ \nu | \nu \in \F^\times \setminus \U, 1 + \nu \in \U\}$. Note that this description is independent of $\lambda$, which shows that $G$ is a regular graph.

Let $\chi$ be the group homomorphism given by the composition $\F^\times \to \F^\times /\U \xrightarrow{\cong} \{\pm 1\}$. Thus $\chi(U) = \{1\}$ and $\chi(\F^\times \setminus \U) = \{-1\}$. Thus we have
\begin{align*}
    1 + \chi(a) = \begin{cases}
        & 2 \quad \text{ if }a \in \U \\
        & 0 \quad \text{ otherwise}
    \end{cases}
\end{align*}
Similarly,
\begin{align*}
    1 - \chi(a) = \begin{cases}
        & 2 \quad \text{ if }a \in \F^\times \setminus \U \\
        & 0 \quad \text{ otherwise}
    \end{cases}
\end{align*}
Thus the cardinality of $N$ is the same as the sum
\begin{align*}
    \sum_{\nu \neq 0,-1} \left( \frac{1-\chi(\nu)}{2} \right) \cdot \left( \frac{1+\chi(1+\nu)}{2} \right)
\end{align*}
which is equal to
\begin{align}
\label{eqn:nonsqsum}
    \sum_{\nu \neq 0, -1} \left( \frac{1 - \chi(\nu) + \chi(1+\nu) - \chi(\nu)\chi(1+\nu)}{4} \right)
\end{align}

Now, since $\chi(a) = 1 \iff a \in \U$ and $\chi(a) = -1$ otherwise, and because $|\U| = \frac{|\F^\times|}{2}$, we have
\begin{align*}
    \sum_{a \in \F^\times} \chi(a) = 0
\end{align*}

Thus we have:
\begin{align*}
    \sum_{\nu \neq 0, -1} \chi(\nu) & = \sum_{a \in \F^\times} \chi(a) - \chi(-1) \\
    & = 0 - 1 = -1,
\end{align*}

\begin{align*}
    \sum_{\nu \neq 0, -1} \chi(1+\nu) & = \sum_{a \neq 1, 0} \chi(a) \\
    & = \sum_{a \in \F^\times} \chi(a) - \chi(1) \\
    & = 0 - 1 = -1,
\end{align*}
and
\begin{align*}
    \sum_{\nu \neq 0, -1} \chi(\nu)\chi(1+\nu) & = \sum_{\nu \neq 0, -1} \chi(\nu)\chi(\nu)\chi(\nu^{-1}+1) \\
    & = \sum_{\nu \neq 0, -1} \chi(\nu^{-1}+1) \\
    & = \sum_{a \neq 1, 0} \chi(a) \\
    & = 0 - 1 = -1
\end{align*}

Plugging all these values into \eqref{eqn:nonsqsum}, we get that the degree of each vertex in $G$ is:
\begin{align*}
     & = \sum_{\nu \neq 0, -1} \left( \frac{1 - \chi(\nu) + \chi(1+\nu) - \chi(\nu)\chi(1+\nu)}{4} \right) \\
     & = \frac{(|\F|-2)-(-1) + (-1) - (-1)}{4} \\
     & = \frac{(4k-1) - (-1)}{4} = \frac{4k}{4} = k
\end{align*}

Thus $G$ is a $k$-regular graph with $|\F^\times \setminus \U| = 2k$ vertices. $k$-regularity implies that any connected component of $G$ must contain at least $k+1$ vertices. So if $G$ had $\geq 2$ components, it would have $\geq 2(k+1)$ vertices, a contradiction. Thus $G$ must be connected.
\end{proof}

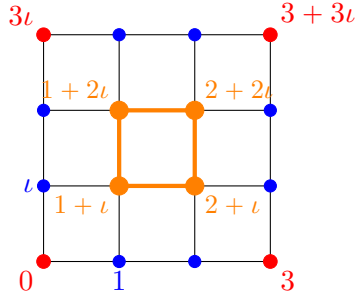
\begin{figure}
    \centering
    \begin{tikzpicture}

    \draw (0,0) -- (3,0); \draw (0, 1) -- (3,1); \draw (0,2) -- (3,2); \draw (0,3) -- (3,3);

    \draw (0,0) -- (0, 3); \draw (1,0) -- (1,3); \draw (2,0) -- (2,3); \draw (3,0) -- (3,3);

    \draw[fill=red, red] (0,0) circle [radius=0.09]; \draw[fill=red, red] (3,0) circle [radius=0.09]; \draw[fill=red, red] (0,3) circle [radius=0.09]; \draw[fill=red, red] (3,3) circle [radius=0.09];

    \draw[fill=blue, blue] (0,1) circle [radius=0.08]; \draw[fill=blue, blue] (0,2) circle [radius=0.08]; \draw[fill=blue, blue] (1,3) circle [radius=0.08]; \draw[fill=blue, blue] (2,3) circle [radius=0.08]; \draw[fill=blue, blue] (3,1) circle [radius=0.08]; \draw[fill=blue, blue] (3,2) circle [radius=0.08]; \draw[fill=blue, blue] (1,0) circle [radius=0.08]; \draw[fill=blue, blue] (2,0) circle [radius=0.08];

    \draw[fill=orange, orange] (1,1) circle [radius=0.12]; \draw[fill=orange, orange] (1,2) circle [radius=0.12]; \draw[fill=orange, orange] (2,1) circle [radius=0.12]; \draw[fill=orange, orange] (2,2) circle [radius=0.12];

    \draw[ultra thick, orange] (1,1) -- (1,2) -- (2,2) -- (2,1) -- (1,1);

    \node[below] at (1,0) {$\textcolor{blue}{1}$};
    \node[left] at (0,1) {$\textcolor{blue}{\io}$};

    \node[font=\small][below left] at (1,1) {$\textcolor{orange}{1+\io}$};
    \node[font=\small][above left] at (1,2) {$\textcolor{orange}{1+2\io}$};
    \node[font=\small][above right] at (2,2) {$\textcolor{orange}{2+2\io}$};
    \node[font=\small][below right] at (2,1) {$\textcolor{orange}{2+\io}$};

    \node[below right] at (3,0) {$\textcolor{red}{3}$};
    \node[below left] at (0,0) {$\textcolor{red}{0}$};
    \node[above right] at (3,3) {$\textcolor{red}{3+3\io}$};
    \node[above left] at (0,3) {$\textcolor{red}{3\io}$};
        
    \end{tikzpicture}
    \caption{An illustration of Lemma \ref{lem:nonsquaresconnected} when $\F = \Z[\io]/(3)$ and $\U$ is the image of $\{\pm1, \pm \io\}$ in $\F$. \textcolor{red}{Red} points indicate multiples of $3$ and \textcolor{blue}{blue} points indicate elements that are $\pm1$ or $\pm\io$ mod $3$. \textcolor{orange}{Orange} points indicate elements that map to $\F^\times \setminus \U$, with orange edges indicating those pairs that differ by an element of $\U$.}
    \label{fig:gaussian}
\end{figure}

\begin{lemma}
    \label{lem:sumofunits}
    Let $\F$ be a finite field, and let $\U \subset \F^\times$ be a multiplicative subgroup such that the quotient group $\F/\U \cong \{\pm 1\}$ has order 2. Then every element $c \in \F^\times \setminus \U$ can be written as $c = u_1 + u_2$ for some $u_1, u_2 \in \U$.
\end{lemma}
\begin{proof}
Suppose first that there is no element of $\F^\times \setminus \U$ that can written as $u_1 + u_2$ for some $u_1, u_2 \in \U$. For two subsets $S, T \subset \F$, let $S + T \coloneq \{s+t | s \in S, t \in T\}$. Let $V \coloneq \U \cup \{0\}$. Our assumption implies that $V + V \subset V$. We can use this to inductively show that 
\begin{align*}
   \LaTeXunderbrace{V + \dots + V}_{n \text{ times}} = V + \LaTeXunderbrace{V + \dots + V}_{n-1 \text{ times}} \subset V + V \subset V 
\end{align*}
for all $n \geq 1$. But this contradicts Lemma \ref{lem:additivegeneration}.

Thus there exists some $c \in \F^\times \setminus \U$ such that $c = u_1 + u_2$ for some $u_1, u_2 \in \U$. Now, the condition that $\F^\times/\U \cong \{\pm 1\}$ implies that any $d \in \F^\times \setminus \U$ can be written uniquely as $d = uc$ for some $u \in \U$, which gives us $d = uu_1 + uu_2$. Thus every element of $\F^\times \setminus \U$ can be written as a sum of two elements of $\U$.
\end{proof}

\subsubsection{The Retraction}\label{retraction}
For later use, we will extend the retraction $\rho: \B_n^\U(\F) \to \BD_n^{\U}(\F)$ from Lemma \ref{BDretract} to the following larger complex.

\begin{definition}
    Let $\BAO_n^\U(\F)$ denote the subcomplex of $\BA_n^\U(\F)$ whose simplices consist of the simplices of $\BDA_n^\U(\F)$ along with all the standard simplices of $\B_n^\U(\F)$.
\end{definition}

\begin{lemma}
    \label{BDAretract}
    There is a retraction $\rho: \BAO_n^\U(\F) \to \BDA_n^\U(\F)$ that extends the retraction from $\B_n^\U(\F) \to \BD_n^{\U}(\F)$.
\end{lemma}
\begin{proof}
    The only simplices of $\BAO_n^\U(\F)$ that are not in $\BDA_n^\U(\F)$ are of the form $\sigma = \{ [\vec{v}_1], \dots, [\vec{v_n}]\}$ with $\det(\vec{v}_1, \dots, \vec{v}_n) \in \F^\times \setminus \U$. As before, we let $S(\sigma)$ denote the result of subdividing $\sigma$ with a new vertex $x_{\sigma}$, and $\rho$ is defined by setting $\rho(x_{\sigma}) = [\vec{w}]$ and extending linearly, where $\vec{w} \in \F^n$ is chosen so that 
\begin{align*}
    \det(\vec{v_1}, \dots, \widehat{\vec{v}}_i, \dots, \vec{v}_n, \vec{w}) \in \U \hspace{1cm} 1 \leq i \leq n 
\end{align*}
\end{proof}

Note that in the above proof, the possible choices for $\vec{w}$ are of the form 

\begin{align*}
    \vec{w} = c_1 \vec{v}_1 + \dots + c_n \vec{v}_n
\end{align*}
where the $c_i \in \F^\times \setminus \U$ are such that $c_i \cdot \det(\vec{v_1}, \dots, \widehat{\vec{v}}_i, \dots, \vec{v}_n) \in \U$ for all $1 \leq i \leq n$.

Thus defining $\rho$ requires making choices for the $c_i$ for all such simplices $\sigma$.

However, the following lemma implies that all such choices results in homotopic $\rho$:

\begin{lemma}
\label{choicefreedom}
Let $\F$ be a finite field and $\U \subset \F^\times$ a multiplicative subgroup such that $-1 \in \U$, and the quotient $\F^\times / \U$ is isomorphic to the cyclic group $\{\pm 1\}$ of order 2.

    For some $n \geq 2$, let $\vec{v}_1, \dots, \vec{v}_n \in \F^n$ be such that $\det(\vec{v}_1, \dots, \vec{v}_n) \in \F^\times \setminus \U$.
    Let $\vec{w}_1, \vec{w}_2$ be such that 
\begin{align*}
    \det(\vec{v}_1, \dots, \widehat{\vec{v}}_i, \dots, \vec{v}_n, \vec{w}_j) \in \U \hspace{1cm} \text{for all } 1\leq i \leq n, 1 \leq j \leq 2
\end{align*}

Then the maps 
\begin{align*}
    \llb [\vec{w}_1] \rrb * \overline{\llb [\vec{v}_1], \dots, [\vec{v}_n] \rrb} : \Delta^0 * \del \Delta^{n-1} \to \BDA_n^\U(\F)
\end{align*}
and 
    \begin{align*}
    \llb [\vec{w}_2] \rrb * \overline{\llb [\vec{v}_1], \dots, [\vec{v}_n] \rrb} : \Delta^0 * \del \Delta^{n-1} \to \BDA_n^\U(\F)
\end{align*}

are homotopic rel $\del(\Delta^0 * \del \Delta^{n-1}) \cong \del \Delta^{n-1}$.
\end{lemma}

\begin{figure}
\centering
\begin{tikzpicture}
    \draw [thick] (2,0) -- (0,2) -- (2,4) -- (4,2) -- (2,0) -- (2.5,1.5) -- (2,4);
        \draw [thick] (0,2) -- (2.5, 1.5) -- (4,2);
        \draw [dotted, thick] (0,2) -- (4,2) ;
        \draw[dashed, thick] (0,2) -- (1.5, 2.5) -- (4,2);
        \draw[dashed] (1.5,2.5) -- (2.5,1.5);
        \draw[dashed] (1.5,2.5) -- (2,4);

    \node[above] at (2,4) {$\vec{v}_4$};
    \node[below] at (2,0) {$\vec{v}_1 + \vec{v}_2 + \vec{v}_3 + \vec{v}_4$};
    \node[left] at (0,2) {$\vec{v}_1 + \vec{v}_2 + \vec{v}_3$};
    \node[right] at (4,2) {$\vec{v}_2$};
    \node[below left] at (2.5,1.5) {$\vec{v}_1$};
    \node[above left] at (1.5,2.5) {$\vec{v}_3$};
        
\end{tikzpicture}
\caption{The disk bounded by the two simplices in the proof of Lemma \ref{choicehelper}, in the case $n=4$.}
\label{fig:bykalternative}
\end{figure}
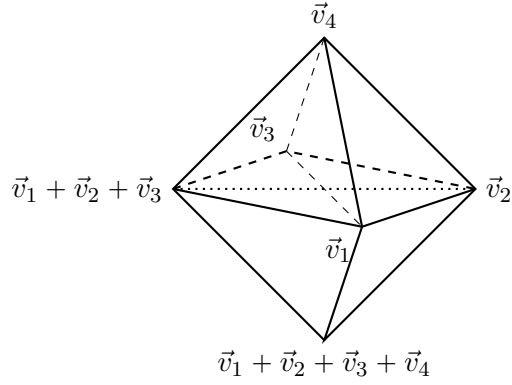

The proof for this will require the following Lemma.
\begin{lemma}
\label{choicehelper}
Let $\F$ be a field, and $\U \subset \F^\times$ a multiplicative subgroup.
    For some $n \geq 2$, let $\{ [\vec{v}_1], \dots, [\vec{v}_n]\}$ be an $(n-1)$-simplex in $\BD_n^\U(\F)$, and let $u_1, \dots, u_n$ be units in $\U$. Then the map
    \begin{align*}
        \overline{\llb [\vec{v}_1], \dots, [\vec{v}_n], [u_1\vec{v}_1 + \dots + u_n\vec{v}_n] \rrb} : \del \Delta^n \cong S^{n-1} \to \BD_n^\U(\F)
    \end{align*}
    is nullhomotopic in $\BDA_n^\U(\F)$.
\end{lemma}
\begin{proof}
Since we have the freedom to multiply the $\vec{v_i}$ with units in $\U$, it is enough to prove this in the case where $u_1 = \dots = u_n = 1$.
We will induct on $n$. Note that if $n=2$, then our map is the boundary of an additive simplex of $\BDA_2^\U(\F)$, and thus nullhomotopic. So assume $n \geq 3$.

Define $\K_n^\U(\F)$ to be the simplicial complex whose vertices are $\U$-vectors of $\F^n$, and simplices are all finite subsets of $\U$-vectors. We will think of $\BDA_n^\U(\F)$ as a subcomplex of $\K_n^\U(\F)$.
Consider the map
\begin{align*}
   f= \overline{\llb [\vec{v}_1], \dots, [\vec{v}_n], [\vec{v}_1 + \dots + \vec{v}_n] \rrb} * \llb [\vec{v}_1 + \dots + \vec{v}_{n-1} \rrb] : \del \Delta^n * \Delta^0 \cong D^n \to \K_n^\U(\F)
\end{align*}
The restriction of $f$ to $\del D^n \cong S^{n-1}$ is precisely the map we are considering, so this gives us a nullhomotopy of our map in the complex $\K_n^\U(\F)$. Note also that $f$ is a bijective map from $\del \Delta^n * \Delta^0 \cong D^n$ onto its image. Note that for $n=3$, the image of $f$ in fact lies in $\BDA_n^\U(\F)$, and so $f$ gives a nullhomotopy of our map.

For $n \geq 4$, the only simplices of $\K_n^\U(\F)$ that are in the image of $f$ but are not in $\BDA_n^\U(\F)$ are
\begin{align*}
    \{ [\vec{v}_1], \dots, [\vec{v}_n], [\vec{v}_1+ \dots + \vec{v}_{n-1}]\} \hspace{0.25cm} \text{ and } \hspace{0.25cm} \{[\vec{v}_1], \dots, [\vec{v}_{n-1}], [\vec{v}_1 + \dots + \vec{v}_n], [\vec{v}_1 + \dots + \vec{v}_{n-1}]\}
\end{align*}
Together, these bound an $n$-dimensional disk $D$ in the image of $f$ (see Figure \ref{fig:bykalternative}), whose boundary is
\begin{align*}
 \overline{\llb [\vec{v}_1], \dots, [\vec{v}_{n-1}], [\vec{v}_1+ \dots + \vec{v}_{n-1}]\rrb} *  \overline{\llb [\vec{v}_n], [\vec{v}_1+ \dots + \vec{v}_n] \rrb}
\end{align*}

This tells us that the map $f|_{D^n \setminus f^{-1}(\text{int}(D))}$ gives us a homotopy between the maps
\begin{align*}
    \overline{\llb [\vec{v}_1], \dots, [\vec{v}_n], [\vec{v}_1 + \dots + \vec{v}_n]\rrb} : \del \Delta^n \cong S^{n-1} \to \BDA_n^\U(\F)
\end{align*}
and
\begin{align*}
    \overline{\llb [\vec{v}_n], [\vec{v}_1+ \dots + \vec{v}_n] \rrb} * \overline{\llb [\vec{v}_1], \dots, [\vec{v}_{n-1}], [\vec{v}_1+ \dots + \vec{v}_{n-1}]\rrb} : \del \Delta^1 * \del \Delta^{n-1} \cong S^{n-1} \to \BDA_n^\U(\F)
\end{align*}
So to finish our proof, it is enough to show that the latter map is nullhomotopic.

We can inductively assume that the map
\begin{align*}
    \overline{\llb [\vec{v}_1], \dots, [\vec{v}_{n-1}], [\vec{v}_1+ \dots + \vec{v}_{n-1}]\rrb} : \del \Delta^{n-1} \cong S^{n-2} \to \BDA_{n-1}^\U(\F)
\end{align*}
is nullhomotopic. Let $H : D^{n-1} \to \BDA_{n-1}^\U(\F)$ be a nullhomotopy of this map. We can assume that $H$ homotopes this map to the constant map from $S^{n-1}$ to a vertex $[\vec{v}]$ of $\BDA_{n-1}^\U(\F)$.
Note that we have a natural inclusion $\io: \BDA_{n-1}^\U(\F) \hookrightarrow \BDA_n^\U(\F)$ induced by the inclusion $\langle \vec{v}_1, \dots , \vec{v}_{n-1}\rangle \hookrightarrow \langle \vec{v}_1, \dots, \vec{v}_n \rangle$.
Then the map
\begin{align*}
    (\io \circ H) * \overline{\llb [\vec{v}_n], [\vec{v}_1+ \dots + \vec{v}_n]\rrb} : D^{n-1} * \del \Delta^1 \cong D^n \to \BDA_n^\U(\F)
\end{align*}
gives us a homotopy from this map to the map
\begin{align*}
    g*\overline{\llb [\vec{v}_n], [\vec{v}_1+ \dots + \vec{v}_n]\rrb} : \del \Delta^{n-1}*\del \Delta^1 \to \BDA_n^\U(\F)
\end{align*}
that maps $\del \Delta^{n-1}$ to the vertex $[\vec{v}]$. Now homotoping the vertices $[\vec{v}_n]$ and $[\vec{v}_1+ \dots + \vec{v}_n]$ to $[\vec{v}]$ along the edges joining them gives us the desired nullhomotopy. 
\end{proof}

\begin{remark}
Miller--Patzt--Putman \cite[Lemma 2.63]{miller2021top} give a simpler proof of this, in the specific case where $\F$ is a quotient of $\Z$ by a prime, and $\U = \{\pm1\}$. Their proof uses the fact that $\BA_n^{\pm}(\Z)$ is $(n-1)$-connected. 
We give an independent proof to avoid having to impose additional assumptions, and to emphasise that this lemma holds for all $\F$ and $\U$.
\end{remark}

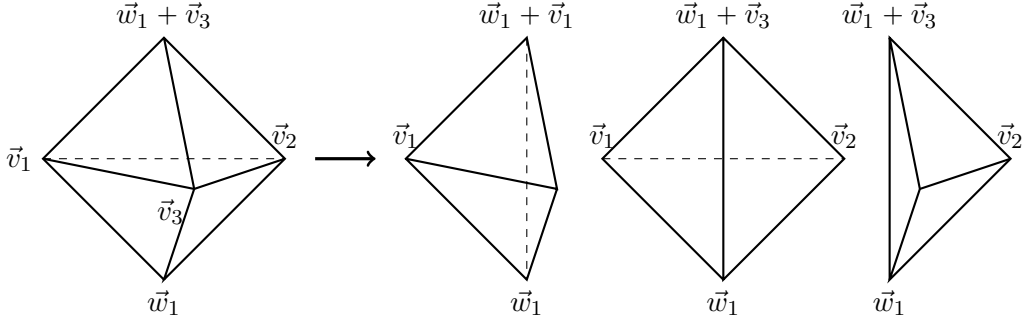
\begin{figure}
\centering
    \begin{tikzpicture}[scale=0.8]
        \draw [thick] (2,0) -- (0,2) -- (2,4) -- (4,2) -- (2,0) -- (2.5,1.5) -- (2,4);
        \draw [thick] (0,2) -- (2.5, 1.5) -- (4,2);
        \draw [dashed] (0,2) -- (4,2) ;

        \node[below] at (2,0) {$\vec{w}_1$};
        \node[left] at (0,2) {$\vec{v}_1$};
        \node[above] at (2,4) {$\vec{w}_1 + \vec{v}_3$};
        \node[above] at (4,2) {$\vec{v}_2$};
        \node[below left] at (2.5,1.5) {$\vec{v}_3$};

        \draw [->, very thick] (4.5,2) -- (5.5,2) ;

        \draw [thick] (8.5,1.5) --(8,0) --(6,2) --(8,4) --(8.5,1.5) --(6,2) ;
        \draw [dashed] (8,4) -- (8,0) ;

        \node[below] at (8,0) {$\vec{w}_1$};
        \node[above] at (6,2) {$\vec{v}_1$};
        \node[above] at (8,4) {$\vec{w}_1 + \vec{v}_1$};

        \draw [thick] (11.25,0) -- (9.25,2) -- (11.25,4) -- (13.25,2) -- (11.25,0) -- (11.25,4);
        \draw [dashed] (9.25,2) -- (13.25,2) ;

        \node[above] at (9.25,2) {$\vec{v}_1$};
        \node[above] at (11.25,4) {$\vec{w}_1 + \vec{v}_3$};
        \node[below] at (11.25,0) {$\vec{w}_1$};
        \node[above] at (13.25,2) {$\vec{v}_2$};

        \draw[thick] (14,0)--(14,4)--(16,2)--(14,0)--(14.5,1.5)--(14,4);
        \draw[thick] (14.5,1.5) -- (16,2);

        \node[above] at (14,4) {$\vec{w}_1 + \vec{v}_3$};
        \node[below] at (14,0) {$\vec{w}_1$};
        \node[above] at (16,2) {$\vec{v}_2$};
        
    \end{tikzpicture}
    \caption{The sphere in the proof of Lemma \ref{choicefreedom} in the case $n=3$, along with the result of breaking it into $n=3$ spheres.}
\label{fig:choicefreedom}
\end{figure}

\begin{proof}[Proof of Lemma \ref{choicefreedom}]
    Write
    \begin{align*}
        \vec{w}_1 = c_1 \vec{v}_1 + \dots + c_n \vec{v}_n ,& &  \vec{w}_2 = d_1 \vec{v}_1 + \dots + d_n \vec{v}_n
    \end{align*}
    with $c_i, d_i \in \F^\times \setminus \U$ for all $1 \leq i \leq n$. It is enough to deal with the case where all but one pair of the $c_i, d_i$ are equal. Thus reordering the $\vec{v}_i$ if necessary, we can assume $c_i = d_i$ for $1 \leq i \leq n-1$ and $c_n \neq d_n$. By Lemma \ref{lem:nonsquaresconnected}, it is further enough to deal with the case where $d_n = c_n + u$ for some $u \in \U$. Thus we have $\vec{w}_2 = \vec{w}_1 + u\vec{v}_n$.

    Our goal is equivalent to proving that 
    \begin{align*}
        \overline{\llb [\vec{w}_1], [\vec{w}_1 + u\vec{v}_n] \rrb} * \overline{\llb [\vec{v}_1], \dots, [\vec{v}_n] \rrb} : \del \Delta^1 * \del \Delta^{n-1} \cong S^{n-1} \to \BD_n^\U(\F)
    \end{align*}
    is nullhomotopic in $\BDA_n^\U(\F)$. As shown in Figure \ref{fig:choicefreedom}, as an element of $\pi_{n-1}(\BDA_n^\U(\F))$, this sphere is the sum of the following $n$ spheres: 

    \begin{align*}
        \overline{\llb [\vec{v}_1], \dots, [\widehat{\vec{v}}_i], \dots, [\vec{v}_n], [\vec{w_1}], [\vec{w}_1 + u \vec{v_n}] \rrb} \hspace{1cm} (1 \leq i \leq n)
    \end{align*}

For $1 \leq i \leq n-1$, these are the boundaries of additive simplices, and thus are trivially nullhomotopic in $\BDA_n^\U(\F)$.
So we need only deal with the $i=n$ case. Let $u' \in \U$ be a unit so that $d_n = u'c_n$, and for each $1 \leq i \leq n-1$, let $c'_i \in \F^\times \setminus \U$ be such that $c_i = u'c'_i$.
We then have
\begin{align*}
    [\vec{w}_1 + u\vec{v}_n] & = [c_1\vec{v}_1 + \dots + c_{n-1}\vec{v}_{n-1} + d_n \vec{v}_n] = [c'_1\vec{v}_1 + \dots + c'_{n-1}\vec{v}_{n-1} + c_n \vec{v}_n] \\
    & = [\vec{w}_1 + (c'_1-c_1)\vec{v_1} + \dots + (c'_{n-1} -c_{n-1})\vec{v}_{n-1}] \\
    & = [\vec{w}_1 + (1-u')c'_1\vec{v_1} + \dots + (1-u')c'_{n-1}\vec{v}_{n-1}]\\
\end{align*}

Note that we have 
\begin{align*}
    d_n = c_n + u = u'c_n
\end{align*}
which implies 
\begin{align*}
    (1-u')c_n = -u \in \U
\end{align*}
Since we are assuming $\F^\times/\U \cong \{\pm 1\}$, this implies that 
\begin{align*}
    1-u' \in \F^\times \setminus \U.
\end{align*}
Thus 
\begin{align*}
    (1-u')c'_i \in \U \hspace{0.5cm} \text{ for all } 1 \leq i \leq n-1
\end{align*}

This implies that the sphere in the $i=n$ case, which is
\begin{IEEEeqnarray*}{l}
    \hspace*{-3em}\overline{\llb [\vec{v}_1], \dots, [{\vec{v}}_{n-1}], [\vec{w_1}], [\vec{w}_1 + u \vec{v_n}] \rrb} \\
    \quad \text{} = \overline{\llb [\vec{v}_1], \dots, [{\vec{v}}_{n-1}], [\vec{w_1}], [\vec{w}_1 + (1-u')c'_1\vec{v_1} + \dots + (1-u')c'_{n-1}\vec{v}_{n-1}] \rrb} 
\end{IEEEeqnarray*}

is nullhomotopic by Lemma \ref{choicehelper}. The result now follows.
\end{proof}

\subsubsection{Killing initial D-triangle maps}\label{triangles}

In this section we will show that the images of initial D-triangle maps in $\BDA_n^\U(\F)$ are nullhomotopic.

\begin{figure}
    \centering
    \includegraphics[width=0.85\linewidth]{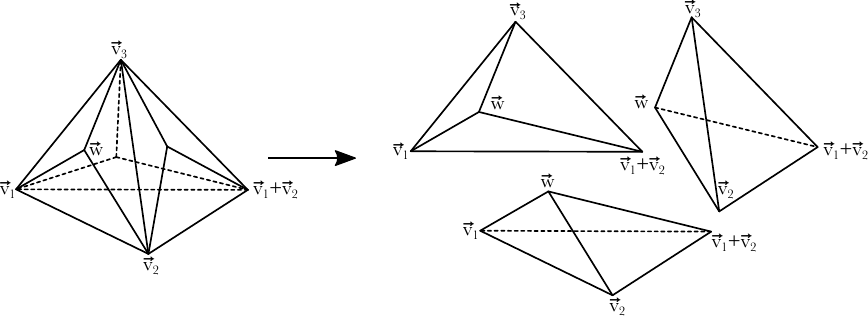}
    \caption{On the left is the sphere $\rho \circ f$ in the proof of Lemma \ref{killtrianglehelper} for $n=3$ with its three subdivided faces. On the right is the $n=3$ spheres it can be cut into (with the required subdivisions omitted to improve readability).}
    \label{fig:killtrianglehelper1}
\end{figure}

To prove Lemma \ref{killtriangle}, we will need the following lemma.

\begin{lemma}
\label{killtrianglehelper}
Let $\F$ be a finite field and $\U \subset \F^\times$ a multiplicative subgroup such that $-1 \in \U$, and the quotient $\F^\times / \U$ is isomorphic to the cyclic group $\{\pm 1\}$ of order 2.

    For some $n \geq 2$, let $\rho: \BAO_n^\U(\F) \to \BDA_n^\U(\F)$ be the retraction constructed in \ref{BDAretract}. Let $\{ \vec{v}_1, \dots, \vec{v}_n\}$ be a basis of $\F^n$. Then 
    \begin{align*}
        \rho \circ \overline{\llb [\vec{v}_1 + \vec{v}_2], [\vec{v}_1], \dots, [\vec{v}_n] \rrb} : \del \Delta^n \to \BDA_n^\U(\F)
    \end{align*}
    is nullhomotopic.
\end{lemma}
\begin{proof}
    Set $f = \overline{\llb [\vec{v}_1 + \vec{v}_2], [\vec{v}_1], \dots, [\vec{v}_n] \rrb}$. If $\det(\vec{v}_1, \dots, \vec{v_n}) \in \U$, then $\rho \circ f = f$ and the image of $f$ is the boundary of an additive simplex in $\BDA_n^\U(\F)$, so the map is trivially nullhomotopic. So assume that $\det(\vec{v}_1, \dots, \vec{v}_n) \in \F^\times \setminus \U$.
    In the image of $\rho \circ f$, exactly 3 faces of the image of $f$ are subdivided, namely the images of 
    \begin{align*}
        \{ [\vec{v}_1], [\vec{v}_2], [\vec{v}_3], \dots , [\vec{v}_n]\},&& \{ [\vec{v}_1+\vec{v}_2], [\vec{v}_1], [\vec{v}_3],\dots , [\vec{v}_n]\}, &&
         \{ [\vec{v}_1+\vec{v}_2], [\vec{v}_2], [\vec{v}_3], \dots , [\vec{v}_n]\}
    \end{align*}
By Lemma \ref{choicefreedom}, we can choose the $\U$-vector we use for each subdivision arbitrarily. Lemma \ref{lem:char2} says that $\mathrm{char}(\F) \neq 2$, so we must have $2 \in \F^\times$. To make our choices for subdivision, we break the proof into two cases, depending on if $2 \in \U$ or if $2 \in \F^\times \setminus \U$.
\begin{enumerate}

\item \textbf{Case 1: $2 \in \F^\times \setminus \U$}

For the subdivision of the face $\{ [\vec{v}_1], [\vec{v}_2], \dots , [\vec{v}_n]\}$, we will use $[\vec{w}]$ with 
\begin{align*}
    \vec{w} = c_1 \vec{v}_1 + \dots + c_n \vec{v}_n
\end{align*}
for $\{ [\vec{v}_1], [\vec{v}_2], [\vec{v}_3], \dots , [\vec{v}_n]\}$, where 
\begin{align*}
    c_3, \dots, c_n \in \F^\times \setminus \U
\end{align*} are chosen arbitrarily, and $c_1, c_2 \in \F^\times \setminus \U$ are chosen so that 
\begin{align*}
    c_2 -c_1 = u \in \U
\end{align*}
(which is possible by Lemma \ref{lem:nonsquaresconnected}).
We shall leave the choice of $\U$-vector for the other two subdivided faces unspecified for the moment.

As shown in Figure \ref{fig:killtrianglehelper1}, in $\pi_{n-1}(\BDA_n^\U(\F))$ the sphere $\rho \circ f$ is the sum of the $n$ spheres $\rho \circ f_i$ with 
\begin{align*}
    f_i = \overline{\llb [\vec{w}], [\vec{v}_1 + \vec{v}_2], [\vec{v}_1], \dots, [\widehat{\vec{v}}_i], \dots, [\vec{v}_n] \rrb} \hspace{0.5cm} (1 \leq i \leq n)
\end{align*}

For $3 \leq i \leq n$, we have $\rho \circ f_i = f_i$, and the image of $f_i$ is the boundary of an augmented simplex in $\BDA_n^\U(\F)$, so that $f_i$ is trivially nullhomotopic. Thus we need only show that $\rho \circ f_i$ is nullhomotopic for $i=1, 2$. The proofs in both cases are similar. We describe the $i=2$ case and leave the $i=1$ case to the reader.

When forming $\rho \circ f_2$ for 
\begin{align*}
    f_2 = \overline{\llb [\vec{w}], [\vec{v}_1 + \vec{v}_2], [\vec{v}_1], [\vec{v}_3], \dots, [\vec{v}_n] \rrb}
\end{align*}

only two faces are subdivided, namely the images of
\begin{align*}
    \{[\vec{v}_1 + \vec{v}_2], [\vec{v}_1], [\vec{v}_3], \dots, [\vec{v}_n]\} \hspace{0.5cm} \text{ and } \hspace{0.5cm} \{[\vec{w}], [\vec{v}_1 + \vec{v}_2], [\vec{v}_3], \dots, [\vec{v}_n]\}
\end{align*}

We shall use the freedom we have of picking the vector we use for each subdivision. The key observation here is that we can use the same vector for both subdivisions, namely $[\vec{w'}]$ with
\begin{align*}
    \vec{w'} = (c_1-u)\vec{v}_1 + c_2\vec{v}_2 + c_3\vec{v}_3 + \dots + c_n \vec{v}_n.
\end{align*}

For the face $\{ [\vec{v}_1 + \vec{v}_2], [\vec{v}_1], [\vec{v}_3], \dots, [\vec{v}_n]\}$ this follows from the fact that 
\begin{align*}
    (c_1-u)\vec{v}_1 + c_2\vec{v}_2 + c_3\vec{v}_3 + \dots + c_n \vec{v}_n & = ((c_2-u)-u)\vec{v}_1 + c_2\vec{v}_2 + c_3\vec{v}_3 + \dots + c_n \vec{v}_n\\
    & = c_2(\vec{v}_1 + \vec{v}_2) - 2u\vec{v_1} + c_3 \vec{v}_3 + \dots + c_n \vec{v_n}
\end{align*}

(Note that since $-u \in \U$ and $2 \in \F^\times \setminus \U$, this implies that $-2u \in \F^\times \setminus \U$).

For $\{[\vec{w}], [\vec{v}_1 + \vec{v}_2], [\vec{v}_3], \dots, [\vec{v}_n]\}$, we have the following:
\begin{IEEEeqnarray*}{l}
    \hspace*{-3em} (c_1-u)\vec{v}_1 + c_2\vec{v}_2 + c_3\vec{v}_3 + \dots + c_n \vec{v}_n \\
    \quad \quad \text{} = (c_1-(c_2-c_1))\vec{v}_1 + (2c_2-c_2)\vec{v}_2 + (2c_3-c_3)\vec{v}_3 + \dots + (2c_n-c_n)\vec{v}_n\\
     \quad \quad \text{} = (2c_1-c_2)\vec{v}_1 + (2c_2-c_2)\vec{v}_2 + (2c_3-c_3)\vec{v}_3 + \dots + (2c_n-c_n)\vec{v}_n\\
     \quad \quad \text{} = 2(c_1\vec{v}_1 + c_2\vec{v}_2 + \dots + c_n\vec{v}_n) - c_2(\vec{v}_1 + \vec{v}_2) - c_3\vec{v}_3 - \dots - c_n\vec{v}_n\\
     \quad \quad \text{} = 2\vec{w} - c_2(\vec{v}_1 + \vec{v}_2) - c_3\vec{v}_3 - \dots - c_n\vec{v}_n 
\end{IEEEeqnarray*}
Thus we can use $\vec{w'}$ for the subdivision of both faces.
The two faces 
\begin{align}
\label{facescase2}
\{[\vec{v}_1 + \vec{v}_2], [\vec{v}_1], [\vec{v}_3], \dots, [\vec{v}_n]\}  \text{ and } \{[\vec{w}], [\vec{v}_1 + \vec{v}_2], [\vec{v}_3], \dots, [\vec{v}_n]\}
\end{align}
meet in a common $(n-2)$-dimensional simplex
\begin{align*}
    \eta = \{[\vec{v}_1 + \vec{v}_2], [\vec{v}_3], \dots, [\vec{v}_n]\}
\end{align*}
    As in Figure \ref{fig:killtrianglehelper2}, we can homotope $\rho \circ f_2$ so as to replace the two subdivisions of the faces \eqref{facescase2} with a single subdivision of the face $\eta$ by $[\vec{w'}]$. The result in $\pi_{n-1}(\BDA_n^\U(\F))$ is the sum of the $(n-1)$ different spheres
    \begin{align*}
        \overline{\llb [\vec{w}], [\vec{v}_1], [\vec{w'}], [\vec{v}_3], \dots, [\vec{v}_n] \rrb}
    \end{align*}
    and 
    \begin{align*}
        \overline{\llb [\vec{w}], [\vec{v}_1], [\vec{w'}], [\vec{v}_1+\vec{v}_2], [\vec{v}_3], \dots, [\widehat{\vec{v}}_i], \dots, [\vec{v}_n] \rrb}  \hspace{0.25cm}  (3 \leq i \leq n)
    \end{align*}
    These correspond to all the ways of replacing a vertex of $\eta$ with $[\vec{w'}]$ and then adding the vertices $[\vec{w}]$ and $[\vec{v}_1]$ that do not appear in $\eta$. Since $\vec{w} = \vec{w'} + u\vec{v}_1$, these are all the boundaries of additive simplices in $\BDA_n^\U(\F)$, and are thus nullhomotopic.

\item \textbf{Case 2: $2 \in \U$}

For the subdivision of the face $\{ [\vec{v}_1], [\vec{v}_2], \dots , [\vec{v}_n]\}$, we will use $[\vec{w}]$ with 
\begin{align*}
    \vec{w} = c_1 \vec{v}_1 + \dots + c_n \vec{v}_n
\end{align*}
for $\{ [\vec{v}_1], [\vec{v}_2], [\vec{v}_3], \dots , [\vec{v}_n]\}$, where 
\begin{align*}
    c_2, c_3, \dots, c_n \in \F^\times \setminus \U
\end{align*} are chosen arbitrarily, and where
\begin{align*}
    c_1 = -c_2 
\end{align*}
We shall leave the choice of $\U$-vector for the other two subdivided faces unspecified for the moment.

By Lemma \ref{lem:char2}, we have $\mathrm{char}(\F) \neq 2$, which implies that $c_2 \neq -c_2$. This implies that $\{ [\vec{w}], [\vec{v}_1+\vec{v}_2], [\vec{v}_3], \dots, [\vec{v}_n]\}$ is a simplex of $\BAO_n^\U(\F)$. In particular, the maps
\begin{align*}
    f_i = \overline{\llb [\vec{w}], [\vec{v}_1 + \vec{v}_2], [\vec{v}_1], \dots, [\widehat{\vec{v}}_i], \dots, [\vec{v}_n] \rrb} : \del \Delta^n \to \BAO_n^\U(\F) 
\end{align*}
are well-defined for all $1 \leq i \leq n$.

Similar to Case 1, in $\pi_{n-1}(\BDA_n^\U(\F))$ the sphere $\rho \circ f$ is the sum of the $n$ spheres $\rho \circ f_i$ as shown in Figure \ref{fig:killtrianglehelper1}, with 
\begin{align*}
    f_i = \overline{\llb [\vec{w}], [\vec{v}_1 + \vec{v}_2], [\vec{v}_1], \dots, [\widehat{\vec{v}}_i], \dots, [\vec{v}_n] \rrb} \hspace{0.5cm} (1 \leq i \leq n)
\end{align*}

For $3 \leq i \leq n$, we have $\rho \circ f_i = f_i$, and the image of $f_i$ is the boundary of an augmented simplex in $\BDA_n^\U(\F)$, so that $f_i$ is trivially nullhomotopic. Thus we need only show that $\rho \circ f_i$ is nullhomotopic for $i=1, 2$. The proofs in both cases are similar. We describe the $i=2$ case and leave the $i=1$ case to the reader.   

Only one simplex in the image of $f_2$ gets subdivided under the map $\rho \circ f_2$, namely 
\begin{align*}
   \{ [\vec{v}_1 + \vec{v}_2], [\vec{v}_1], [\vec{v}_3], \dots, [\vec{v}_n]\} 
\end{align*}
    
The key now is that we can pick $\vec{w}$ for this subdivision. This is because:
\begin{align*}
    \vec{w} & = -c_2\vec{v}_1 + c_2\vec{v}_2 + c_3\vec{v}_3 + \dots + c_n\vec{v}_n \\
    & = c_2(\vec{v}_1 + \vec{v}_2) -2c_2\vec{v}_1 + c_3\vec{v}_3 + \dots + c_n\vec{v}_n 
\end{align*}
(Since $2 \in \U$ and $-c_2 \in \F^\times \setminus \U$, this implies $-2c_2 \in \F^\times \setminus \U$).

This implies that the image of $\rho \circ f_2$ in $\BDA_n^\U(\F)$ is the degenerate sphere
\begin{align*}
    \overline{\llb [\vec{w}], [\vec{w}] \rrb} * \overline{\llb [\vec{v}_1+\vec{v}_2],[\vec{v}_1], [\vec{v}_3], \dots, [\vec{v}_n] \rrb}
\end{align*}
which is trivially nullhomotopic.

\end{enumerate}
\end{proof}

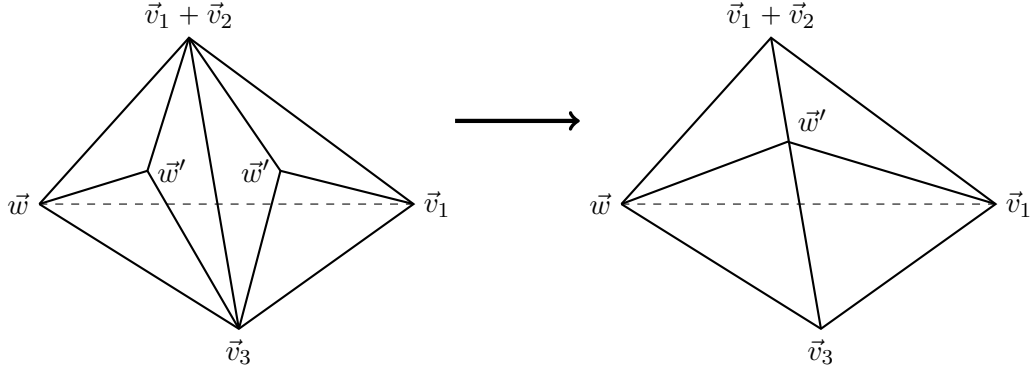
\begin{figure}
    \centering
    \begin{tikzpicture}[scale=1.1]
        \draw[thick] (0,0) -- (1.8,2) -- (4.5,0) -- (2.4, -1.5) -- (0,0);
        \draw[thick] (1.8,2) -- (2.4, -1.5); \draw[dashed] (0,0) -- (4.5,0);

        \draw[thick] (0,0) -- (1.3,0.4) -- (1.8,2); \draw[thick] (1.3, 0.4) -- (2.4, -1.5);

        \draw[thick] (1.8,2) -- (2.9, 0.4) -- (4.5,0); \draw[thick] (2.9, 0.4) -- (2.4, -1.5);

        \node[left] at (0,0) {$\vec{w}$}; \node[below] at (2.4, -1.5) {$\vec{v}_3$};
        \node[right] at (4.5,0) {$\vec{v}_1$}; \node[above] at (1.8, 2) {$\vec{v}_1 + \vec{v}_2$};

        \node[right] at (1.3,0.4) {$\vec{w}'$}; \node[left] at (2.9, 0.4) {$\vec{w}'$};

        \draw[->, ultra thick] (5,1) -- (6.5,1);

        \draw[thick] (7,0) -- (8.8,2) -- (11.5,0) -- (9.4, -1.5) -- (7,0);
        \draw[thick] (8.8,2) -- (9.4, -1.5); \draw[dashed] (7,0) -- (11.5,0);
        
        \draw[thick] (7,0) -- (9, 0.75) -- (11.5,0);

        \node[left] at (7,0) {$\vec{w}$}; \node[below] at (9.4, -1.5) {$\vec{v}_3$};
        \node[right] at (11.5,0) {$\vec{v}_1$}; \node[above] at (8.8, 2) {$\vec{v}_1 + \vec{v}_2$};

        \node[above right] at (9, 0.75) {$\vec{w}'$};
        
    \end{tikzpicture}
    \caption{On the left is the sphere appearing in the case $i=2$ in the proof of Lemma \ref{killtrianglehelper} for $n=3$. On the right is the result of homotoping it into the union of two tetrahedra.}
    \label{fig:killtrianglehelper2}
\end{figure}

\begin{proof}[Proof of Lemma \ref{killtriangle}]
For some $n \geq 3$, let $g: S^{n-1} \to \B_n^\U(\F)$ be an initial D-triangle map, let $\rho: \B_n^\U(\F) \to \BD_n^\U(\F)$ be the retraction given by Lemma \ref{BDretract}, and let $\iota: \BD_n^\U(\F) \hookrightarrow \BDA_n^\U(\F)$ be the inclusion. We must prove that $\iota \circ \rho \circ g : S^{n-1} \to \BDA_n^\U(\F)$ is nullhomotopic.

By definition, the initial D-triangle map $g$ is of the following form: let $\sigma = \{ [\vec{v}_0], [\vec{v}_1], [\vec{v}_2]\}$ be a 2-dimensional additive simplex of $\BA_n^\U(\F)$ with $\vec{v}_0 = \lambda \vec{v}_1 + \nu \vec{v}_2$. By multiplying $\vec{v}_1, \vec{v}_2$ with units in $\U$ is necessary, we may assume that $\lambda = \nu = 1$. Let $f: S^{n-3} \to \Link_{\BA_n^\U(\F)}(\sigma)$ be a simplicial map for some triangulation of $S^{n-3}$. We then have
\begin{align*}
    g: \overline{\llb [\vec{v}_0], [\vec{v}_1], [\vec{v}_2] \rrb} * f = \overline{\llb [\vec{v}_1 + \vec{v}_2], [\vec{v}_1], [\vec{v}_2] \rrb} * f : \del \Delta^2 * S^{n-3} \cong S^{n-1} \to \B_n^\U(\F).
\end{align*}
Our goal is to show that 
\begin{align*}
    \iota \circ \rho \left( \overline{\llb [\vec{v}_1 + \vec{v}_2], [\vec{v}_1], [\vec{v}_2] \rrb} * f \right) : \del \Delta^2 * S^{n-3} \to \BDA_n^\U(\F)
\end{align*}
is nullhomotopic.

It is enought to show that this map extends over $\Delta^2 * S^{n-3}$. The only simplices of $\Delta^2 * S^{n-3}$ whose images under this map are not simplices of $\BDA_n^\U(\F)$ are of the form $\Delta^2 * \tau$ where $\tau$ maps to a simplex $\{ [\vec{v}_3], \dots, [\vec{v}_n]\}$ such that $\det(\vec{v}_1, \dots, \vec{v_n}) \in \F^\times \setminus \U$. By obstruction theory, it is enough to show that $\del(\Delta^2 * \tau)$ is mapped to an $(n-1)$-sphere that is nullhomotopic. Since the restriction of our map to $\del(\Delta^2 * \tau)$ is precisely
\begin{align*}
    \overline{\llb [\vec{v}_1 + \vec{v}_2], [\vec{v}_1], [\vec{v}_2], [\vec{v}_3], \dots, [\vec{v}_n] \rrb} 
\end{align*}
this follows immediately from Lemma \ref{killtrianglehelper}.
\end{proof}

\subsubsection{Killing initial D-suspend maps}\label{suspends}

In this section we prove Lemma \ref{killsuspend}, i.e. we show that the images of initial D-suspend maps in $\BDA_n^\U(\F)$ are nullhomotopic.

The proof will require two lemmas.

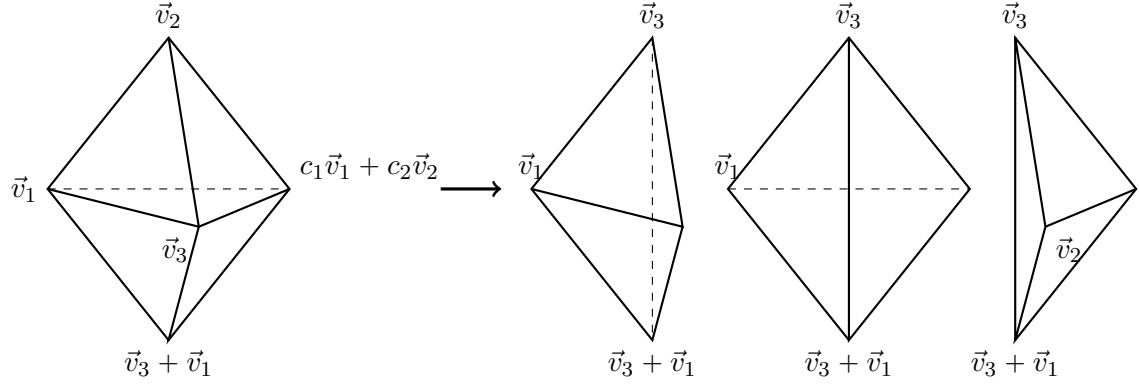
\begin{figure}
    \begin{tikzpicture}[xscale=0.8]
        \draw [thick] (0,0) -- (-2,2) -- (0,4) -- (2,2) -- (0,0) -- (0.5,1.5) -- (0,4);
        \draw [thick] (-2,2) -- (0.5, 1.5) -- (2,2);
        \draw [dashed] (-2,2) -- (2,2) ;

        \node[below] at (0,0) {$\vec{v}_3+\vec{v}_1$};
        \node[left] at (-2,2) {$\vec{v}_1$};
        \node[above] at (0,4) {$\vec{v}_2$};
        \node[above right] at (2,2) {$c_1\vec{v}_1 + c_2\vec{v}_2$};
        \node[below left] at (0.5,1.5) {$\vec{v}_3$};

        \draw [->, very thick] (4.5,2) -- (5.5,2) ;

        \draw [thick] (8.5,1.5) --(8,0) --(6,2) --(8,4) --(8.5,1.5) --(6,2) ;
        \draw [dashed] (8,4) -- (8,0) ;

        \node[below] at (8,0) {$\vec{v}_3 + \vec{v}_1$};
        \node[above] at (6,2) {$\vec{v}_1$};
        \node[above] at (8,4) {$\vec{v}_3$};

        \draw [thick] (11.25,0) -- (9.25,2) -- (11.25,4) -- (13.25,2) -- (11.25,0) -- (11.25,4);
        \draw [dashed] (9.25,2) -- (13.25,2) ;

        \node[above] at (9.25,2) {$\vec{v}_1$};
        \node[above] at (11.25,4) {$\vec{v}_3$};
        \node[below] at (11.25,0) {$\vec{v}_3 + \vec{v}_1$};

        \draw[thick] (14,0)--(14,4)--(16,2)--(14,0)--(14.5,1.5)--(14,4);
        \draw[thick] (14.5,1.5) -- (16,2);

        \node[above] at (14,4) {$\vec{v}_3$};
        \node[below] at (14,0) {$\vec{v}_3 + \vec{v}_1$};
        \node[below right] at (14.5,1.5) {$\vec{v}_2$};
        
    \end{tikzpicture}
    \caption{The sphere in the proof of Lemma \ref{killsuspendhelper1} in the case $n=3$, along with the result of breaking it into $n=3$ spheres.}
    \label{fig:killsuspend1}
\end{figure}

\begin{lemma}
\label{killsuspendhelper1}
Let $\F$ be a finite field and $\U \subset \F^\times$ a multiplicative subgroup such that $-1 \in \U$, and the quotient $\F^\times / \U$ is isomorphic to the cyclic group $\{\pm 1\}$ of order 2.

For some $n \geq 2$, let $\rho: \BAO_n^\U(\F) \to \BDA_n^\U(\F)$ be the retraction constructed in \ref{BDAretract}. Let $\vec{v}_1, \dots, \vec{v}_n$ be a basis for $\F^n$ such that $\det(\vec{v}_1, \dots, \vec{v}_n) \in \F^\times \setminus \U$. Pick some $c_1, \dots, c_{n-1} \in \F^\times \setminus \U$ and some $\vec{v} \in \langle \vec{v}_1, \dots, \vec{v}_{n-1} \rangle \subset \F^n$. Then the maps

    \begin{align*}
        \rho \circ \left( \llb [\vec{v}_n] \rrb * \overline{\llb [\vec{v}_1], \dots, [\vec{v}_{n-1}], [c_1\vec{v}_1 + \dots + c_{n-1}\vec{v}_{n-1}] \rrb} \right) : \Delta^0 * \del \Delta^{n-1} \to \BD_n^\U(\F)
    \end{align*}
    and
    \begin{align*}
        \rho \circ \left( \llb [\vec{v}_n + \vec{v}] \rrb * \overline{\llb [\vec{v}_1], \dots, [\vec{v}_{n-1}], [c_1\vec{v}_1 + \dots + c_{n-1}\vec{v}_{n-1}] \rrb} \right) : \Delta^0 * \del \Delta^{n-1} \to \BD_n^\U(\F)
    \end{align*}
    are homotopic in $\BDA_n^\U(\F)$ through maps fixing $\del(\Delta^0 * \del \Delta^{n-1}) = \del \Delta^{n-1}$.
\end{lemma}
\begin{proof}
    By Lemma \ref{lem:additivegeneration}, since $\F$ is additively generated by the units $\U$, it is enough to deal with the case when $\vec{v} = u \vec{v}_i$ for some $u \in \U$ and $1 \leq i \leq n-1$; the general case can then be deduced via a sequence of these homotopies. We may assume WLOG that $i=1$. Also, by multiplying $v_1$ with a unit in $\U$ if necessary, we can further assume that $u=1$, so that $\vec{v} = \vec{v}_1$.

    Our goal is equivalent to showing that the map
    \begin{align*}
        \rho \circ \left( \overline{\llb [\vec{v}_n], [\vec{v}_n + \vec{v}_1] \rrb} *  \overline{\llb [\vec{v}_1], \dots, [\vec{v}_{n-1}], [c_1\vec{v}_1 + \dots + c_{n-1}\vec{v}_{n-1}] \rrb}\right) : \del \Delta^1 * \del \Delta^{n-1} \to \BD_n^\U(\F)
    \end{align*}
    is nullhomotopic in $\BDA_n^\U(\F)$. As shown in Figure \ref{fig:killsuspend1}, as an element of $\pi_{n-1}(\BDA_n^\U(\F))$ this is the sum of $n$ spheres.
    The first of these is the sphere
    \begin{align*}
        \rho \circ \overline{\llb [\vec{v_n}], [\vec{v}_n + \vec{v}_1], [\vec{v}_1], \dots, [\vec{v}_{n-1}] \rrb} : \del \Delta^n \to \BDA_n^\U(\F)
    \end{align*}
    which is nullhomotopic by Lemma \ref{killtrianglehelper}.

    The other $(n-1)$ spheres are
    \begin{align*}
        \rho \circ \overline{\llb [\vec{v_n}], [\vec{v}_n + \vec{v}_1], [\vec{v}_1], \dots, [\widehat{\vec{v}}_i], \dots, [\vec{v}_{n-1}], [c_1\vec{v}_1+ \dots + c_{n-1}\vec{v}_{n-1}] \rrb} : \del \Delta^n \to \BDA_n^\U(\F)
    \end{align*}
    for $1 \leq i \leq n-1$.

    For $2 \leq i \leq n-1$, these spheres are nullhomotopic by Lemma \ref{killtrianglehelper}.

    For $i=1$, observe that exactly one face of this sphere gets subdivided by $\rho$, namely
    \begin{align*}
        \llb [\vec{v}_n], [\vec{v}_n + \vec{v}_1], [\vec{v}_2], \dots, [\vec{v}_{n-1}] \rrb
    \end{align*}
    By Lemma \ref{choicefreedom}, we can arbitrarily choose the vertex we use in this subdivision. Thus we can use $[\vec{w}]$ with 
    \begin{align*}
        \vec{w} & = -c_1\vec{v}_n + c_1(\vec{v}_n + \vec{v}_1) + c_2\vec{v}_2 + \dots + c_{n-1}\vec{v}_{n-1} \\
        & = c_1\vec{v}_1 + c_2\vec{v}_2 + \dots + c_{n-1}\vec{v}_{n-1}.
    \end{align*}
    Then our sphere in this case is the degenerate sphere
    \begin{align*}
        \overline{\llb [c_1\vec{v}_1 + c_2\vec{v}_2 + \dots + c_{n-1}\vec{v}_{n-1}], [c_1\vec{v}_1 + c_2\vec{v}_2 + \dots + c_{n-1}\vec{v}_{n-1}] \rrb} * \overline{\llb [\vec{v}_n], [\vec{v}_n + \vec{v}_1], [\vec{v}_2], \dots, [\vec{v}_{n-1}] \rrb},
    \end{align*}
    which is trivially nullhomotopic.
\end{proof}

\begin{figure}
\centering
\begin{tikzpicture}[scale=0.8]
    \draw [thick] (2,0) -- (0,2) -- (2,4) -- (4,2) -- (2,0) ;
    \draw[thick] (2,0) -- (2,1) -- (0,2) -- (2,3) -- (2,4) ;
    \draw[thick] (2,3) -- (4,2) -- (2,1) ;
    \draw[thick] (0,2) -- (4,2) ;

    \node[left] at (0,2) {$\vec{v}_1$};
    \node[below] at (2,0) {$\vec{e}_3 + \vec{u}$};
    \node[right] at (4,2) {$\vec{v}_2$};
    \node[above] at (2,4) {$\vec{e}_3$};

    \draw [->, very thick] (5,2) -- (7,2);

    \draw [thick] (10,0) -- (8,2) -- (10,4) -- (12,2) -- (10,0) ;
    \draw[thick] (8,2) -- (12,2);
    \draw[thick] (10,0) -- (10,4);

    \node[left] at (8,2) {$\vec{v}_1$};
    \node[below] at (10,0) {$\vec{e}_3 + \vec{u}$};
    \node[right] at (12,2) {$\vec{v}_2$};
    \node[above] at (10,4) {$\vec{e}_3$};

    \draw [->] (10.1,2.1) -- (12,3.5);
    \node[above] at (12,3.5) {$c_1\vec{v}_1 + c_2\vec{v}_2$};
\end{tikzpicture}
\caption{The homotopy we are trying to achieve in Lemma \ref{killsuspendhelper2} for $n=3$.}
\label{fig:killsuspend2}
\end{figure}
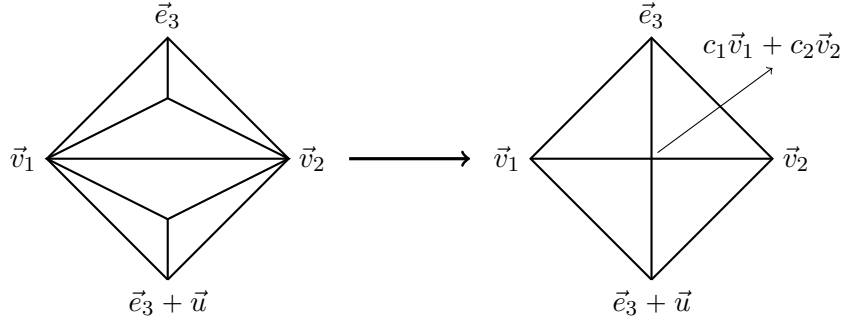

\begin{lemma}
\label{killsuspendhelper2}
Let $\F$ be a finite field and $\U \subset \F^\times$ a multiplicative subgroup such that $-1 \in \U$, and the quotient $\F^\times / \U$ is isomorphic to the cyclic group $\{\pm 1\}$ of order 2.

For some $n \geq 3$, let $\rho: \BAO_n^\U(\F) \to \BDA_n^\U(\F)$ and $\rho' : \BAO_{n-1}^\U(\F) \to \BDA_{n-1}^\U(\F)$ be the retractions constructed in \ref{BDAretract}. Let $\{ \vec{e}_1, \dots, \vec{e}_n\}$ be the standard basis for $\F^n$, let $\{\vec{v}_1, \dots \vec{v}_{n-1}\}$ be some basis for $\F^{n-1} \subset \F^n$, and let $\vec{v} \in \F^{n-1}$. Then the maps
    \begin{align*}
        \rho \circ \left( \overline{\llb [\vec{e}_n], [\vec{e}_n + \vec{v}] \rrb} * \llb [\vec{v}_1], \dots, [\vec{v}_{n-1}] \rrb \right) : \del \Delta^1 * \Delta^{n-1} \to \BDA_n^\U(\F)
    \end{align*}
    and 
    \begin{align*}
        \overline{\llb [\vec{e}_n], [\vec{e}_n + \vec{v}] \rrb} * \left( \rho' \circ  \llb [\vec{v}_1], \dots, [\vec{v}_{n-1}] \rrb \right) : \del \Delta^1 * \Delta^{n-1} \to \BDA_n^\U(\F)
    \end{align*}
    are homotopic through maps fixing $\del(\del \Delta^1 * \Delta^{n-1}) = \del \Delta^1 * \del \Delta^{n-1}$.
\end{lemma}

\begin{proof}
    If $\det(\vec{v}_1, \dots, \vec{v}_{n-1}) \in \U$, then both of these maps are equal, so assume that this determinant is in $\F^\times \setminus \U$. Note that $\rho'$ subdivides the image of $\llb [\vec{v}_1], \dots, [\vec{v}_{n-1}] \rrb$, and we can choose the vector for this subdivision to be $[\vec{w}]$ with
    \begin{align*}
        \vec{w} = c_1\vec{v}_1 + \dots + c_{n-1}\vec{v}_{n-1}
    \end{align*}
    for some $c_1, \dots, c_{n-1} \in \F^\times \setminus \U$. See Figure \ref{fig:killsuspend2} for the homotopy we are trying to achieve.

    The key observation now is that this is really a disguised form of Lemma \ref{killsuspendhelper1}.
    Indeed, the images of 
    \begin{align*}
        \rho \circ \left( \overline{\llb [\vec{e}_n], [\vec{e}_n + \vec{v}] \rrb} * \llb [\vec{v}_1], \dots, [\vec{v}_{n-1}] \rrb \right) : \del \Delta^1 * \Delta^{n-1} \to \BDA_n^\U(\F)
    \end{align*}
    and 
    \begin{align*}
        & \overline{\llb [\vec{e}_n], [\vec{e}_n + \vec{v}] \rrb} * \left( \rho' \circ  \llb [\vec{v}_1], \dots, [\vec{v}_{n-1}] \rrb \right) \\
        & = \hspace*{-1.5em} \sum_{1 \leq i \leq n-1} \hspace*{-1em} \overline{\llb [\vec{e}_n], [\vec{e}_n + \vec{v}] \rrb} *  \llb [\vec{v}_1], \dots, [\widehat{\vec{v}}_i], \dots, [\vec{v}_{n-1}], [c_1\vec{v}_1 + \dots + c_{n-1}\vec{v}_{n-1}] \rrb : \del \Delta^1 * \Delta^{n-1} \to \BDA_n^\U(\F)
    \end{align*}

    together are exactly the image of the map

    \begin{align*}
        \rho \circ \left( \overline{\llb [\vec{e}_n], [\vec{e}_n + \vec{v}] \rrb} * \overline{\llb [\vec{v}_1], \dots, [\vec{v}_{n-1}], [c_1\vec{v}_1 + \dots + c_{n-1}\vec{v}_{n-1}] \rrb} \right) : \del \Delta^1 * \Delta^{n-1} \to \BDA_n^\U(\F)
    \end{align*}
    which Lemma \ref{killsuspendhelper1} says is nullhomotopic.
\end{proof}

\begin{proof}[Proof of Lemma \ref{killsuspend}]
    For some $n \geq 3$, let $g: S^{n-1} \to \B_n^\U(\F)$ be an initial D-suspend map, let $\rho: \B_n^\U(\F) \to \BD_n^\U(\F)$ be the retraction given by Lemma \ref{BDretract}, and let $\iota : \BD_n^\U(\F) \hookrightarrow \BDA_n^\U(\F)$ be the inclusion. Assume that $\pi_{n-2}(\BDA_{n-1}^\U(\F)) = 0$. We must prove that $\iota \circ \rho \circ g : S^{n-1} \to \BDA_n^\U(\F)$ is nullhomotopic.

    By definition, the initial D-suspend map $g$ is of the following form. Let $\vec{v} \in \F^n$ be a non-zero vector, let $W \subset \F^n$ be an $(n-1)$-dimensional subspace such that $\F^n = \langle \vec{v} \rangle \oplus W$, and let $\vec{w} \in W$ be non-zero. Let $f: S^{n-2} \to \B^\U(W)$ be a simplicial map for some triangulation of $S^{n-2}$. We then have
    \begin{align*}
        g = \overline{\llb [\vec{v}], [\vec{v} + \vec{w}] \rrb} * f : \del \Delta^1 * S^{n-2} \cong S^{n-1} \to \B_n^\U(\F)
    \end{align*}
    Let $\{\vec{e}_1, \dots, \vec{e}_n\}$ be the standard basis for $\F^n$. Changing coordinates with an element of $\SL_n(\F)$, we can assume that $\vec{v} = \vec{e}_n$ and $W = \F^{n-1}$. Our map $f$ thus lands in $\B_{n-1}^\U(\F)$, and our goal is to prove the map
    \begin{align*}
        \iota \circ \rho \circ \left( \overline{\llb [\vec{e}_n], [\vec{e}_n + \vec{w}] \rrb} * f \right) : \del \Delta^1 * S^{n-2} \to \BDA_n^\U(\F)
    \end{align*}
    is nullhomotopic.

    Let $\rho': \BAO_{n-1}^\U(\F) \to \BDA_{n-1}^\U(\F)$ be the retraction constructed in \ref{BDAretract}. Applying Lemma \ref{killsuspendhelper2} to $S^0 * \sigma$ for each $(n-2)$-simplex $\sigma$ of $S^{n-2}$, we see that our map is homotopic to 
    \begin{equation}\label{rho'}
        \overline{\llb [\vec{e}_n], [\vec{e}_n + \vec{w}] \rrb} * (\rho' \circ f) : \del \Delta^1 * S^{n-2} \to \BDA_n^\U(\F).
    \end{equation}

Since $\BDA_{n-1}^\U(\F)$ is $(n-2)$-connected, the map $\rho' \circ f$ is nullhomotopic in $\BDA_{n-1}^\U(\F)$. Since the suspension of $\BDA_{n-1}^\U(\F)$ with suspension points $[\vec{e}_n]$ and $[\vec{e}_n + \vec{w}]$ lies in $\BDA_n^\U(\F)$, we conclude that \eqref{rho'} is nullhomotopic, as desired.
\end{proof}

\subsubsection{Simple connectivity of $\BDA_2^\U(\F)$}
\label{sec:BDA2}

In this section we prove Lemma \ref{lem:BDA2conn}, i.e. we show that $\BDA_2^\U(\F)$ is 1-connected. To do this we will use a simple special case of a technique that is sometimes called discrete Morse theory (see, for example, \cite{bestvina2008pl} or \cite[Theorem 3.9]{miller2025rank}).

\begin{lemma}
    \label{lem:discretemorse}
    Let $X$ be a simplicial complex and $Y$ a full subcomplex of $X$ such that there is exactly one vertex $v$ of $X$ that is not in $Y$. Suppose $Y$ is simply connected and that $\Link_X(v)$ is a path-connected subcomplex of $Y$. Then $X$ is simply connected.
\end{lemma}

We will also need the following lemma.

\begin{lemma}
\label{lem:filledtriangles}
    Let $\F$ be a field, and $\U$ a (multiplicative) subgroup of $\F^\times$. All triangles in $\BDA_2^\U(\F)$ are filled in, i.e. if $[\vec{v}_0], [\vec{v}_1], [\vec{v}_2]$ are vertices of $\BDA_2^\U(\F)$ such that $\{[\vec{v}_i], [\vec{v}_j]\}$ is an edge of $\BDA_2^\U(\F)$ for every $0 \leq i < j \leq 2$, then $\{[\vec{v}_0], [\vec{v}_1], [\vec{v}_2]\}$ is a 2-simplex in $\BDA_2^\U(\F)$.
\end{lemma}
\begin{proof}
    The condition that $\{[\vec{v}_1], [\vec{v}_2]\}$ is an edge in $\BDA_2^\U(F)$ implies that 
    \begin{align*}
      \det(\vec{v}_1, \vec{v}_2) \in \U  
    \end{align*}
    Since we are working with $\U$-vectors, we can multiply $\vec{v}_1$ by a unit in $\U$ if necessary and assume that
    \begin{align*}
        \det(\vec{v}_1, \vec{v}_2) = 1
    \end{align*}
    Note that $\SL_2(\F)$ acts via simplicial isomorphisms on $\BDA_2^\U(\F)$. Thus we can act by an element of $\SL_2(\F)$ that sends $\vec{v}_i \mapsto \vec{e}_i$ for $i=1,2$.
    So now assume that $\vec{v}_1 = \vec{e}_1$ and $\vec{v}_2 = \vec{e}_2$.
    We have 
    \begin{align*}
        \vec{v}_0 = \lambda\vec{v}_1 + \nu\vec{v}_2
    \end{align*} 
    for some unique $\lambda, \nu \in \F$.

    Since $\{[\vec{v}_0], [\vec{v}_i]\}$ is an edge of $\BDA_2^\U(\F)$ for $i=1,2$, we have
    \begin{align*}
        \det(\vec{v}_0, \vec{v}_i) \in \U  \hspace{0.25cm} (1 \leq i \leq 2)
    \end{align*}

 which is equivalent to saying that 
    \begin{align*}
        \lambda, \nu \in \U
    \end{align*}
    Thus $\{[\vec{v}_0], [\vec{v}_1], [\vec{v}_2]\}$ is an additive 2-simplex of $\BDA_2^\U(\F)$.
\end{proof}

We are now ready to prove Lemma \ref{lem:BDA2conn}.

\begin{proof}[Proof of Lemma \ref{lem:BDA2conn}]
    Recall the statement that we want to prove: Let $\F$ be a finite field and $\U \subset \F^\times$ a multiplicative subgroup such that $-1 \in \U$ and the quotient group $\F^\times/\U \cong \{\pm 1\}$ has order 2. Then $\BDA^\U_2(\F)$ is simply connected.

    Pick an arbitrary element $c \in \F^\times \setminus \U$. Since $\F^\times / \U \cong \{\pm 1\}$, for any $\U$-vector of $\F^2$ we can pick a unique representative $\lambda \vec{e}_1 + \nu \vec{e}_2$ where one of the following three cases hold:
    \begin{itemize}
        \item $\nu = 0$ and $\lambda = 1$ or $c$
        \item $\nu = 1$ and $\lambda \in \F$
        \item $\nu = c$ and $\lambda \in \F$
    \end{itemize}

    We will build up the complex $\BDA_2^\U(\F)$ in a finite number of steps, at each stage using Lemma \ref{lem:discretemorse} to conclude that the simplicial complex is simply connected.

    \begin{enumerate}
        \item Let $Y_0'$ be the star of the vertex $[\vec{e}_1]$ in $\BDA_2^\U(\F)$, i.e. $Y_0'$ consists of all those simplices $\sigma$ so that $\sigma \cup \{[\vec{e}_1]\}$ is also a simplex of $\BDA_2^\U(\F)$. Note that $Y_0'$ is contractible.

        Let us note what the link of $[\vec{e}_1]$ looks like. Any $\U$-vector $[\vec{v}]$ that is joined to $[\vec{e}_1]$ by an edge can be uniquely written as $[\lambda \vec{e}_1 + \vec{e}_2]$ for $\lambda \in \F$. Two such vertices $[\lambda\vec{e}_1+\vec{e}_2]$ and $[\lambda'\vec{e}_1+\vec{e}_2]$ are joined by an edge if and only if $\lambda - \lambda' \in \U$. In this case, Lemma \ref{lem:filledtriangles} says that the edge between $[\lambda\vec{e}_1+\vec{e}_2]$ and $[\lambda'\vec{e}_1+\vec{e}_2]$ also lies in the link of $[\vec{e}_1]$. Thus the link of $[\vec{e}_1]$ is the graph whose vertices correspond to elements of $\F$ and two vertices are joined by an edge if they differ by an element of $\U$.

        \item Let $Y_0$ be the full subcomplex of $\BDA_2^\U(\F)$ spanned by the vertices $[\vec{e}_1]$ and $[\lambda\vec{e}_1 + \vec{e}_2]$ for $\lambda \in \F$. The only simplices that are in $Y_0$ but not in $Y_0'$ (if any) are 2-simplices whose boundary edges are all present in $\Link_{\BDA_2^\U(\F)}([\vec{e}_1])$. We deduce using Lemma \ref{lem:discretemorse} that $Y_0$ is simply connected: For each simplex $\{v_0, v_1, v_2\}$ in $Y_0$ that is not in $Y_0'$, subdivide it with a new vertex $v$ and apply Lemma \ref{lem:discretemorse} with $Y=Y_0$. The link of $v$ precisely consists of the three edges $\{v_i, v_j\}$ for $0 \leq i < j \leq 2\}$, and is path-connected.

        \item Let $Y_1$ be the full subcomplex of $\BDA_2^\U(\F)$ spanned by the vertices of $Y_0$ as well as the vertex $[c\vec{e}_2]$. A vertex $[\lambda\vec{e}_1 + \vec{e}_2]$ is in the link of $[c\vec{e}_2]$ if and only if $c\lambda \in \U$, i.e. if $\lambda \in \F^\times \setminus \U$. As discussed in Step 1, two such vertices are connected by an edge if and only if they differ by an element of $\U$, in which case the edge would also lie in $\Link_{Y_1}([c\vec{e}_2])$ by Lemma \ref{lem:filledtriangles}. So $\Link_{Y_1}([c\vec{e}_2])$ is a graph whose vertices correspond to elements of $\F^\times \setminus \U$, and where two vertices are joined by an edge if they differ by an element of $\U$. This graph is path-connected by Lemma \ref{lem:nonsquaresconnected}, and so $Y_1$ is simply connected by Lemma \ref{lem:discretemorse}.

        \item We will now add in the remaining vertices of the form $[a\vec{e}_1+c\vec{e}_2]$ with $a \in \F$ one at a time. That is, suppose we already have a full subcomplex $Y_{k-1}$ of $\BDA_2^\U(\F)$ that is simply connected. Then pick a vertex $[a\vec{e}_1+c\vec{e}_2]$ that is not in $Y_{k-1}$ (and if no such vertex exists, then move on to Step 5). Let $Y_k$ be the full subcomplex of $\BDA_2^\U(\F)$ on the vertices of $Y_{k-1}$ and the vertex $[a\vec{e}_1 + c\vec{e}_2]$. We want to show that $\Link_{Y_k}([a\vec{e}_1+c\vec{e}_2])$ is path-connected. Let us first look at the full subcomplex of $\Link_{Y_k}([a\vec{e}_1+c\vec{e}_2])$ spanned by vertices of the form $[\lambda\vec{e}_1 + \vec{e}_2]$. Call this complex $Z_k$.

        The condition that 
        \begin{align*}
            \det(\lambda\vec{e}_1+\vec{e}_2, a\vec{e}_1+c\vec{e}_2) \in \U
        \end{align*}
        holds if and only if $c\lambda - a \in \U$, or equivalently,
        \begin{align*}
            \lambda \in \{c^{-1}u + c^{-1}a | u \in \U\}
        \end{align*}
        Since $\F^\times /\U \cong \{\pm 1\}$ and $c \in \F^\times \setminus \U$, we have $c^{-1}u \in \F^\times \setminus \U$. Thus the above condition is equivalent to saying that
        \begin{align*}
            \lambda \in \{ \nu + c^{-1}a | \nu \in \F^\times \setminus \U\}
        \end{align*}
As argued in Step 1, two such vertices $[\lambda\vec{e}_1 + \vec{e}_2]$ are joined by an edge if and only if the corresponding $\lambda$'s differ by an element of $\U$ (in which case the edge also lies in $\Link_{Y_k}([a\vec{e}_1+c\vec{e}_2])$ by Lemma \ref{lem:filledtriangles}). Thus the full subcomplex $Z_k$ of $\Link_{Y_k}([a\vec{e}_1+c\vec{e}_2])$ spanned by vertices of the form $[\lambda\vec{e}_1 + \vec{e}_2]$ is a graph with vertex set $\{ \lambda = \nu + c^{-1}a | \nu \in \F^\times \setminus \U\}$, with two vertices joined by an edge if they differ by an element of $\U$. This is naturally isomorphic to $\Link_{Y_1}([c\vec{e}_2])$ (the isomorphism being just to subtract $c^{-1}a$ from $\lambda$), and so is path-connected.

Now, $\Link_{Y_k}([a\vec{e}_1+c\vec{e}_2])$ may also contain vertices of the form $[b\vec{e}_1+c\vec{e}_2]$. We'll show that for any such vertex, we can pick a vertex of $Z_k$ that is connected to it by an edge in $\Link_{Y_k}([a\vec{e}_1+c\vec{e}_2])$. Since $Z_k$ is path-connected, this will imply that $\Link_{Y_k}([a\vec{e}_1+c\vec{e}_2])$ is path-connected.

The condition that $[a\vec{e}_1+c\vec{e}_2]$ and $[b\vec{e}_1+c\vec{e}_2]$ are connected by an edge translates to saying that 
\begin{align*}
    \det(a\vec{e}_1+c\vec{e}_2, b\vec{e}_1+c\vec{e}_2) \in \U,
\end{align*}
or equivalently, that
\begin{align*}
    c(a-b) \in \U
\end{align*}
      Since $\F^\times/\U \cong \{\pm 1\}$ and $c \in \F^\times \setminus \U$, this holds of and only if $a-b \in \F^\times \setminus \U$.

      By Lemma \ref{lem:sumofunits}, we can write 
      \begin{align*}
          a-b = u_1 + u_2
      \end{align*}
      for some $u_1, u_2 \in \U$.

      Now pick $\lambda$ so that 
      \begin{align*}
          a - c\lambda = u_1
      \end{align*}

      This implies that
      \begin{align*}
          b - c\lambda = -u_2
      \end{align*}

      The above two equations imply that the vertex $[\lambda\vec{e}_1 + \vec{e}_2]$ is connected to both $[a\vec{e}_1 + c\vec{e}_2]$ and $[b\vec{e}_1 + c\vec{e}_2]$ by an edge in $\BDA_2^\U(\F)$, which is exactly what we wanted.

      Keep repeating this step until we've added all vertices of the form $[a\vec{e}_1 + c\vec{e}_2]$. Call this subcomplex $Y$. Thus $Y$ is the full subcomplex of $\BDA_2^\U(\F)$ on all vertices except for $[c\vec{e}_1]$, and is simply connected.

      \item Finally, we add in the vertex $[c\vec{e}_1]$. Its link consists of all vertices of the form $[a\vec{e}_1 + c\vec{e}_2]$ for $a \in \F$. As argued in Step 4, two such vertices are joined by an edge if and only if they differ by an element of $\F^\times \setminus \U$, in which case the edge would also like in the link of $[c\vec{e}_1]$ by Lemma \ref{lem:filledtriangles}. Thus the link of $[c\vec{e}_1]$ is a regular graph on $|\F|$ vertices, with each vertex having degree $|\F^\times \setminus \U| = \frac{|\F|-1}{2}$. Using an argument similar to the one in the proof of Lemma \ref{lem:additivegeneration}, we conclude that this is a connected graph.
    \end{enumerate}
This concludes the proof that $\BDA_2^\U(\F)$ is simply connected.
\end{proof}

\section{The Lee--Szczarba conjecture}
\label{sec:ssargument}

\subsection{Preliminaries}

\subsubsection{The map-of-posets spectral sequence}

In this section, we review some results about the homology of posets with coefficients in a functor and about the map-of-posets spectral sequence. Much of this is due to Quillen \cite{quillen2006higher} and Charney \cite{charney1987generalization}. We begin with some definitions concerning posets.

\begin{definition}
    Let $\X$ be a poset and $x \in \X$. We say $x$ has height $m$ and write $\het(x) = m$ if $m$ is the largest integer such that there exists a chain
    \begin{align*}
        x_0 < \dots < x_m = x \hspace{0.25cm} \text{ with } x_i \in \X \text{ for all } 0 \leq i \leq m.
    \end{align*}
    We write $\X_{>x}$ for the subposet of $\X$ consisting of elements strictly larger than $x$. For a map $f: \Y \to \X$ of posets, we write $f_{\leq x}$ for the subposet of $\Y$ consisting of all $y \in \Y$ such that $f(y) \leq x$.
\end{definition}

A poset $\X$ can be viewed as a category with an object for each $x \in \X$ and a single morphism from $x\in \X$ to $x' \in \X$ precisely when $x \leq x'$. Letting $\Ab$ denote the category of abelian groups, we now recall the definition of the homology of a poset with coefficients in a functor $F: \X \to \Ab$.

\begin{definition}
    Let $\X$ be a poset and let $F: \X \to \Ab$ be a functor. Define $C_{\bullet}(\X; F)$ to be the following chain complex. For $k \geq 0$, we set
    \begin{align*}
        C_k(\X; F) = \bigoplus_{x_0 < \dots < x_k} F(x_0)
    \end{align*}
    where the $x_i$ are understood to be elements of $\X$. The differential $\del: C_k(\X; F) \to C_{k-1}(\X; F)$ is defined to be $\sum_{i=0}^k (-1)^i\del_i$, where $\del_i: C_k(\X; F) \to C_{k-1}(\X; F)$ is as follows: 
    \begin{itemize}
        \item For $0 < i \leq k$, the map $\del_i$ takes the $x_0 < \dots < x_k$ summand of $C_k(\X; F)$ to the $x_0 < \dots < \widehat{x_i}< \dots < x_k$ summand of $C_{k-1}(\X; F)$ via the identity map $F(x_0) \to F(x_0)$.
        \item The map $\del_0$ takes the $x_0 < \dots < x_k$ summand of $C_k(\X; F)$ to the $x_1 < \dots < x_k$ summand of $C_{k-1}(\X; F)$ via the induced map $F(x_0) \to F(x_1)$.
    \end{itemize}
    We define $\h_k(\X; F) = \h_k(C_{\bullet}(\X; F)$.
\end{definition}

\begin{ex}
    Fix a poset $\X$. For a commutative ring $R$, we write $\underline{R}$ for the constant functor on $\X$ with value $R$. We then have $\h_k(\X; \underline{R}) \cong \h_k(|\X|; R)$, where $|\X|$ is the geometric realization of $\X$. We will often simply write this as $\h_k(\X; R)$.
\end{ex}

These homology groups can be very difficult to calculate. One case where there is an easy formula is where the functor $F$ is \emph{supported on elements of height $m$}, i.e. where $F(x) = 0$ for all $x \in \X$ with $\het(x) \neq m$. We then have the following lemma. See e.g. \cite{miller2020non} for a proof.

\begin{lemma}
\label{posetheightm}
    Let $\X$ be a poset and let $F: \X \to \Ab$ be a functor that is supported on elements of height $m$. Then
    \begin{align*}
        \h_k(\X; F) \cong \bigoplus_{\het(x)=m} \widetilde{\h}_{k-1}(|\X_{>x}|; F(x))
    \end{align*}
    where the coefficients $F(x)$ are simply regarded as an abelian group.
\end{lemma}

The reason we will need to consider the homology of a poset with coefficients in a functor is due to the following spectral sequence. See \cite{miller2020non} for a proof.

\begin{theorem}[Map-of-posets spectral sequence]
\label{posetss}
    Let $f: \Y \to \X$ be a map of posets. Then there is a homologically graded spectral sequence
    \begin{align*}
        \E_{kh}^2 = \h_k(\X; [x \mapsto \h_h(f_{\leq x})]) \implies \h_{k+h}(\Y).
    \end{align*}
\end{theorem}

\begin{remark}
    Following Miller--Patzt--Putman \cite{miller2021top}, we use the nonstandard indices $(k,h)$ since for us $p$ is always a prime, so we cannot use $(p,q)$, and $n$ is always a dimension, so we cannot use $(n,m)$.
\end{remark}

\subsubsection{The quotient of the Tits building by a congruence subgroup}

We will need a concrete description of the quotient of the Tits building by a congruence subgroup.

\begin{definition}[Tits building]
    Let $R$ be a commutative ring and let $V$ be a finite-rank free $R$=module. Define $\T(V)$ to be the poset of proper nonzero direct summands of $V$, ordered by inclusion. Also, let $\cT(V)$ denote the geometric realization of $\T(V)$, viewed as a simplicial complex. For $n \geq 1$, we will write $\T_n(R) = \T(R^n)$ and $\cT_n(R) = \cT(R^n)$.
\end{definition}

\begin{theorem}[Solomon--Tits, \cite{solomon1969steinberg}, \cite{brown1989buildings}] \label{thm:soltit}
    Let $K$ be a field. The Tits building $\mc{T}_n(K)$ is homotopy equivalent to a wedge of spheres of dimension $n-2$. The reduced homology group $\St_n(K) = \widetilde{\h}_{n-2}(\mc{T}_n(K))$ is generated by certain cycles called \emph{apartment classes}, which are indexed by bases of $K^n$.
    For a basis $\vec{v}_1, \vec{v}_2, \dots, \vec{v}_{n}$ of $K^n$, the corresponding \emph{apartment class} is the homology cycle determined by the chain
\begin{align*}
    \sum_{\sigma \in \Sigma_{n}} (-1)^{\sgn(\sigma)} [\lin{\vec{v}_{\sigma(1)}}, \lin{\vec{v}_{\sigma(1)}, \vec{v}_{\sigma(2)}}, \dots, \lin{\vec{v}_{\sigma(1)}, \dots, \vec{v}_{\sigma(n-1)}}]
\end{align*}
where $\Sigma_n$ denotes the Weyl group of permutations of the set $\{ 1, 2, \dots, n\}$.
\end{theorem}

\begin{lemma}
    Suppose $R$ is a PID, and $K$ its field of fractions. For $n \geq 1$, we have $\T_n(R) \cong \T_n(K)$.
\end{lemma}
\begin{proof}
    This follows from the fact that there is a bijection between subspaces of $K^n$ and direct summands of $R^n$ taking a subspace $V \subset K^n$ to $V \cap R^n$ and a direct summand $W \subset R^n$ to $W \otimes_R K$.
\end{proof}

We now decorate our buildings by appropriate versions of orientations.

\begin{definition}[$\U$-orientation]
    Let $R$ be a commutative ring and let $\U$ be a multiplicative subgroup of $R^\times$ such that $-1 \in \U$. Let $V$ be a rank-$d$ free $R$-module, so $\bigwedge^dV \cong R^1$. An \emph{orientation} on $V$ is an element $\w \in \bigwedge^dV$ that generates it as an $R$-module. The group $R^\times$ of units acts transitively on the set of orientations on $V$ by scalar multiplication. A $\U$-orientation on $V$ is a $\U$-vector $[\w]$ such that $\w$ is an orientation on $V$.
\end{definition}

\begin{remark}
    Note that if $\U = R^\times$, then there is a unique $\U$-orientation on V.
\end{remark}

\begin{definition}[$\U$-oriented Tits building]
    Let $R$ be a commutative ring, $\U$ a multiplicative subgroup of $R^\times$ with $-1 \in \U$, and let $V$ be a finite-rank free $R$-module. Define $\T^\U(V)$ to be the poset of proper nonzero direct summands of $V$ equipped with a $\U$-orientation. The poset structure is simply inclusion; the $\U$-orientations play no role in it. Let $\cT^\U(V)$ denote the geometric realization of $\T^\U(V)$, viewed as a simplicial complex. Finally, let $\T_n^\U(R) = \T^\U(R^n)$ and $\cT_n^\U(R) = \cT^\U(R^n)$.
    We call $\cT_n^\U(R)$ the \emph{$\U$-oriented Tits building}.
\end{definition}

\begin{remark}
    We have $\cT_n^\U(R) = \cT_n(R)$ if and only if $R^\times = \U$.
\end{remark}

The following is the analogue of the Solomon--Tits theorem \cite{solomon1969steinberg, brown1989buildings} for $\U$-oriented Tits buildings.

\begin{lemma}
\label{solomontitsU}
    For any field $\F$ and a multiplicative subgroup $\U$ of $\F^\times$ such that $-1 \in \U$, the complex $\cT_n^\U(\F)$ is Cohen-Macaulay of dimension $(n-2)$, for all $n \geq 1$.
\end{lemma}
\begin{proof}
    The Solomon--Tits theorem \cite{solomon1969steinberg, brown1989buildings} implies that $\cT(\F)$ is Cohen-Macaulay of dimension $(n-2)$. The complex $\cT_n^\U(\F)$ is a complete join complex over $\cT_n(\F)$ in the sense of Hatcher--Wahl \cite[Definition 3.2]{hatcher2010stabilization}, and so the lemma follows from \cite[Proposition 3.5]{hatcher2010stabilization}.
\end{proof}

Now we come to the main result of this section. It can be proved using the same proof in \cite[Proposition 3.16]{miller2021top}.

\begin{prop}[{\cite[Proposition 3.16]{miller2021top}}]
\label{Gammaquotient}
    Suppose $R$ is  a PID, and $p \in R$ a prime. Let $\F$ denote the field obtained by quotienting $R$ by the ideal generated by $p$, and let $\U$ denote the subgroup of $F^\times$ formed by the image of $R^\times$ under the map $R \surj \F$. Let $\Gamma_n(p)$ denote the principal level-$p$ congruence subgroup of $\SL_n(R)$. Then for $n \geq 1$, we have $\cT_n(R)/\Gamma_n(p) \cong \cT_n^\U(\F)$.
\end{prop}

\subsection{The proof of Theorems \ref{thm:kmpwanswer} and \ref{thm:generallsiso}}
\label{sec:generalisoproof}

\subsubsection{Relating augmented partial bases to the Steinberg module}
\label{sec:thm1.4&1.5proof}

\begin{definition}
    Let $R$ be a commutative ring, and $\U$ a multiplicative subgroup of $R^\times$ such that $-1 \in \U$. Define $\BD_n^\U(R)'$ to be the subcomplex of $\BD_n^\U(R)$ consisting of simplices $\{[\vec{v}_0], \dots, [\vec{v}_k]\}$ such that the $R$-span of the $\vec{v}_i$ is a proper submodule of $R^n$, and $\BDA_n^\U(R)'$ to be the subcomplex of $\BDA_n^\U(R)$ consisting of simplices whose $R$-span of the $\vec{v}_i$ is a proper submodule of $R^n$. 
\end{definition}

In \cite{church2017codimension}, Church--Putman gave an alternate proof of a theorem of Ash--Rudolph \cite{ash1979modular} that gives a generating set for the Steinberg module over the field of fractions of a Euclidean domain. In the course of the proof, they established the following result.

\begin{lemma}[{\cite{church2017codimension}}]
\label{CPBykovskii}
    Suppose $R$ is a Euclidean domain and $K$ is its field of fractions. Let $\U = R^\times$. Then for all $n\geq 1$, the map
    \begin{align*}
        \C_{n-1}(\B_n^\U(R)) \to \St_n(K)
    \end{align*}
    that maps a basis element corresponding to an $(n-1)$-dimensional simplex of $\B_n^\U(R)$ to the apartment determined by its vertices, is a surjection.
\end{lemma}

Using arguments similar to Church--Putman \cite{church2017codimension}, we can prove the analogue of Lemma \ref{CPBykovskii} for the complexes of determinant-$\U$ partial bases:

\begin{lemma}
    \label{lem:TUSurj}
    Let $\F$ be a field, and let $\U$ be a multiplicative subgroup of $\F^\times$ such that $-1 \in \U$.
    For $n \geq 2$, we have surjections
    \begin{align*}
        \h_{n-1}(\BD_n^\U(\F), \BD_n^\U(\F)')  \overset{\del}{\twoheadrightarrow} \widetilde{\h}_{n-2}(\BD_n^\U(\F)') \overset{\Phi_*}{\twoheadrightarrow} \widetilde{\h}_{n-2}(\cT_n^\U(\F))
    \end{align*}
    where $\del$ and $\Phi$ are as follows:
    \begin{itemize}
        \item $\del$ is the boundary map in the long exact sequence of a pair in reduced homology.
        \item $\Phi: \mathcal{P}(\BD_n^\U(\F)') \to \T^\U_n(\F)$ is the poset map taking a simplex $\{[\vec{v}_0], \dots, [\vec{v}_k]\}$ of $\BD_n^\U(\F)'$ to the $\F$-span of the $\vec{v}_i$.
    \end{itemize}

\end{lemma}

\begin{proof}
The long exact sequence in homology for the pair $(\BD_n^\U(\F), \BD_n^\U(\F)')$ includes the segment
    \begin{align*}
        \h_{n-1}(\BD_n^\U(\F), \BD_n^\U(\F)') \xrightarrow{\del} \widetilde{\h}_{n-2}(\BD_n^\U(\F)') \to \widetilde{\h}_{n-2}(\BD_n^\U(\F))
    \end{align*}
    and so the fact that $\del$ is a surjection follows immediately from the fact that $\BD_n^\U(\F)$ is $(n-2)$-connected (Proposition \ref{BDconn}).

    To prove that $\Phi_*$ is an isomorphism, we will proceed by studying the map-of-posets spectral sequence (Theorem \ref{posetss}) of the poset map $\Phi: \mathcal{P}(\BD_n^\U(\F)') \to \T_n^\U(\F)$. This takes the form
\begin{align*}
    \E_{kh}^2 = \h_k(\T_n^\U(\F); [V \mapsto \h_h(\Phi_{\leq V})]) \implies \h_{k+h}(\mathcal{P}(\BD_n^\U(\F)')
\end{align*}
We will use Lemma \ref{posetheightm} to compute the groups on the $\E^2$ page. This will require the following two facts. Consider $V \in \T_n^\U(\F)$.
\begin{itemize}
    \item We have
    \begin{align*}
        \Phi_{\leq V} \cong \BD_{\dim(V)}^\U(\F) = \BD_{\het(V)+1}^\U(\F).
    \end{align*}
    Our assumption says that this is $(\het(V)-1)$-connected. Since it has dimension equal to $\het(V)$, we conclude that
    \begin{align}
    \label{surjhsupport1}
        \widetilde{\h}_h(|\Phi_{\leq V}|) = 0 \hspace{0.5cm} \text{ for all } h \neq \het(V)
    \end{align}

  The upshot is that for $h \geq 1$, the functor
\begin{align*}
    V \mapsto \h_h(\Phi_{\leq V})
\end{align*}
is supported on elements of height $h$.
    \item We have
    \begin{align*}
        (\T_n^\U(\F))_{>V} \cong \T_{n-\dim(V)}^\U(\F) = \T_{n-1-\het(V)}^\U(\F).
    \end{align*}
    Lemma \ref{solomontitsU} says that this is Cohen-Macaulay of dimension $(n-3-\het(V))$, which implies that
    \begin{align}
    \label{surjhsupport2}
        \widetilde{\h}_k(|(\T_n^\U(\F))_{>V}|) = 0 \hspace{0.5cm} \text{ for all } k \neq n-3-\het(V).
    \end{align}
    \end{itemize}
We first analyze the bottom row $\E_{k0}^2$ of the $\E^2$ page. Proposition \ref{BDconn} implies that $\BD_{\dim(V)}^\U(\F)$ is connected when $\dim(V) \geq 2$. When $\dim(V)=1$, $\BD_1^\U(\F)$ is a single point and thus also connected. (This is a key place where the determinant restriction is important).
We see that
\begin{align*}
    \E_{k0}^2 = \h_k(\T_n^\U(\F); [V \mapsto \h_0(\Phi_{\leq V})]) = \h_k(\T_n^\U(\F); \underline{\Z}) = \begin{cases}
        {\h}_{n-2}(\cT_n^\U(\F)) & \text{if } k= n-2 \\
        \Z & \text{if } k=0 \text{ and } n \neq 2 \\
        0 & \text{if } k \neq 0, n-2
    \end{cases}
\end{align*}
The last equality uses Lemma \ref{solomontitsU}.

We now analyze the rows for which $h \geq 1$.

Fact \ref{surjhsupport1} implies that for $h \geq 1$, we have
\begin{align*}
     \E_{kh}^2 = \h_k(\T_n^\U(\F); [V \mapsto \h_h(\Phi_{\leq V})]) & = \bigoplus_{\het(V) = h} \widetilde{\h}_{k-1}(|(\T_n^\U(\F)_{>V})|; \h_h(\Phi_{\leq V}))
\end{align*}
Fact \ref{surjhsupport2} now implies that if $k-1 \neq n-3-\het(V) = n-3-h$, or equivalently if $k+h \neq n-2$, then this last group is 0.
Thus we have
\begin{align}
    \label{E2bigrows}
    \E_{kh}^2 = 0 \hspace{0.25cm} \text{ for } h\geq 1, k+h \neq n-2
\end{align}

If $n=2$, then the only nonzero group on the $\E^2$-page is $\E^2_{0,0} = \h_{n-2}(\cT_n^\U(\F))$, and so we get 
\begin{align*}
    \h_{n-2}(\BD_n^\U(\F)) \cong \h_{n-2}(\cT_n^\U(\F)), \text{ } n=2
\end{align*} 
One can check that the edge map of our spectral sequence that induces this isomorphism is $\Phi_*$.

If $n \geq 3$, then the only nonzero groups on the $\E^2$-page are at $\E^2_{0,0}$, and along the $n-2$ diagonal. Note that there can be no nonzero differentials between these groups -- indeed, the only instance where there can be a non-zero differential between the $k+h=n-2$ and $k+h=0$ diagonals is if $n=3$, but even then any nontrivial differential cannot survive past the $\E^1$-page.

Thus, in this case we get a surjection induced by the edge map in our spectral sequence
\begin{align*}
    \Phi_* : \h_{n-2}(\BD_n^\U(\F)) \surj \E^2_{n-2,0} = \h_{n-2}(\cT_n^\U(\F))
\end{align*}
      as desired.
\end{proof}

As a corollary we get the following:

\begin{corollary}
    \label{cor:CPBykAnalogue}
    Suppose $\F$ is a field and $\U \subset \F^\times$ a multiplicative subgroup with $-1 \in \U$. Then for all $n\geq 1$, the map
    \begin{align*}
        \C_{n-1}(\BD_n^\U(\F)) \to \widetilde{\h}_{n-2}(\cT_n^\U(\F))
    \end{align*}
    that maps a basis element corresponding to an $(n-1)$-dimensional simplex of $\B_n^\U(R)$ to the apartment determined by its vertices, is a surjection.
\end{corollary}

\begin{proof}
    We have surjections
    \begin{align*}
        \C_{n-1}(\BD_n^\U(\F)) \surj \C_{n-1}(\BD_n^\U(\F), \BD_n^\U(\F)') \surj \h_{n-1}(\BD_n^\U(\F), \BD_n^\U(\F)')
    \end{align*}
    where the latter map is a surjection since $\BD_n^\U(\F)$ is $(n-1)$-dimensional. Thus the result follows from Lemma \ref{lem:TUSurj}.
\end{proof}

Using a similar argument to the one in Lemma \ref{lem:TUSurj}, we can use the high connectivity of $\BDA_n^\U(\F)$ to obtain a presentation for $\widetilde{\h}_{n-2}(\cT_n^\U(\F))$.

\begin{lemma}
    \label{TUBykovskii}
    Let $\F$ be a field, and let $\U$ be a multiplicative subgroup of $\F^\times$ with $-1 \in \U$. Suppose $\BDA_n^\U(\F)$ is $(n-1)$-connected for all $n \geq 1$. Then for $n \geq 2$, we have isomorphisms
    \begin{align*}
        \h_{n-1}(\BDA_n^\U(\F), \BDA_n^\U(\F)') \xrightarrow[\cong]{\del} \widetilde{\h}_{n-2}(\BDA_n^\U(\F)') \xrightarrow[\cong]{\Psi_*} \widetilde{\h}_{n-2}(\cT^\U_n(\F)),
    \end{align*}
    where $\del$ and $\Psi$ are as follows:
    \begin{itemize}
        \item $\del$ is the boundary map in the long exact sequence of a pair in reduced homology.
        \item $\Psi: \mathcal{P}(\BDA_n^\U(\F)') \to \T^\U_n(\F)$ is the poset map taking a simplex $\{[\vec{v}_0], \dots, [\vec{v}_k]\}$ of $\BDA_n^\U(\F)'$ to the $\F$-span of the $\vec{v}_i$.
    \end{itemize}
\end{lemma}
\begin{proof}
    The long exact sequence in homology for the pair $(\BDA_n^\U(\F), \BDA_n^\U(\F)')$ includes the segment
    \begin{align*}
        \h_{n-1}(\BDA_n^\U(\F)) \to \h_{n-1}(\BDA_n^\U(\F), \BDA_n^\U(\F)') \xrightarrow{\del} \widetilde{\h}_{n-2}(\BDA_n^\U(\F)') \to \widetilde{\h}_{n-2}(\BDA_n^\U(\F))
    \end{align*}
    and so the fact that $\del$ is an isomorphism follows immediately from the fact that $\BDA_n^\U(\F)$ is $(n-1)$-connected.

    To prove that $\Psi_*$ is an isomorphism, we will proceed by studying the map-of-posets spectral sequence (Theorem \ref{posetss}) of the poset map $\Psi: \mathcal{P}(\BDA_n^\U(\F)') \to \T_n^\U(\F)$. This takes the form
\begin{align*}
    \E_{kh}^2 = \h_k(\T_n^\U(\F); [V \mapsto \h_h(\Psi_{\leq V})]) \implies \h_{k+h}(\mathcal{P}(\BDA_n^\U(\F)')
\end{align*}
We will use Lemma \ref{posetheightm} to compute groups on the $\E^2$ page. This will require the following two facts. Consider $V \in \T_n^\U(\F)$.
\begin{itemize}
    \item We have
    \begin{align*}
        \Psi_{\leq V} \cong \BDA_{\dim(V)}^\U(\F) = \BDA_{\het(V)+1}^\U(\F).
    \end{align*}
    Our assumption says that this is $\het(V)$-connected. Since it has dimension atmost $(\het(V)+1)$, we conclude that
    \begin{align}
    \label{hsupport1}
        \widetilde{\h}_h(|\Psi_{\leq V}|) = 0 \hspace{0.5cm} \text{ for all } h \neq \het(V)+1
    \end{align}
Note that the dimension is exactly $(\het(V)+1)$ except in the degenerate case of $\het(V)=0$. In this case, $\BDA_1^\U(\F)$ is a single point.
The upshot is that for $h \geq 1$, the functor
\begin{align*}
    V \mapsto \h_h(\Psi_{\leq V})
\end{align*}
is supported on elements of height $h-1$.
    \item We have
    \begin{align*}
        (\T_n^\U(\F))_{>V} \cong \T_{n-\dim(V)}^\U(\F) = \T_{n-1-\het(V)}^\U(\F).
    \end{align*}
    Lemma \ref{solomontitsU} says that this is Cohen-Macaulay of dimension $(n-3-\het(V))$, which implies that
    \begin{align}
    \label{hsupport2}
        \widetilde{\h}_k(|(\T_n^\U(\F))_{>V}|) = 0 \hspace{0.5cm} \text{ for all } k \neq n-3-\het(V).
    \end{align}
    \end{itemize}
We first analyze the bottom row $\E_{k0}^2$ of the $\E^2$ page. Proposition \ref{BDAconnectivity} implies that $\BDA_{\dim(V)}^\U(\F)$ is connected when $\dim(V) \geq 2$. When $\dim(V)=1$, it is a single point and thus is also connected.
We see that
\begin{align*}
    \E_{k0}^2 = \h_k(\T_n^\U(\F); [V \mapsto \h_0(\Psi_{\leq V})]) = \h_k(\T_n^\U(\F); \underline{\Z}) = \begin{cases}
        {\h}_{n-2}(\cT_n^\U(\F)) & \text{if } k= n-2 \\
        \Z & \text{if } k=0 \text{ and } n \neq 2 \\
        0 & \text{if } k \neq 0, n-2
    \end{cases}
\end{align*}
The last equality uses Lemma \ref{solomontitsU}.

We now analyze the rows for which $h \geq 1$.

Fact \ref{hsupport1} implies that for $h \geq 1$, we have
\begin{align*}
     \E_{kh}^2 = \h_k(\T_n^\U(\F); [V \mapsto \h_h(\Psi_{\leq V})]) & = \bigoplus_{\het(V) = h-1} \widetilde{\h}_{k-1}(|(\T_n^\U(\F)_{>V})|; \h_h(\Psi_{\leq V}))
\end{align*}
Fact \ref{hsupport2} now implies that if $k-1 \neq n-3-\het(V) = n-3-h+1$, or equivalently if $k+h \neq n-1$, then this last group is 0.
Thus we have
\begin{align}
    \label{E2bigrows}
    \E_{kh}^2 = 0 \hspace{0.25cm} \text{ for } h\geq 1, k+h \neq n-1
\end{align}

    The above two calculations together show that the only nonzero entry in the $(n-2)$-diagonal is
    \begin{align*}
        \E_{n-2,0}^2 = \h_{n-2}(\cT_n^\U(\F))
    \end{align*}
    Note that there are no nontrivial differentials going into $\E_{n-2,0}^2$. Furthermore, we also see that the $(n-3)$-diagonal either contains no nonzero entries (if $n \neq 3$), or consists of the single entry $\E_{0,0}^2 = \Z$ (if $n=3$). In either case, there are no nontrivial differentials going out of $\E_{n-2,0}^2$.
    Thus this term survives until $\E^\infty$. Since there are no other nonzero entries in the $(n-2)$-diagonal, and the edge value in our spectral sequence is the image of the map
    \begin{align*}
        \Psi_* : \h_{n-2}(\BDA_n^\U(\F)') \to \h_{n-2}(\cT_n^\U(\F)),
    \end{align*}
we conclude that this map is an isomorphism, as desired.
\end{proof}

As an immediate corollary, we get the following:
\begin{corollary}
\label{TUCoker}
Let $\F$ be a field, and let $\U$ be a multiplicative subgroup of $\F^\times$ with $-1 \in \U$. Suppose $\BDA_n^\U(\F)$ is $(n-1)$-connected for all $n \geq 1$. Then we have an isomorphism
\begin{align*}
    \widetilde{\h}_{n-2}(\cT_n^\U(\F)) \cong \coker(\C_{n}(\BDA_n^\U(\F), \BDA_n^\U(\F)') \to \C_{n-1}(\BDA_n^\U(\F), \BDA_n^\U(\F)'))
\end{align*}
\end{corollary}
\begin{proof}
    Since $\BDA_n^\U(\F)'$ contains the $(n-2)$-skeleton of $\BDA_n^\U(\F)$, we have 
    \begin{align*}
        \h_{n-1}(\BDA_n^\U(\F), \BDA_n^\U(\F)') = \coker(\C_{n}(\BDA_n^\U(\F), \BDA_n^\U(\F)') \to \C_{n-1}(\BDA_n^\U(\F), \BDA_n^\U(\F)')),
    \end{align*} and so this follows immediately from Lemma \ref{TUBykovskii}.
\end{proof}

We now have the following key lemma. 
\begin{lemma}
\label{BRingPrime}
  Let $R$ be a Euclidean domain, let $p \in R$ be a prime, and let $\F$ be the quotient field $R/(p)$. Let $\mc{U} = R^\times$ and let $\U$ be the image of $\mc{U}$ under the map $R \to \F$. Then for all $n \geq 1$, we have isomorphisms 
  \begin{align*}
      \C_{n-1}(\BDA_n^\U(\F), \BDA_n^\U(\F)') \cong \C_{n-1}(\BD_n^\U(\F)) \cong (\C_{n-1}(\B_n^\mc{U}(R)))_{\Gamma_n(p)}
  \end{align*}
\end{lemma}
\begin{proof}
   The chain group $\C_{n-1}(\BDA_n^\U(\F), \BDA_n^\U(\F)')$ is the free abelian group generated by all the $(n-1)$-dimensional simplices of $\BDA_n^\U(\F)$ whose vertices span $\F^n$, and these are the same as the $(n-1)$-dimensional simplices of $\BD_n^\U(\F)$, so the first isomorphism follows.

For the second isomorphism, note that there is a map 
\begin{align*}
    \C_{n-1}(\B_n^\mc{U}(R)) = \C_{n-1}(\BD_n^\mc{U}(R)) \to \C_{n-1}(\BD_n^\U(\F))
\end{align*} that sends a set of $\mc{U}$-vectors $\{[\vec{v}_1], \dots, [\vec{v}_n]\}$ in $R^n$ to the $\U$-vectors in $\F^n$ determined by the mod $p$ reductions of the $\vec{v}_i$, and that this map is equivariant with respect to the $\Gamma_n(p)$-action on $\C_{n-1}(\B_n^\mc{U}(R))$. Now, note that the chain group $\C_{n-1}(\BD_n^\U(\F))$ is freely generated by collections of $\U$-vectors $\{[\vec{v}_1], \dots, [\vec{v}_n]\}$ in $\F^n$ such that we can pick the representatives $\vec{v}_i$ so that they are the columns a matrix in $\SL_n(\F)$. We only need matrices of determinant 1 (rather than having determinant in $\U$) since we are free to multiply the $\vec{v}_i$ with elements of $\U$ as needed. Similarly, $\C_{n-1}(\BD_n^\mc{U}(R))$ is freely generated by sets of $\mc{U}$-vectors $\{[\vec{v}_1], \dots, [\vec{v}_n]\}$ in $R^n$ such that the $\vec{v}_i$ are the columns of a matrix in $\SL_n(R)$. So surjectivity follows from the classical fact that the map
    \begin{align*}
        \SL_n(R) \longrightarrow \SL_n(\F)
    \end{align*}
    that reduces matrices modulo $p$ is surjective. For injectivity, suppose that there are two sets of $\mc{U}$-vectors $\{ [\vec{V}_1],\dots, [\vec{V}_n]\}$ and $\{ [\vec{W}_1],\dots, [\vec{W}_n]\}$ of $R^n$ that both map to the same set of $\U$-vectors $\{[\vec{v}_1], \dots, [\vec{v}_n]\}$ of $\F^n$. We can choose these vector representatives such that
    \begin{align*}
        & \det(\vec{V}_1, \dots, \vec{V}_n) = 1 = \det(\vec{W}_1, \dots, \vec{W}_n) = 1\\
        & \det(\vec{v}_1, \dots, \vec{v}_n) = 1\\
        &\vec{V}_i \mapsto \vec{v}_i \text{ under the map }R^n\to \F^n\text{ for } 1\leq i \leq n
    \end{align*}
    Let $\vec{w}_i$ be the mod $p$ reductions of the $\vec{W}_i$. Thus
    \begin{align*}
        \det(\vec{w}_1, \dots, \vec{w}_n) = 1
    \end{align*}
    Since $[\vec{v}_i] = [\vec{w}_i]$ for all $i$, we have
    \begin{align*}
        \vec{v}_i = u_i\vec{w}_i
    \end{align*}
    for some $u_1, \dots, u_n \in \U$. The determinant condition implies that
    \begin{align*}
        u_1u_2\dots u_n = 1
    \end{align*}
    Choose units $u_1', \dots, u_n' \in \mc{U}$ that map to $u_1, \dots, u_n$ respectively under the map $R \to \F$. 
    Then note that the matrices
    \begin{align*}
        (\vec{V_1}, \dots ,\vec{V}_n) \text{ and } (u_1'\vec{W}_1, \dots, u_n'\vec{W}_n )
    \end{align*}
    map to the same matrix, namely
    \begin{align*}
       \left(\vec{v_1}, \dots ,\vec{v}_n\right) 
    \end{align*}
    under the mod $p$ reduction map. Thus there exists a matrix $A \in \Gamma_n(p)$ that takes one matrix to the other. Since $[\vec{W}_i] = [u_i'\vec{W}_i]$ for all $i$, we deduce that the sets $\{[\vec{V}_1],\dots, [\vec{V}_n]\}$ and $\{[\vec{W}_1],\dots, [\vec{W}_n]\}$ are in the same coset of $(\C_{n-1}(\B_n^\mc{U}(R)))_{\Gamma_n(p)}$.
\end{proof}

\begin{theorem}
\label{St_U_Iso}
    Let $R$ be a Euclidean domain with field of fractions $K$. Let $p \in R$ be a prime and let $\F$ be the quotient field $R/(p)$. Let $\mc{U} = R^\times$ and let $\U$ be the image of $\mc{U}$ under the map $R \to \F$.
    Consider the map
    \begin{align}
    \label{eqn:lsisogen}
        (\St_n(K))_{\Gamma_n(p)} \to \widetilde{\h}_{n-2}(\cT_n^\U(\F))
    \end{align}
    induced by the natural map $\T_n(K) \cong \T_n(R) \to \T_n^\U(\F)$ that sends a summand $V \subset R^n$ to its mod $p$ reduction equipped with the induced $\U$-orientation on it. This map is always surjective for $n \geq 2$.
     If in addition $\BDA_n^\U(\F)$ is $(n-1)$-connected for all $n \geq 2$, then the map \eqref{eqn:lsisogen} is an isomorphism.
\end{theorem}
\begin{proof}
    Lemma \ref{CPBykovskii} states that $ \C_{n-1}(\B_n^\mc{U}(R)) \to \St_n(K)$ is surjective, and this implies that the induced map
    \begin{align*}
         (\C_{n-1}(\B_n^\mc{U}(\R)))_{\Gamma_n(p)} \to (\St_n(K))_{\Gamma_n(p)}
    \end{align*}
    is surjective. 

Corollary \ref{cor:CPBykAnalogue} tells us that $\C_{n-1}(\BD_n^\U(\F)) \to \widetilde{\h}_{n-2}(\cT_n^\U(\F))$ is surjective.

It is not hard to check that the composition 
\begin{align*}
    (\C_{n-1}(\B_n^\mc{U}(\R)))_{\Gamma_n(p)} \surj (\St_n(K))_{\Gamma_n(p)} \to \widetilde{\h}_{n-2}(\cT_n^\U(\F))
\end{align*}

is the same as the composition

\begin{align*}
    (\C_{n-1}(\B_n^\mc{U}(\R)))_{\Gamma_n(p)} \cong \C_{n-1}(\BD_n^\U(\F)) \surj \widetilde{\h}_{n-2}(\cT_n^\U(\F)),
\end{align*}

where the isomorphism in the above expression comes from Lemma \ref{BRingPrime}. This implies surjectivity of $(\St_n(K))_{\Gamma_n(p)} \to \widetilde{\h}_{n-2}(\cT_n^\U(\F))$ for all $n \geq 2$.

Now, assuming that $\BDA_n^\U(\F)$ is $(n-1)$-connected for all $n \geq 1$, we turn to showing that this map is an isomorphism.
    
    By tracing the descriptions of the maps $\del$ and $\Psi$ from Lemma \ref{TUBykovskii}, we see that the composition
    \begin{align}\label{comp1}
    \C_{n-1}(\BDA_n^\U(\F), \BDA_n^\U(\F)') 
    \cong (\C_{n-1}(\B_n^\mc{U}(\R)))_{\Gamma_n(p)} \to (\St_n(K))_{\Gamma_n(p)} \to \widetilde{\h}_{n-2}(\cT_n^\U(\F)) 
    \end{align}
    is the same as the composition
    \begin{IEEEeqnarray}{l}
        \C_{n-1}(\BDA_n^\U(\F), \BDA_n^\U(\F)')  \to \coker(\C_n(\BDA_n^\U(\F), \BDA_n^\U(\F)') \to \C_{n-1}(\BDA_n^\U(\F), \BDA_n^\U(\F)')) \nonumber \\
        \hspace{10cm} \cong \widetilde{\h}_{n-2}(\cT_n^\U(\F)),
        \label{comp2}
    \end{IEEEeqnarray}
    where the isomorphism in \eqref{comp2} is due to Corollary \ref{TUCoker}. This (once again) implies surjectivity of $(\St_n(K))_{\Gamma_n(p)} \to \widetilde{\h}_{n-2}(\cT_n^\U(\F))$. 
    
    We need to show that this map is also injective. Note that the equality of the compositions \eqref{comp1} and \eqref{comp2} implies that showing injectivity is equivalent to showing that anything in the kernel of the map
    \begin{align*}
        \C_{n-1}(\BDA_n^\U(\F), \BDA_n^\U(\F)') \to \widetilde{\h}_{n-2}(\cT_n^\U(\F))
    \end{align*}
    is also in the kernel of the map
    \begin{align*}
        \C_{n-1}(\BDA_n^\U(\F), \BDA_n^\U(\F)') \to (\St_n(K))_{\Gamma_n(p)}
    \end{align*}
    By Corollary \ref{TUCoker}, we know that the kernel of 
    \begin{align*}
        \C_{n-1}(\BDA_n^\U(\F), \BDA_n^\U(\F)') \to \widetilde{\h}_{n-2}(\cT_n^\U(\F))
    \end{align*}
    is equal to 
    \begin{align*}
        \im(\C_n(\BDA_n^\U(\F), \BDA_n^\U(\F)') \to \C_{n-1}(\BDA_n^\U(\F), \BDA_n^\U(\F)'))
    \end{align*}
    This image is generated by sums of the form
    \begin{align}\label{FBykSum}
        [[\vec{v}_1], \dots, [\vec{v}_n]] - [[\vec{v}_0], [\vec{v}_2], \dots, [\vec{v}_n]] + [[\vec{v}_0], [\vec{v}_1], [\vec{v}_3], \dots, [\vec{v}_n]]
    \end{align}
    where $[\vec{v}_1], \dots, [\vec{v}_n]$ form a determinant-$\U$ basis of $\F^n$ and $\vec{v}_0 = \lambda\vec{v}_1 + \nu \vec{v}_2$ for some $\lambda, \nu \in \U$. We can pick the representatives so that
    \begin{align*}
        \det(\vec{v}_1, \dots, \vec{v}_n) = 1
    \end{align*}
    and so the surjectivity of
    \begin{align*}
        \SL_n(R) \to \SL_n(\F)
    \end{align*}
    implies that for $1 \leq i \leq n$, we can pick vectors $\vec{V}_i \in R^n$ that map to $\vec{v}_i$ under the mod $p$ reduction map. Let $\vec{V}_0 = \lambda'\vec{V}_1 + \nu'\vec{V}_2$, where $\lambda', \nu' \in \mc{U}$ are preimages of $\lambda, \nu \in \U$.
    The map  $\C_{n-1}(\BDA_n^\U(\F), \BDA_n^\U(\F)') \to (\St_n(K))_{\Gamma_n(p)}$ sends the sum \eqref{FBykSum} to the image in $(\St_n(K))_{\Gamma_n(p)}$ of the following sum of apartments of $\St_n(K)$:
    \begin{align}
        \label{BykSum}
        [[\vec{V}_1], \dots, [\vec{V}_n]] - [[\vec{V}_0], [\vec{V}_2], \dots, [\vec{V}_n]] + [[\vec{V}_0], [\vec{V}_1], [\vec{V}_3], \dots, [\vec{V}_n]]
    \end{align}
    The sum \eqref{BykSum} is a relation among apartments in $\St_n(K)$, i.e. it is 0 in $\St_n(K)$ -- this can easily be verified using the description of apartment classes stated in Theorem \ref{thm:soltit} -- and thus so is its image in $(\St_n(K))_{\Gamma_n(p)}$. This completes the proof.
\end{proof}
\begin{remark}
    Note that for the last step of the proof, we needed to know that the Bykovskii relation \eqref{BykSum} always holds in $\St_n(K)$, but we did not need to know that this type of sum generates \emph{all} the relations between apartments in $\St_n(K)$. The Bykovskii relation does give a full presentation for $\St_n(K)$ in certain cases, namely for $R = \Z$ (see \cite{church2017codimension}), and for $R = \Z[\io]$ or $\Z[\w]$ (see \cite{kupers2022generalized}). Since we don't assume the Bykovskii presentation here, this allows us to compute the top cohomology of $\Gamma_n(p)$ for primes in many other Euclidean rings.
\end{remark}

\begin{proof}[Proof of Theorems \ref{thm:kmpwanswer} and \ref{thm:generallsiso}]
The assertion of surjectivity in Theorem \ref{thm:generallsiso} follows from Theorem \ref{St_U_Iso} and Proposition \ref{Gammaquotient}. For the injectivity claim, note that if $\U = \F^\times$, then we have $\BDA_n^\U(\F) \cong \BA_n^\U(\F)$, which is $(n-1)$-connected by Proposition \ref{BAUconnectivity}. In the other case we have $(n-1)$-connectivity of $\BDA_n^\U(\F)$ by Proposition \ref{primeconditions} as long as $\F$ is finite, in which case Theorem \ref{thm:generallsiso} again follows immediately from Theorem \ref{St_U_Iso} and Proposition \ref{Gammaquotient}.  The finiteness of $\F$ holds because of a general fact: the quotient of any number ring by a non-zero prime ideal is finite. Indeed, say $\mc{O}_K$ is a number ring with a prime ideal $\mathfrak{p}$. Let $0 \neq \alpha \in \mathfrak{p}$. Since $\alpha$ is an algebraic integer, it satisfies a polynomial equation with integer coefficients:
\begin{align*}
    \alpha^m + b_{m-1}\alpha^{m-1} + \dots + b_1\alpha + b_0 = 0
\end{align*}
Thus we have
\begin{align*}
    b_0 = -\alpha(\alpha^{m-1} + b_{m-1}\alpha^{m-2} + \dots + b_1) \in \mathfrak{p}
\end{align*}
Thus $\mathfrak{p}$ contains an integer $b_0 \in \Z$. Since $\mathfrak{p}$ is prime, it must contain a prime factor of $b_0$, say $p$. So $\mc{O}_K/\mathfrak{p}$ is a vector space over the finite field $\F_p$. Now, since $\mc{O}_K$ is a free $\Z$-module of finite rank, it follows that $\mc{O}_K/\mathfrak{p}$ has finite dimension over $\F_p$, and thus must be finite.

    It remains to show that in the specific cases stated in Theorem \ref{thm:kmpwanswer}, that we have $|\F^\times/\U| = 1$ or $2$.
    In any algebraic number ring $R$ that is Euclidean, the number of elements in the quotient $R/(p)$ for some $p \in R$ is precisely the norm of $p$. We will make use of this in all our examples.

    \begin{itemize}
        \item ${R} = {\Z}[\io]$: The norm of an element $a+ b\io \in \Z[\io]$ is $a^2+b^2$. Thus the primes $p = 3$ and $p = 2\io+1$ have norms 9 and 5, respectively, and the quotient field $\F = R/(p)$ in either case has size 9 and 5, respectively.
        Observe that $R^\times = \Z[\io]^\times = \{\pm 1, \pm \io\}$. The image $U$ of $R^\times$ in $\F = R/(p)$ in both cases of $p = 3$ or $2\io +1$ has 4 elements. Thus when $p = 3$, we have $|\F^\times/\U| = \frac{|\F|-1}{|\U|} = \frac{8}{4} = 2$, and when $p = 2\io +1$, we have $|\F^\times/\U| = \frac{|\F|-1}{|\U|} = \frac{4}{4} = 1$.

        \item $R = \Z[\w]$: The norm of an element $a+b\w \in \Z[\w]$ is $a^2-ab+b^2$. Thus, for $p = 4\w+1$ or $4\w+3$, the quotient field $\F = R/(p)$ has size $13$. We have $R^\times = \Z[\w]^\times = \{\pm1, \pm \w, \pm \w^2\}$, and the image $U$ of $R^\times$ in $\F$ in both cases has size 6. Thus in both cases we have $|\F^\times/\U| = \frac{|\F|-1}{|\U|} = \frac{12}{6} = 2$. Similarly, in the cases when $p = 3\w+1$ or $3\w+2$, we have $|\F| = 7$, and hence $|\F^\times /\U| = 1$.

        \item $R = \Z[\sqrt{2}]$: The norm of an element $a+b\sqrt{2} \in \Z[\sqrt{2}]$ is $a^2-2b^2$. Thus $p = 5$ and $p = 3+\sqrt{2}$ have norm 25 and 7, respectively.
        
        The group of units $R^\times$ is $\{\pm 1\} \times \{(1+\sqrt{2})^k | k \in \Z\}$.
        Now, upon writing down the powers of $1+\sqrt{2}$ modulo $p=5$, we find that $(1+\sqrt{2})^3 = 7+5\sqrt{2} \equiv 2 \mod{5}$. Thus we have $(1+\sqrt{2})^6 \equiv 2^2 \equiv -1 \mod{5}$, and so $1+\sqrt{2}$ has order 12 in $\F = \Z[\sqrt{2}]/(5)$. From this and the fact that $(1+\sqrt{2})^6 = -1$ in $\F$, it follows that the image $\U$ of $R^\times = \{\pm1\}\times \{(1+\sqrt{2})^k | k \in \Z\}$ in $\F$ has order 12. Thus in this case we have $|\F^\times/\U| = \frac{|\F|-1}{|\U|} = \frac{24}{12} = 2$.

        Now we do the same thing for $p = 3+\sqrt{2}$. Set $\F= \Z[\sqrt{2}]/(3+\sqrt{2})$. Note that $(3+\sqrt{2})(3-\sqrt{2}) = 7$, so $p$ divides $7$. We have $(1+\sqrt{2})^3 = 7+5\sqrt{2} \equiv -2\sqrt{2} \mod{7}$. Thus $(1+\sqrt{2})^3 \equiv -2\sqrt{2} \mod{p}$. Since $p = 3+\sqrt{-2}$, we further have $-2\sqrt{2} \equiv 6 \mod{p}$. Finally, since $6 \equiv -1 \mod{7}$ and $p$ divides $7$, we end up with $(1+\sqrt{2})^3 \equiv -1 \mod{p}$. It follows that $(1+\sqrt{2})$ has order 6 in the quotient field $\F = \Z[\sqrt{2}]/(3+\sqrt{2})$. Since $\F$ contains exactly 7 elements, we have $|\F^\times /\U| = 1$.

        \item $R = \Z[\zeta_5]$: The Galois groups of $\Q[\zeta_5]$ is cyclic of order 4, consisting of the automorphisms $\zeta_5 \mapsto \zeta_5^k$ for $k = 1,2, 3, 4$. The norm of an element in $\Z[\zeta_5]$ is the product of its images under the action of each automorphism of the Galois group. In particular, the norm of $3 \in \Z[\zeta_5]$ is $3^4 = 81$. Thus $\F = \Z[\zeta_5]/(3)$ has 81 elements.

        Now, the group of units of $\Z[\zeta_5]$ is $\{\pm 1\} \times \{\zeta_5^k | 0 \leq k \leq 4\} \times \{ \left( \frac{1+\sqrt{5}}{2} \right)^k | k \in \Z\}$. $\zeta_5$ has order 5 in the ring $\Z[\zeta_5]$, and thus its image in $\F = \Z[\zeta_5]/(3)$ also has order 5. Upon writing the powers of $\frac{1+\sqrt{5}}{2}$ modulo 3, we see that $\left( \frac{1+\sqrt{5}}{2} \right)^4 = \frac{7+3\sqrt{5}}{2} \equiv \frac{1}{2} \equiv -1 \mod{3}$. Thus the image of $\frac{1+\sqrt{5}}{2}$ in $\F$ has order 8.

        It follows that the image $\U$ of the group of units $\Z[\zeta_5]^\times$ in $\F$ is generated by an element of order 5 and an element of order 8, and thus has order 40. Since $|\F| = 81$, it follows that $|\F^\times /\U| = 2$.
        
    \end{itemize}
\end{proof} 

\subsubsection{Computational Result}

In this section we state a recursive formula for the rank of $\widetilde{\h}_{n-2}(\cT_n^\U(\F))$ when $\F$ is a finite field. Thus in instances when the isomorphism of Theorem \ref{St_U_Iso} holds, this lets us compute the rank of $(\St_n(K))_{\Gamma_n(p)}$.

For $\F$ a field, let $\Gr_k(\F^n)$ denote the set of $k$-dimensional subspaces of $\F^n$.
We state our result as follows. It can be proved using the same proof as \cite[Theorem C]{miller2021top}. They work in the specific case when $\F = \F_p$ for some prime $p \in \Z$, and $\U = \{\pm\}$, but the same proof works more generally. Their proof is inspired by Bestvina's discrete Morse theory proof of the Solomon--Tits theorem in \cite[Proof of Theorem 5.1]{bestvina2008pl}.

\begin{prop}
    Suppose $\F$ is a finite field, and let $\U$ be a subgroup of $\F^\times$. For $n \geq 1$, let $t_n$ be the rank of $\widetilde{\h}_{n-2}(\cT_n^\U(\F))$. We then have $t_1 = 1$ and 
    \begin{align*}
        t_n = \left((|\F^\times/\U|-1) + (|\F^\times/\U|)|\F|^{n-1}\right)t_{n-1} + |\F/\U|(|\F^\times/\U|-1)\sum_{k=1}^{n-2}|\F|^k |\Gr_k(\F^{n-1})|t_kt_{n-k-1}
    \end{align*}
    for $n \geq 2$.
\end{prop}

Note that Proposition \ref{prop:Urank} now immediately follows using this result and Proposition \ref{Gammaquotient}.

\bibliographystyle{amsalpha}
\bibliography{Bibliography}
\quad \\ 
 
\end{document}